\documentclass[11pt]{article}
\usepackage{amssymb,latexsym,amsmath,amsbsy,amsthm,amsxtra,amsgen,graphicx,dsfont,color}
\usepackage{enumitem}
\usepackage{float} 
\newfont{\msbm}{msbm10 at 11pt}

\newcommand\QED{\pushQED{\qed} 
\qedhere
\popQED}

\newcommand\Quote[1]{``#1"}

\newcommand {\PP} {\mathbb{P}}

\newtheorem{Theo}{Theorem}[section]
\newtheorem{Lemma}{Lemma}[section]
\newtheorem{Cor}{Corollary}[section]
\newtheorem{Prop}{Proposition}[section]
\newtheorem{Exm}{Example}[section]
\newtheorem{Dfn}{Definition}[section]
\newtheorem{Rmk}{Remark}[section]
\newtheorem{Assum}{Assumption} 

\begin{document}

\title{Perturbing the principal Dirichlet 
eigenfunction}
\author{Brian Chao and Laurent Saloff-Coste \\ Cornell University}
\maketitle
\begin{abstract}
We study the principal Dirichlet eigenfunction $\varphi_U$ when the domain $U$ is a perturbation of a bounded inner uniform domain in a strictly local regular Dirichlet space. We prove that if $U$ is suitably contained in between two inner uniform domains, then $\varphi_U$ admits two-sided bounds in terms of the principal Dirichlet eigenfunctions of the two approximating domains. The main ingredients of our proof include domain monotonicity properties associated to Dirichlet boundary conditions, intrinsic ultracontractivity estimates, and parabolic Harnack inequality. As an application of our results, we give explicit expressions comparable to $\varphi_U$ for certain domains $U\subseteq \mathbb{R}^n$, as well as improved Dirichlet heat kernel estimates for such domains. We also prove that under a uniform exterior ball condition on $U$, a point achieving the maximum of $\varphi_U$ is separated away from the boundary, complementing a result of Rachh and Steinerberger \cite{rs}. Our principal Dirichlet eigenfunction estimates are applicable to second-order uniformly elliptic operators in Euclidean space, Riemannian manifolds with nonnegative Ricci curvature, and Lie groups of polynomial volume growth.
\end{abstract}

{\small {{\it AMS 2020 subject classifications}:  Primary 31C25, 35B51, 35J05, 60J60;
Secondary 35J25, 35K08}


}

\tableofcontents

\section{Introduction}
\label{introduction}
For a bounded domain $U\subseteq \mathbb{R}^n$, it is well-known that the Dirichlet Laplacian $-\Delta$ on $U$ has a discrete spectrum. The smallest positive eigenvalue $\lambda_U$ is given by
\begin{align}
    \label{rayleigh0}
    \lambda_U=\inf_{f \in W^{1,2}_0(U)\backslash \{0\}}\frac{\int_U |\nabla f|^2 dx}{\int_U f^2 dx},
\end{align}
and the principal Dirichlet Laplacian eigenfunction $\varphi_U$, normalized so that $\|\varphi_U\|_{L^2(U)}=1$, can be characterized as the minimizer of (\ref{rayleigh0}). Dirichlet eigenvalues and eigenfunctions have been intensively studied from the point of view of spectral theory, Markov diffusion operators, analysis on manifolds, dynamical systems, and quantum mechanics, to name a few. 

This work investigates how perturbations of a bounded inner uniform domain (Definition \ref{iudefn}) affect the principal Dirichlet Laplacian eigenfunction of that domain. We more generally work over a strictly local regular Dirichlet space $(X,d,\mu,\mathcal{E},\mathcal{D}(\mathcal{E}))$ that supports volume doubling and Poincar\'{e} inequalities up to some scale $R\in (0,\infty]$ (in the Euclidean case, $R=\infty$). In this general setting, Dirichlet Laplacian eigenfunctions are replaced with Dirichlet eigenfunctions of the infinitesimal generator associated to the Dirichlet form. We refer to Sections \ref{dirichletspace} and \ref{R-DPH} for the precise definitions. For readers unfamiliar with the theory of Dirichlet forms, it is helpful to refer to Section \ref{examples} and the Appendix for explicit examples in Euclidean space.

The aim of this paper is to provide a general framework for comparing principal Dirichlet eigenfunctions of different domains, where one domain is included in the other. Our main results, roughly speaking, come in the following form. Suppose $V$ is a bounded inner uniform domain and, for each $c\geq 1$, $V_c\supseteq V$ is a \Quote{larger version} of $V$ (Definition \ref{Vc}). (A useful example to keep in mind is a star-shaped bounded domain $V\subseteq \mathbb{R}^n$ with $ V_c=cV$ being the dilation of $V$ by $c\geq 1$.) We show that for domains $U$ with $V\subseteq U\subseteq V_c$, one has, up to constants controlled by the volume doubling and Poincar\'{e} constant of $X$, the inner uniformity constants of $V$, an upper bound on the scale of dilation $c\geq 1$, and other tractable quantities associated to the domains $U$ and $V$,
\begin{align}
    \label{mainresult}
    \varphi_V\lesssim \varphi_U\lesssim \varphi_{V_c},
\end{align}
In (\ref{mainresult}), the left inequality holds on $V$ and the right inequality holds on $U$. We write $A\lesssim B$ to denote that $A\leq CB$ where $C>0$ is some constant independent of $A$ and $B$ and controlled only by important parameters which will be specified in each theorem statement. Our main results are new even for the Laplacian, and for example imply the following theorem.

\begin{Theo}    
    \label{sample}
    Consider Euclidean space $\mathbb{R}^n$ equipped with the Laplacian and Lebesgue measure $d\mu$. Let $V\subseteq \mathbb{R}^n$ be a bounded $(C_0,c_0)$-inner uniform star-shaped domain with respect to the origin $0\in \mathbb{R}^n$. Let $\varphi_V>0$ denote the principal Dirichlet Laplacian eigenfunction of the domain $V$ normalized so that $\|\varphi_V\|_{L^2(V)}=1$.
    \begin{enumerate}
        \item  Suppose $U$ is a $(C_0,c_0)$-inner uniform domain with $V\subseteq U\subseteq cV$, where the dilation factor $c\geq 1$ is uniformly bounded above by $c_1$. Then 
    \begin{align}
        \label{sample1.1}
        & \varphi_{V}(x)\lesssim\varphi_U(x),\enspace \forall x\in V,
        \\ \label{sample1.2}& \varphi_U(x)\lesssim\varphi_{cV}(x),\enspace \forall x\in U.
    \end{align}
    The implied constants in (\ref{sample1.1}) and (\ref{sample1.2}) depend only on $C_0,c_0,c_1,$ and $n$.
        \item Now assume that for some constants $A,B>0$,
    $$\frac{\mu(\{x\in V:\textup{dist}(x,\partial V)<\delta\textup{diam}(V)\})}{\mu(V)}\leq A\delta^B,\enspace \forall \delta>0.$$
    Then there exists a constant $c=c(n,C_0,c_0)>1$ sufficiently close to $1$ such that if $U$ is any arbitrary domain with $V\subseteq U\subseteq cV$, 
    \begin{align}
        \label{sample2.1}
        & \varphi_{V}(x)\lesssim\varphi_U(x),\enspace \forall x\in V,
        \\ \label{sample2.2}& \varphi_U(x)\lesssim\varphi_{cV}(x),\enspace \forall x\in U.
    \end{align}
    The implied constant depends only on $C_0,c_0, n$ in the upper bound (\ref{sample2.1}), and depends on $C_0,c_0,n,A,B$ in the lower bound (\ref{sample2.2}).
    \end{enumerate}
\end{Theo}

For numerous variations and generalizations of Theorem \ref{sample}, we refer the reader to Theorems \ref{soue}, \ref{cylinderthm}, \ref{ub}, \ref{lowerbound}, \ref{lowerbound2}, and \ref{lowerbound3}. We summarize below the other main results of this paper.
\begin{enumerate}
    \item In Theorem \ref{separate2}, we consider a Dirichlet space $X$ satisfying a scale-invariant parabolic Harnack inequality, and show that a bounded domain $U\subseteq X$ satisfying an $\alpha$-exterior ball condition must have $d(x_U,\partial U)\gtrsim \lambda_U^{-1/2}$, where $x_U\in U$ is any point such that $\varphi_U(x_U)$ is sufficiently close to $\|\varphi_U\|_{L^{\infty}(U)}$. This complements a result of Rachh and Steinerberger \cite{rs}. See Theorem \ref{separationexample} for an illustrative result in $\mathbb{R}^n$.
    \item In Theorem \ref{separate3}, we consider uniformly elliptic operators in $\mathbb{R}^2$. Given a bounded domain $U$ satisfying an $\alpha$-exterior ball condition, we prove that the hypothesis $\varphi_U(x_U)\gtrsim \mu(U)^{-1/2}$ is sufficient for $d(x_U,\partial U)\gtrsim \lambda_U^{-1/2}$ to hold.
    \item In Theorems \ref{triangleprofile}, \ref{polygonal}, and \ref{polythm}, we give explicit expressions (i.e. a \textit{caricature function}) comparable to the principal Dirichlet Laplacian eigenfunctions of triangles, polygons, and regular polygons respectively. In Section \ref{examples}, Example \ref{triangleexample}, we propose a general method to obtain a caricature function for $\varphi_U$ when $U$ is a small polygonal perturbation of a polygonal shape, such as a planar triangle. In Example \ref{3dcube}, we discuss the difficulty of obtaining caricature functions in dimensions $n\geq 3$, then in Theorem \ref{cubethm}, we give a nontrivial example of a caricature function for a $3$-dimensional domain. Furthermore, in Section 6.1 of the Appendix, we obtain caricature functions for certain $C^{1,1}$ Euclidean domains, including convex sets in $\mathbb{R}^n$ with \Quote{rounded corners}. 
    \item In Section \ref{applylierl} of the Appendix, we discuss how the above caricature functions can be used to give essentially explicit two-sided estimates of Dirichlet heat kernels. 
\end{enumerate}

There are at least two reasons why we choose to work with inner uniform domains, which were independently introduced by Martio and Sarvas \cite{martiosarvas}, as well as Jones \cite{jones}. The first reason is the abundance of such domains. For example, the class of inner uniform domains include many Lipschitz domains in $\mathbb{R}^n$ (Page 73, \cite{jones}). Moreover, inner uniform domains also include domains with irregular boundary points or fractal boundary. The second reason is the existing works on Dirichlet heat kernel estimates on inner uniform domains, which play a crucial role in this work. Results of Grigor'yan, Saloff-Coste, and Sturm (in \cite{grigoryan}, \cite{saloff-coste}, and \cite{sturm2} respectively) established the equivalence between (i) volume doubling and Poincar\'{e} inequality on all geodesic balls, (ii) two-sided Gaussian heat kernel bounds, and (iii) the parabolic Harnack inequality. In \cite{gyryalsc}, the aforementioned result was used in conjunction with a Doob $h$-transform technique to prove two-sided Dirichlet heat kernel bounds on unbounded inner uniform domains. (Here $h$ denotes a positive harmonic function on an unbounded domain which vanishes on the boundary; we call $h$ a \textit{harmonic profile}.)
In \cite{lierllsc}, the authors considered a Doob $\varphi$-transform technique with the principal Dirichlet eigenfunction $\varphi$ playing the role of a \textit{profile function}. They then proved two-sided Dirichlet heat kernel bounds for bounded inner uniform domains in the context of possibly non-symmetric Dirichlet forms. We attempt to understand similar results for bounded non-inner uniform domains which are perturbations of inner uniform domains; we do so by studying how the profile function $\varphi_U$ behaves under domain perturbations of $U$.

Since explicit expressions for the principal Dirichlet eigenfunction $\varphi_U$ are rare even for the Dirichlet Laplacian in Euclidean space, our results contribute to understanding the shape of $\varphi_U$ for domains $U$ approximated by ones where the principal Dirichlet eigenfunction is better understood. More specifically, we aim to understand $\varphi_U$ by finding some explicit \textit{caricature function} $\Phi_U$ defined on $U$ with $\Phi_U\asymp \varphi_U$, and such that $\Phi_U$ captures the behavior of $\varphi_U$ on and away from the boundary $\partial U$. For illustration, consider the simplest case when $U=(0,a)\subseteq \mathbb{R}$ is a bounded interval and $\varphi_U(x)=(2/a)^{1/2}\sin(\pi x/a)$ is the principal Dirichlet Laplacian eigenfunction of $U$. In this case, we have $\Phi_U\leq \varphi_U\leq \pi \Phi_U/2$, where $\Phi_U(x)=2\sqrt{2}a^{-3/2}\min\{x,a-x\}$. The caricature function $\Phi_U$ captures the fact that $\phi_U(x)$ decays linearly as $x\to 0$ or $x\to a$, and that $\varphi_U\asymp a^{-1/2}$ at scale $a$ away from the boundary. In general, finding a caricature function $\Phi_U$ is highly nontrivial even if explicit expressions for $\varphi_U$ are available. For example, there are exact expressions for $\varphi_T$ when $T\subseteq \mathbb{R}^2$ is an equilateral  (see e.g. \cite{prager}) or right isosceles triangle. However, they are expressed as sums and products of trigonometric functions, which obscures the caricature of $\varphi_T$. In Theorem \ref{triangleprofile}, we consider the class of planar or spherical triangles $T$ with angles uniformly bounded below, and we propose caricature functions $\Phi_T$ that are uniformly comparable to $\varphi_T$.    

Still working with the Laplacian in Euclidean space, consider a regular polygon $P\subseteq \mathbb{R}^2$ with $n$ vertices and with each side having length $1$. We consider a domain $U\supseteq P$ whose boundary is an arbitrary Jordan curve, as pictured to the left in Figure \ref{polygon1}. We write $cP$ to denote the dilation of $P$ by a factor of $c$ with respect to the center of $P$. 
\begin{figure}[H]
  \centering
  \includegraphics[width=0.35\textwidth]{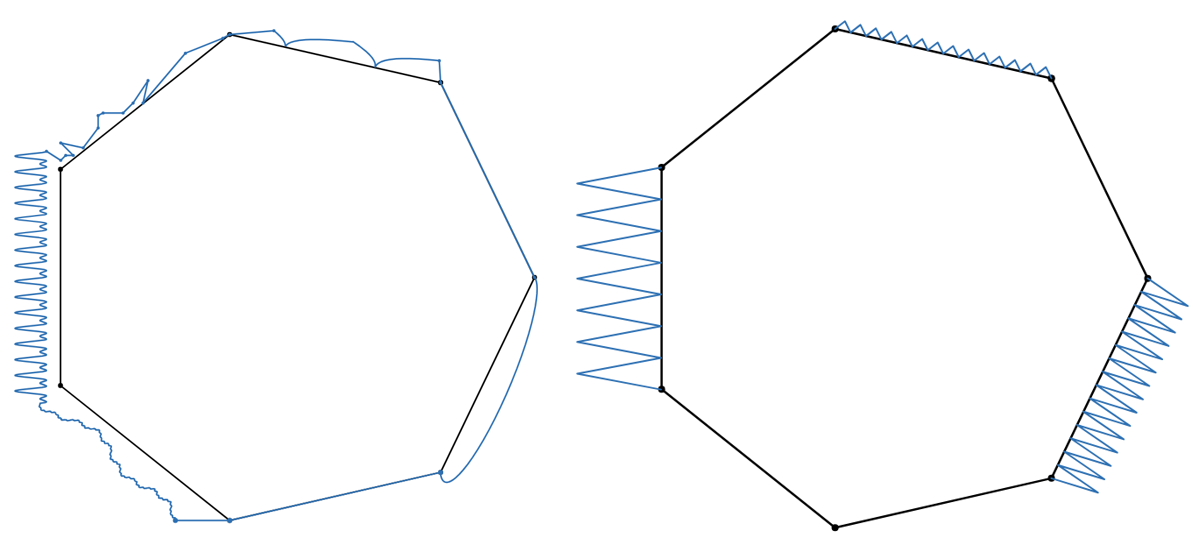}
  \caption{A regular $7$-gon $P$ is perturbed in two different ways; the perturbed domain is $U$.}
  \label{polygon1}
\end{figure}
Our results imply the stability of $\varphi_P$ under small perturbations: there is an absolute constant $c>1$ sufficiently close to $1$ such that if $U$ is contained in $c P$, then $\varphi_U \asymp \varphi_P$ on $0.99P$. A caricature function $\Phi_P\asymp \varphi_P$ is given in Theorem \ref{polythm}; therefore $\varphi_U$ is comparable to $\Phi_P$ on $0.99P$. Since we only assume $\partial U$ is a Jordan curve, there is no general description of the boundary behavior of $\varphi_U$ near $\partial U$. However, if $U$ has polygonal boundary, then we can use a boundary Harnack principle from \cite{lierllsc} to obtain boundary estimates for $\varphi_U$. Combining estimates for $\varphi_U$ near and away from the boundary, we then obtain an explicit caricature function $\Phi_U\asymp \varphi_U$. 

For a concrete example, let $N\geq 1$, and for each (or some) of the $n$ sides of the regular $n$-gon $P$, choose constants $\alpha_i>0$ and $\beta_i\in \mathbb{N}$. On the $i$th side, place $N^{\beta_i}$ identical triangles with height of order $N^{-\alpha_i}$ and perpendicular to the $i$th side. Let $U\supseteq P$ be the domain obtained by replacing the $i$th side of $P$ by the above triangles; $U$ is depicted to the right of Figure \ref{polygon1}. Our results (Theorems \ref{ub} and \ref{lowerbound}) imply that $\varphi_U\asymp \varphi_P$ on $0.99P$ with implied constants independent of $N,\alpha_i,\beta_i$. Now suppose that $\alpha_i\geq \beta_i$ so that the angles of $U$ are bounded away from zero. In a small ball centered at a vertex point of $U$ with angle $\nu$, a boundary Harnack argument implies that $\varphi_U$ is comparable to the harmonic profile $h(r,\theta)=r^{\pi/\nu }\sin(\pi \theta/\nu)$ of an infinite planar cone with opening angle $\nu$. Thus, the above estimates can be combined to give an explicit caricature function $\Phi_U$ for $\varphi_U$ when $U$ is a reasonable polygonal perturbation of $P$ such as in Figure \ref{polygon1}; the perturbations can occur at several different scales and locations. We will neither write out the caricature function $\Phi_U$ for $\varphi_U$ at this level of generality, nor will we address every admissible polygonal perturbation for which our result applies. However, we will present a similar construction in Example \ref{triangleexample} below using the same ideas, and present an illustrative result in Theorem \ref{addonetriangle2}.

 Also, understanding perturbations of principal Dirichlet eigenfunctions is understanding the behavior of diffusions with killing boundary conditions, on a class of bounded domains larger than bounded inner uniform domains. To explain why, we recall that for a bounded domain $U$, the Dirichlet heat kernel $p^D_U(t,x,y)$ can be expanded as
\begin{align*}
    p^D_U(t,x,y)=\sum_{j=1}^{\infty}e^{-\lambda_j t}\varphi_j(x)\varphi_j(y),
\end{align*}
where the $\lambda_j>0$ are the Dirichlet eigenvalues (counting multiplicity) listed in increasing order and $\varphi_j$ are the corresponding Dirichlet eigenfunctions normalized so that $\|\varphi_j\|_{L^2(U)}=1$. We can then define a new Markov kernel on $U$ by 
\begin{align}
    \nonumber \widetilde{p}_U(t,x,y):=\frac{e^{\lambda_1t}p^D_U(t,x,y)}{\varphi_1(x)\varphi_1(y)},
\end{align}
for which the associated process stays in $U$ for all times and has invariant (probability) measure $\varphi_1^2 dx$. 
This is known as the Doob transform technique and has been used extensively by many authors. It was shown in \cite{lierllsc} that if $U$ is a bounded inner uniform domain, then $\widetilde{p}_U(t,x,y)$ satisfies two-sided Gaussian heat kernel bounds, which is equivalent to volume doubling and Poincar\'{e} inequalities on $U$ with respect to the measure $\varphi_1^2 dx$. We expect bounds on $\varphi_U$ of the form (\ref{mainresult}) to be useful for extending such functional inequalities to domains approximated by inner uniform domains. 
 
We make some comments regarding the technical tools used in this paper. First, suppose $U$ is a bounded domain admitting a positive principal Dirichlet eigenfunction $\varphi_U$. Following \cite{lierllsc}, we say that the Dirichlet heat kernel $p^D_U(t,x,y)$ is \textit{intrinsically ultracontractive} if, for each $t>0$, there is a constant $A_t>0$ such that 
\begin{align}
\label{intrinsicu}
p^D_U(t,x,y)\leq A_t\varphi_U(x)\varphi_U(y),\hspace{0.2in}x,y\in U.
\end{align}
The term \Quote{intrinsic ultracontractivity} was first introduced in the seminal paper \cite{daviessimon} of Davies and Simon, and has since then been proven in a variety of contexts. 
It is known that bounded inner uniform domains are intrinsically ultracontractive (see e.g. \cite{lierllsc}).

Second, we will frequently make use of a \textit{scale-invariant parabolic Harnack inequality} (PHI, see Definition \ref{RPHI}), which establishes the comparability of pointwise values of a weak solution to the heat equation. In the case of $\mathbb{R}^n$ equipped with a second-order uniformly elliptic operator (Example \ref{soue}), PHI has been proven by Moser (\cite{moser}, \cite{moser1}), and it is actually sufficient for all of our results of the form (\ref{mainresult}) to hold, assuming we only wanted to work over Euclidean space equipped with such a differential operator. 

Third, we will often use a \textit{boundary Harnack principle} (BHP) to obtain explicit boundary estimates of $\varphi_U$ (e.g. Examples \ref{triangleexample}, \ref{polygons}, and \ref{regularpoly} in Section \ref{examples}). The classical form of the BHP states that given a domain $U$ and a sufficiently small ball $B=B(x,2r)$ with $x\in \partial U$, two harmonic functions  vanishing on $B\cap \partial U$ have comparable ratios on $B(x,r)\cap U$. The BHP for the Laplacian on Lipschitz domains has been independently proven by Ancona \cite{ancona}, Dahlberg \cite{dahlberg}, and Wu \cite{wu}. Since then the BHP has been proven in many settings, for example on uniform domains \cite{aikawa}, as well as on inner uniform domains in a Dirichlet space satisfying either PHI (\cite{gyryalsc},\cite{lierllsc},\cite{lierllsc2},\cite{lierl}) or an elliptic Harnack inequality  \cite{barlowmurugan}. We refer the reader to the aforementioned works and the references therein for more information about the BHP. We note that the BHP proven in \cite{lierllsc2} applies to Dirichlet forms that are not necessarily symmetric and with lower order terms, in particular allowing us to treat $\varphi_U$ as a \Quote{harmonic function} in the BHP.    

Fourth, an important feature of the principal Dirichlet eigenfunction $\varphi_U$ is the location of points $x_U\in U$ achieving the maximum of $\varphi_U$ (or achieving a value comparable to the maximum). This is because knowledge of the location of such points, especially the distance of such points to the boundary, will allow us to prove the lower bound $\varphi_U\gtrsim \varphi_V$ in cases where the scale of dilation $c\geq 1$ is not sufficiently close to $1$ (e.g. Theorems \ref{cylinderthm} and \ref{lowerbound2}). Brascamp and Lieb \cite{blieb} proved that the principal Dirichlet Laplacian eigenfunction $\varphi_U$ of a convex domain $U\subseteq \mathbb{R}^n$ is log-concave, which implies that $\varphi_U$ achieves a unique maximum. Sakaguchi \cite{sakaguchi} also proved the log-concavity of $\varphi_U$ for a smooth convex domain $U\subseteq \mathbb{R}^n$ and the Laplacian replaced with a quasilinear elliptic operator $\text{div}(|\nabla u|^{p-2}u),$ $p>1$. For planar convex domains, Greiser and Jerison \cite{gj} showed that $x_U$ is uniformly close to a point defined in terms of principal eigenfunctions of certain one-dimensional Schr\"{o}dinger operators. More recently, Rachh and Steinerberger \cite{rs} showed that if $U\subseteq \mathbb{R}^2$ is any bounded simply connected domain, then $\text{dist}(x_U,\partial U)\gtrsim \lambda_U^{-1/2}$ up to an universal constant. 
 
The organization of this paper is as follows. In Section \ref{examples}, we give very concrete examples for which our results apply, mostly in Euclidean space. In Section \ref{preliminaries}, we introduce some technical assumptions as well as background information about Dirichlet spaces, spectrum of the associated infinitesimal generator, and inner uniform domains. In Section \ref{proofsofmainresults}, we present the proofs of the main results about eigenfunction comparison. In Section \ref{separationsection}, we present and prove results of the form $\text{dist}(x_U,\partial U)\gtrsim \lambda_U^{-1/2}$. In Section \ref{round} of the Appendix, we find a caricature function for certain $C^{1,1}$ domains $U\subseteq \mathbb{R}^n$. In Section \ref{applylierl} of the Appendix, we discuss how our main results can be used to improve heat kernel estimates from \cite{lierllsc}.

\section{Examples}\label{examples}

\begin{Exm} \label{soue}
    \normalfont
    Let $X=\mathbb{R}^n$ and $\mu$ be the Lebesgue measure on $\mathbb{R}^n$. Consider a divergence form second-order uniformly elliptic operator 
    \begin{align}
        \label{soueo}\mathcal{L}u=\sum_{i,j=1}^{n}\frac{\partial}{\partial x_i}\Big(a_{ij}\frac{\partial}{\partial x_j}u\Big),
    \end{align}
    where $(a_{ij})$ is a symmetric matrix-valued measurable function on $\mathbb{R}^n$ satisfying uniform ellipticity with constant $\Lambda\geq 1$, i.e.
    $$\forall x,\xi\in \mathbb{R}^n,\enspace \frac{1}{\Lambda}|\xi|^2\leq \sum_{i,j=1}^{n}a_{ij}(x)\xi_i\xi_j \leq \Lambda |\xi|^2.$$
    The corresponding Dirichlet form is
    \begin{align*}
        \mathcal{E}(f,f)=\int_{\mathbb{R}^n} \sum_{i,j=1}^{n}a_{ij}\frac{\partial f}{\partial x_i}\frac{\partial f}{\partial x_j} d\mu,\hspace{0.2in} f\in W^{1,2}(\mathbb{R}^n). 
    \end{align*}
    As will be seen later in Section \ref{preliminaries}, we will equip $\mathbb{R}^n$ with a distance function (\ref{intrinsicdist}) induced by $\mathcal{E}$. This distance is comparable to the Euclidean distance with constants depending only on the ellipticity $\Lambda\geq 1$. The hypotheses of volume doubling and Poincar\'{e} inequalities in Assumption \ref{assumption2} below are satisfied with $R=\infty$. Moreover, the volume doubling and Poincar\'{e} constants depend only on $n$ and $\Lambda$. 

    As a consequence of Lemma \ref{convexvolume}, Theorem \ref{ub}, and Theorem \ref{lowerbound}, we have the following theorem, which shows the stability of $\varphi_V$ under small perturbations of a bounded convex domain $V\subseteq \mathbb{R}^n$.

    \begin{Theo}
        \label{examplethm1}
        Consider a second-order uniformly elliptic operator on $\mathbb{R}^n$ as in (\ref{soueo}), with ellipticity constant $\Lambda\geq 1$. Let $V\subseteq \mathbb{R}^n$ be a bounded convex domain, and assume $V$ has bounded eccentricity (Definition \ref{boundedecc}) with constant $K\geq 1$. For $c\geq 1$, let $cV$ denote the dilation of $V$ with respect to any fixed point in the closure of $V$. There exists a constant $c=c(n,\Lambda,K)>1$ sufficiently close to $1$ such that if $U\subseteq \mathbb{R}^n$ is any domain with $V\subseteq U\subseteq cV$, then
        \begin{align*}
             & \varphi_V(x) \lesssim \varphi_{U}(x),\hspace{0.1in} x\in V
            \\ & \varphi_U(x) \lesssim \varphi_{cV}(x),\hspace{0.1in} x\in U.
        \end{align*}
        The implied constants depend only on $n,\Lambda,$ and $K$.
    \end{Theo}

\end{Exm}
Examples show that the conclusion of Theorem \ref{polygonal} fails if $c>1$ is not sufficiently close to $1$. We do not know the optimal value of the constant $c>1$.

\begin{Exm} \normalfont
    Example \ref{soue} above includes an inequality of the form $\varphi_V\lesssim \varphi_U\lesssim \varphi_{cV}$ (whenever both sides of the inequality are defined) where $V\subseteq U\subseteq cV$, $c>1$ is sufficiently close to $1$, and $V$ is inner uniform. Under certain additional assumptions on the domains $U$ and $V$, similar results hold even when $c>1$ is merely bounded above. In this example we consider two situations where this is the case.

    One situation is when the domain $U$ itself is also an inner uniform domain. Then, as shown in Theorem \ref{ub}, we still have $\varphi_V\lesssim \varphi_U\lesssim \varphi_{cV}$, but with the implied constants also depending on an upper bound on $c\geq 1$. A nontrivial example is to consider an Euclidean ball $V\subseteq \mathbb{R}^2$ centered at the origin, a Von Koch snowflake $\mathcal{V}\mathcal{K}\subseteq \mathbb{R}^2$ that contains $V$, and $c>1$ chosen large enough so that $cV\supseteq \mathcal{V}\mathcal{K}$. Since both $V$ and $\mathcal{V}\mathcal{K}$ are inner uniform, Theorem \ref{ub} implies the inequalities $\varphi_V\lesssim \varphi_{\mathcal{V}\mathcal{K}}$ and $\varphi_{\mathcal{V}\mathcal{K}}\lesssim\varphi_{cV}$, whenever both sides of the inequality are defined.

    Another situation is when $U$ contains at least one point $x_U\in U$ with $\varphi_U(x_U)\gtrsim 1/\sqrt{\mu(U)}$ and such that $x_U$ connects to $V$ via a bounded number of balls contained inside $U$ with radii bounded below. In this case, as made precise in Theorem \ref{lowerbound2}, we still have $\varphi_V\lesssim \varphi_U\lesssim \varphi_{cV}$. Let us give an illustrative example of the lower bound $\varphi_U\gtrsim \varphi_V$ holding on $V$ for $U\supseteq V$ ranging over a family of domains which is not uniformly inner uniform. Let $V=B(0,1)\subseteq \mathbb{R}^3$ be the unit ball, and consider the domain $U\subseteq \mathbb{R}^3$ formed by adding to the unit ball three cylinders of radii $r_1,r_2,r_3\in (0,1]$ and heights $h_1,h_2,h_3\geq 1$, which can be expressed as
    \begin{align}
        \label{Udefn}
        U=B(0,1)\cup \{y^2+z^2<r_1^2,|x|<h_1\}\cup \{x^2+z^2<r_2^2,|y|<h_2\}\cup \{x^2+y^2<r_3^2,|z|<h_3\}.
    \end{align}
    See Figure \ref{cylinderfig} below. When the cylinders are not too thin, $U$ is an inner uniform domain, and thus Theorem \ref{ub} applies. When the cylinders become thinner and thinner, $U$ becomes less inner uniform, but Theorem \ref{separate2} implies that the maximum value of $\varphi_U$ is achieved inside $B(0,1)$, in which case (\ref{cor2}) of Corollary \ref{corollary1} applies. We thus get the following illustrative result.
    \begin{Theo}
        \label{cylinderthm}
         Consider a second-order uniformly elliptic operator $\mathcal{L}$ on $\mathbb{R}^3$ as in (\ref{soueo}), with ellipticity constant $\Lambda\geq 1$. Let $U\subseteq \mathbb{R}^3$ be defined as in (\ref{Udefn}) and as depicted in Figure \ref{cylinderfig}. Suppose the heights of the cylinders are uniformly bounded above: $h_1,h_2,h_3\leq h$ for some $h>1$. Let $\varphi_{U}$ (resp. $\varphi_{B(0,1)}$) be the principal Dirichlet eigenfunction of the domain $U$ (resp. $B(0,1)$) with respect to $\mathcal{L}$, normalized so that $\|\varphi_U\|_{L^2(U)}=1$ (resp. $\|\varphi_{B(0,1)}\|_{L^2(B(0,1))}=1$). Then, for any $x\in B(0,1)$,
         $$\varphi_U(x)\gtrsim \varphi_{B(0,1)}(x)$$
         where the implied constant depends only on $h$ and $\Lambda$.
    \end{Theo}
    \begin{figure}[H]
  \centering
  \includegraphics[width=0.3\textwidth]{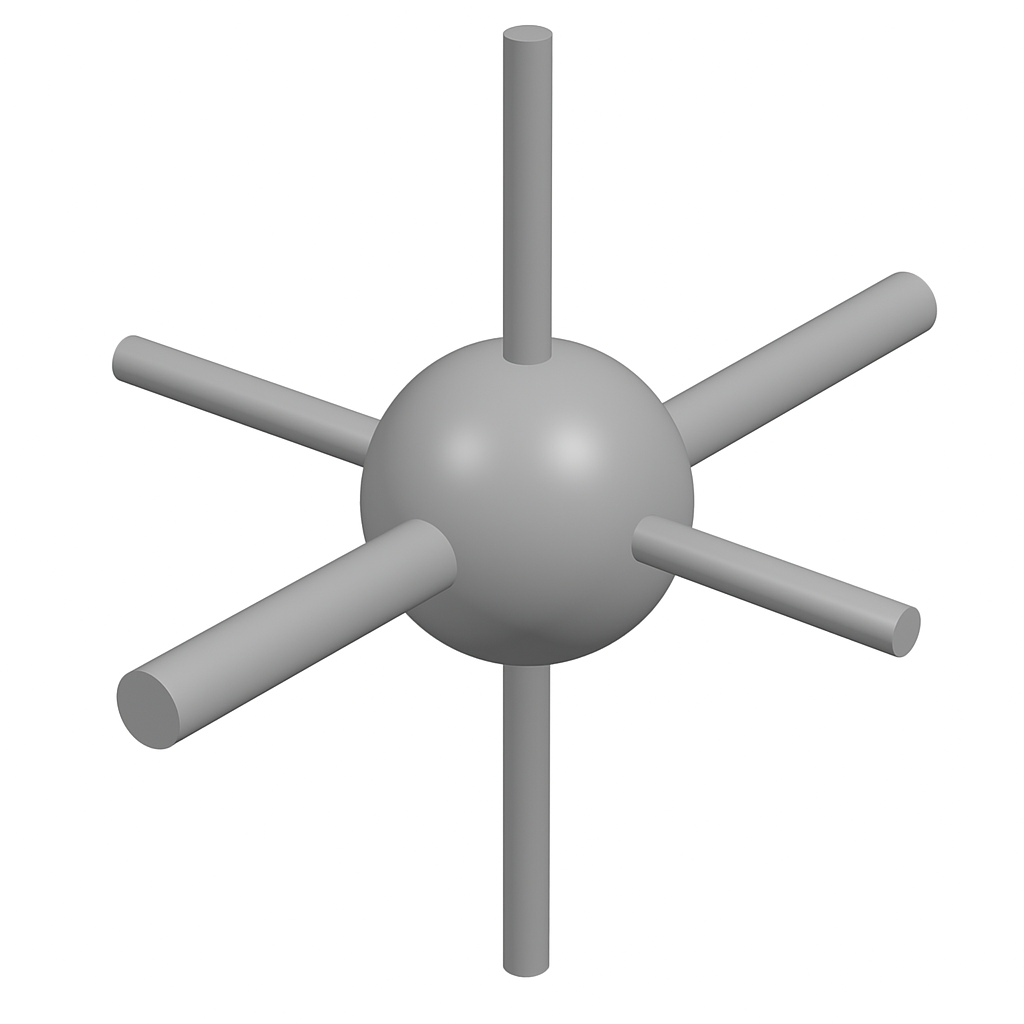}
  \caption{$U$ is obtained by adding three bounded cylinders to a ball.}
  \label{cylinderfig}
\end{figure}
    In Theorem \ref{cylinderthm}, the particular choices of the domains $U$ and $V$ are not essential, and are chosen only for concreteness. For example, a similar result holds if the unit ball of $\mathbb{R}^3$ is replaced by a bounded convex domain in $\mathbb{R}^n$ with bounded eccentricity (Definition \ref{boundedecc}), or if cylinders are replaced by twisted tubes. The only essential features of $U$ and $V$ needed for Theorem \ref{cylinderthm} to hold is (i) the inner uniformity of $V$, (ii) the family of domains $U$ to uniformly satisfy an exterior ball condition in the sense of Definition \ref{alphaball}, and uniformly be contained inside a bounded dilate of $V$, as well as (iii) the radius of the \Quote{tubes} to be uniformly small relative to $V$. In other words, Theorem \ref{cylinderthm} is not specific to three dimensions, and it is stable under reasonable perturbations.     
    
    Since any point $x_U\in U$ with $\varphi_U(x_U)=\|\varphi_U\|_{L^{\infty}(U)}$ has $\varphi_U(x_U)\geq 1/\sqrt{\mu(U)}$, it is natural to consider if such points $x_U$ must be uniformly separated from $\partial U$. By modifying the proof strategy in Rachh and Steinerberger \cite{rs}, we answer this question in the affirmative in Theorem \ref{separate2} (in a Dirichlet space and $x_U$ such that $\varphi_U(x_U)\geq (1-\varepsilon)\|\varphi_U\|_{L^{\infty}(U)}$), and in Theorem \ref{separate3} (for both $\mathbb{R}$ or $\mathbb{R}^2$, and $x_U$ such that $\varphi_U(x_U)\gtrsim 1/\sqrt{\mu(U)}$), by showing that $d(x_U,\partial U)\gtrsim 1/\sqrt{\lambda_U}$. Note that  Theorem \ref{separate3} is not covered by Theorem \ref{separate2} since the hypothesis on $x_U$ in Theorem \ref{separate3} is weaker. We provide an illustrative result below which follows from Theorem \ref{separate2}.
    
    \begin{Theo}
        \label{separationexample}
        Consider a second-order uniformly elliptic operator $\mathcal{L}$ on $\mathbb{R}^n$ as in (\ref{soueo}), with ellipticity constant $\Lambda\geq 1$. Let $U\subseteq \mathbb{R}^n$ be any bounded convex domain. Let $\{\lambda_j(U):j\geq 1\}$ be the Dirichlet eigenvalues of $U$ with respect to $\mathcal{L}$, listed in increasing order and counting multiplicity. Let $\varphi^U_j$ be the Dirichlet eigenfunction corresponding to $\lambda_j(U)$. There exists an absolute constant $c$ depending only on $n$ and $\Lambda$, such that for any $x_j\in U$ with $|\varphi^U_j(x_j)|=\|\varphi^U_j\|_{L^{\infty}(U)}$,   
        $$d(x_j,\partial U)\geq \frac{c}{\sqrt{\lambda_j(U)}}.$$
    \end{Theo}

\end{Exm}

Theorem \ref{separationexample} should be compared with the main result of \cite{rs}, which holds for all simply connected domains $U\subseteq \mathbb{R}^2$, and with $\mathcal{L}$ replaced by a Schr\"{o}dinger operator $\Delta+V$. 

\begin{Exm}
    \label{triangleexample}
    \normalfont As pictured below in Figure \ref{trianglesfig}, consider either a planar triangle in $\mathbb{R}^2$ or a spherical triangle on the $2$-sphere $\mathbb{S}^2$. (A spherical triangle is determined by three great circles on $\mathbb{S}^2$.)

    \begin{figure}[H]
  \centering
  \includegraphics[width=0.5\textwidth]{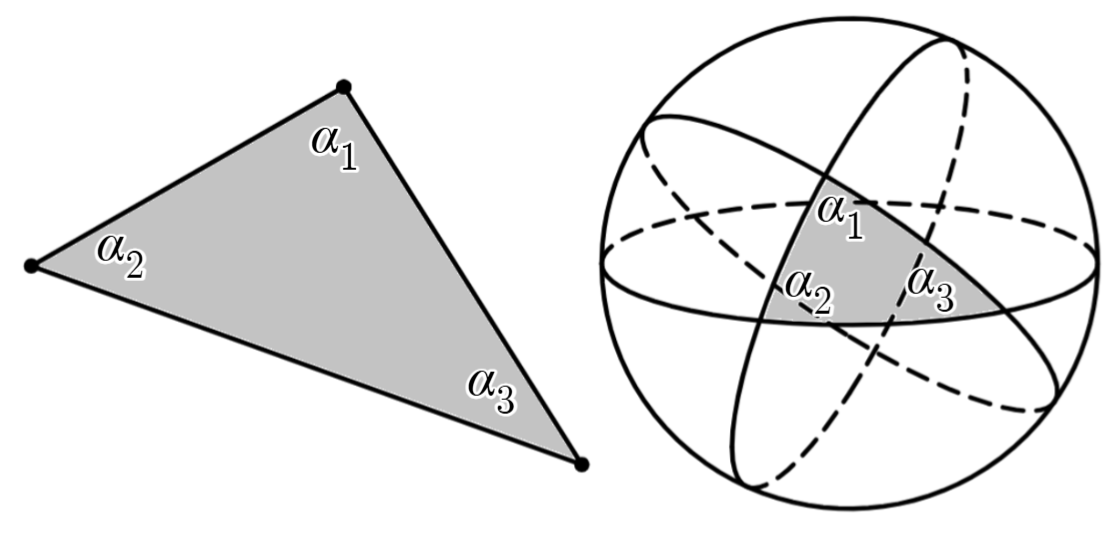}
  \caption{Triangles in $\mathbb{R}^2$ or $\mathbb{S}^2$ with angles $\alpha_1,\alpha_2,\alpha_3$.}
  \label{trianglesfig}
\end{figure}
        
    No explicit formula is known for the principal Dirichlet Laplacian eigenfunctions of such triangles except for  equilateral or right isosceles triangles in $\mathbb{R}^2$. However, for triangles with angles bounded below, we have the following result.


    \begin{Theo}
    \label{triangleprofile}
        Let $T$ be either a planar triangle or a spherical triangle with angles $\alpha_1,\alpha_2,\alpha_3$. Assume that $\min\{\alpha_1,\alpha_2,\alpha_3\}\geq \alpha>0$ for some $\alpha>0$. Let $l_i$ denote the side of $T$ opposite to $\alpha_i$. For $x\in T$ and $i=1,2,3$ let $d_i(x)=\textup{dist}(x,l_i)$, where all distances are geodesic distances on $T$. Then if $\varphi_T$ is the principal Dirichlet Laplacian eigenfunction of $T$ normalized so that $\|\varphi_T\|_{L^2(T)}=1$, we have
        \begin{align}
            \label{triangleprofile1}
            \varphi_T\asymp \frac{d_1 d_2 d_3 (d_1+d_3)^{\pi/\alpha_2-2}(d_2+d_3)^{\pi/\alpha_1-2}(d_1+d_2)^{\pi/\alpha_3-2}}{\textup{diam}(T)^{\pi/\alpha_1+\pi/\alpha_2+\pi/\alpha_3-2}}.
        \end{align}
        The implied constants depend only on $\alpha$.
    \end{Theo}
\end{Exm}

A proof of Theorem \ref{triangleprofile} is sketched in Section \ref{proofsketch}; the proof strategy is not new and it uses a boundary Harnack principle developed in \cite{lierllsc}. For the four-player Gambler's Ruin problem on a discrete simplex in four dimensions, a result similar in nature to Theorem \ref{triangleprofile} is given in \cite{ocsc}. 
For some polygonal domains $U$ which are perturbations of triangles, our results can be combined with Theorem \ref{triangleprofile} to give explicit expressions comparable to $\varphi_U$ (i.e. a \textit{caricature function} for $\varphi_U$).  We give an illustrative result in Theorem \ref{addonetriangle2} below, then outline the general method.

To begin, let $T\subseteq \mathbb{R}^2$ be a triangle with angles $\alpha_1,\alpha_2,\alpha_3$ bounded below by $\alpha>0$. Add a small equilateral triangle of side length $\varepsilon>0$ to one the sides of $T$, as in Figure \ref{trianglesperturbation2} below. 

  \begin{figure}[H]
  \centering
  \includegraphics[scale=0.25]{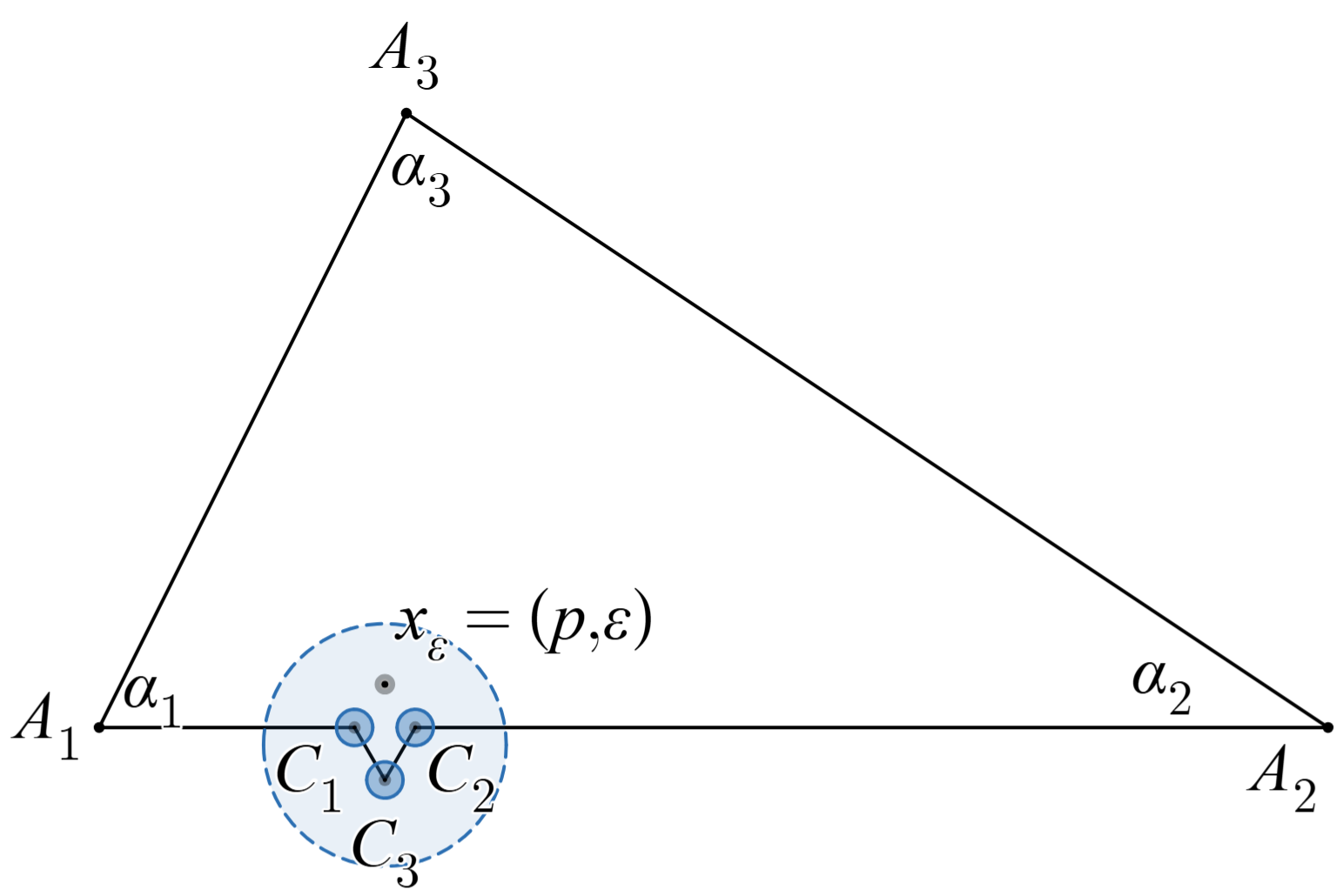}
  \caption{A triangle $T\subseteq \mathbb{R}^2$ is perturbed by adding a small equilateral triangle of side length $\varepsilon$.}
  \label{trianglesperturbation2}
\end{figure}

 In Figure \ref{trianglesperturbation2}, for concreteness, fix a coordinate system on $\mathbb{R}^2$ so that $A_1=(0,0)$ and $\overline{A_1A_2}$ lies on the positive $x$-axis. Let $l$ denote the side length of $\overline{A_1A_2}$. For $\varepsilon \in (0,l/2)$, pick $p\in (\varepsilon,l-\varepsilon)$ and add to $T$ a small equilateral triangle with side length $\varepsilon>0$, with vertices $C_1=(p-\varepsilon/2,0)$, $C_2=(p+\varepsilon/2,0)$, and $C_3$. Call this domain $U$. On the domain $U$, we define the following distances:
\begin{align*}
     d_1(x)=\textup{dist}(x,\overline{A_1C_1}),\hspace{0.1in}d_2(x)=\textup{dist}(x,\overline{C_1C_3}),\hspace{0.1in}d_3(x) &= \textup{dist}(x,\overline{C_2C_3}), \hspace{0.1in} d_4(x)=\textup{dist}(x,\overline{C_2A_2}).
 \end{align*}

For $\varepsilon>0$, we also define in coordinates a distinguished point $x_{\varepsilon}\in T$ by $x_{\varepsilon}:=(p,\varepsilon)\in \mathbb{R}^2$. This particular choice of $x_{\varepsilon}$ is chosen only for concreteness, and may be replaced with any other $x_{\varepsilon}'\in T$ that is distance $\asymp \varepsilon$ from the centroid of $\triangle C_1C_2C_3$ and distance $\gtrsim \varepsilon$ away from $\partial T$.

For $i=1,2,3$, consider a small Euclidean ball $B_i$ centered at $C_i$ with radius comparable to $\varepsilon$. An application of Theorem \ref{sample} and a boundary Harnack principle due to \cite{lierllsc}, which will be explained later in this example, implies that uniformly for all $x\in B_1$,
\begin{align}
    \label{0.01}
    \varphi_U(x)\asymp \frac{d_1(x)d_2(x)(d_1(x)+d_2(x))^{-5/4}}{(d_1(x)+\varepsilon)(d_2(x)+\varepsilon)(d_1(x)+d_2(x)+\varepsilon)^{-5/4}}\varphi_T(x_{\varepsilon}).
\end{align}
The expression for $\varphi_U$ on the small Euclidean ball $B_2$ is the same as (\ref{0.01}), but with $d_1,d_2$ replaced by $d_3,d_4$. For all $x\in B_3$, we have
\begin{align*}
    \varphi_U(x)\asymp \frac{d_2(x)d_3(x)(d_2(x)+d_3(x))}{(d_2(x)+\varepsilon)(d_3(x)+\varepsilon)(d_2+d_3+\varepsilon)}\varphi_T(x_{\varepsilon}).
\end{align*}
On the other hand, near the small triangle $\triangle C_1C_2C_3$ and away from the vertices $C_1,C_2,C_3$, $\varphi_U$ decays linearly in $\text{dist}(x,\partial U)$. Combining the different local boundary behaviors of $\varphi_U$ into a single expression, we obtain the following. 

\begin{Theo}
    \label{addonetriangle2} As in Figure \ref{trianglesperturbation2} above, consider a triangle $T=\triangle A_1A_2A_3\subseteq \mathbb{R}^2$ with angles $\alpha_1,\alpha_2,\alpha_3\geq \alpha>0$. Let $U\subseteq \mathbb{R}^2$ be the domain obtained by adding a small $\varepsilon$-triangle to $T$, as described above. Let $B$ be a ball of radius $\varepsilon>0$ centered at the centroid of $\triangle C_1C_2C_3$, and let $2B$ be the same ball with radius $2\varepsilon$. There exists $\varepsilon_{\alpha}>0$ such that for all $\varepsilon\in (0,\varepsilon_{\alpha})$,
we have
\begin{align*}
    \varphi_U \asymp\begin{cases}\displaystyle
        \frac{d_1d_2d_3d_4(d_2+d_3)(d_1+d_2+\varepsilon)^{5/4}(d_3+d_4+\varepsilon)^{5/4}}{(d_1+\varepsilon)(d_2+\varepsilon)(d_3+\varepsilon)(d_4+\varepsilon)(d_2+d_3+\varepsilon)(d_1+d_2)^{5/4}(d_3+d_4)^{5/4}} \varphi_T(x_{\varepsilon}) & \text{on }2B\cap U,
        \\ \varphi_T& \text{on }T\backslash B,
    \end{cases} 
\end{align*}
The implied constants depend only on $\alpha$. An explicit expression comparable to $\varphi_T$ is given by Theorem \ref{triangleprofile}.
\end{Theo}

Given the statement of Theorem \ref{addonetriangle2}, some remarks are in order. First, the ball $B$ always contains the small triangle $\triangle C_1C_2C_3$, and its radius $\varepsilon$ is only taken sufficiently small so that $B$ stays at a positive distance away from the two other sides. Second, note that the two cases $x\in 2B\cap U$ and $x\in T\backslash B$ are not disjoint: they overlap when $x\in U$ is in $2B\backslash B$, in which case $\varphi_U(x)$ is uniformly comparable to either expression written in Theorem \ref{addonetriangle2}. Third, as $\varepsilon\to 0$, the expression for $\varphi_U$ \Quote{converges} to the expression for $\varphi_T$ from Theorem \ref{triangleprofile}.

We now describe the general method that can be used to prove Theorem \ref{addonetriangle2} and many of its variations. In Figure \ref{trianglesperturbation} below, consider a planar triangle $T=\triangle A_1A_2A_3$ with angles bounded below by $\alpha>0$. By rescaling if necessary, assume the side $l_1$ has length $1$.
Let $\gamma>0$ and $\beta, N\in \mathbb{N}$. Replace the side $l_1$ by $N^{\beta}$ identical isosceles triangles with height $N^{-\gamma}$ and perpendicular to $l_1$. Let $U\supseteq T$ be the domain resulting from this replacement, and let $\theta$ denote the angle $\angle A_3 X_1Y_1$. 

    \begin{figure}[H]
  \centering
  \includegraphics[width=0.9\textwidth]{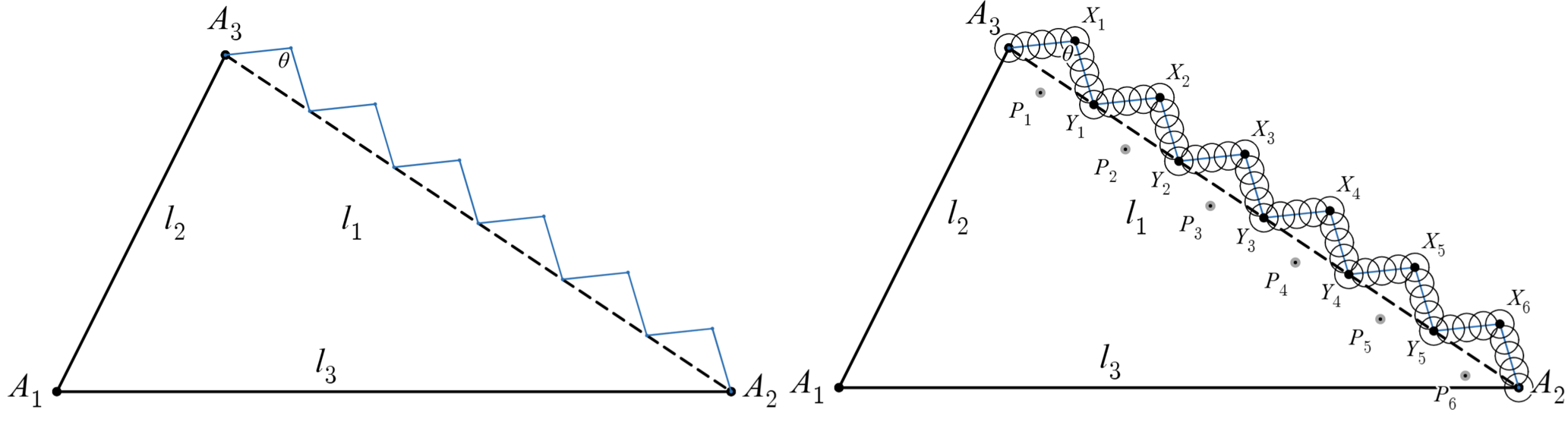}
  \caption{A triangle $T$ is perturbed by adding $6$ smaller triangles to one of its sides.  }
  \label{trianglesperturbation}
\end{figure}

Note that by construction, if $\theta$ is small, then the height of the isosceles triangles must also be small, and Theorem \ref{examplethm1} applies. On the other hand if $\theta$ is bounded away from $0$, then $U$ is an inner uniform domain and Theorem \ref{ub} applies. Combining the two cases, there exists a sufficiently small $\varepsilon=\varepsilon(\alpha)>0$ such that on $(1-\varepsilon)T$ (the dilation being with respect to $A_1$),
\begin{align}
    \label{triangleper1}
    \varphi_U\asymp \varphi_T,
\end{align}
with the implied constants depending only on $\alpha$. 

Now we assume that $\gamma\geq \beta$ so that $\theta$ is bounded away from zero. Consider a chain of $\asymp N^{\gamma}$ balls $\{B_i\}$ centered on the boundary of $U$ and with radius comparable to $N^{-\gamma}$. Choose the $B_i$ so that their centers include the vertices $\{A_3,X_1,Y_1,X_2,Y_2,...,X_{N^{\beta}-1},Y_{N^{\beta}-1},X_{N^{\beta}},A_2\}$ of the perturbed domain $U$.

On each boundary ball $B_i$, a boundary Harnack principle (Theorem 5.12, \cite{lierllsc}) implies that $\varphi_U(x)/\varphi_U(y)\asymp h(x)/h(y)$ for all $x,y\in \overline{B_i\cap U}$, where $h$ is any nonnegative harmonic function vanishing on $\overline{B_i}\cap \partial U$.
Combining this with our results on eigenfunction comparison, we get that on $\overline{B_i\cap U}$
\begin{align}  
    \label{triangleper2}
    \varphi_U\asymp \frac{\varphi_T(P_i)}{h(P_i)}h,
\end{align}
where $P_i\in T\backslash (1-\varepsilon)T$ is any point such that $d(P_i,\partial T)\gtrsim N^{-\gamma}$ and $d(P_i,B_i)\asymp N^{-\gamma}$; such points $P_i$ exist due to the inner uniformity of $U$ (see Proposition \ref{ballinside}). See Figure \ref{trianglesperturbation} above for the $P_i$ corresponding to the balls $B_i$ centered at $X_i$. If $B_i$ is centered on a vertex of $U$, say $X_1$, then we can take (in appropriate polar coordinates) $$h(r,\theta)=r^{\pi/\theta_i}\sin(\pi\theta/\theta_i) \asymp d(\cdot, \overline{A_3X_1})d(\cdot,\overline{X_1Y_1})\Big\{d(\cdot, \overline{A_3X_1})+d(\cdot,\overline{X_1Y_1})\Big\}^{\pi/\theta_i-2},$$
where $\theta_i$ is the interior angle of the vertex. On the other hand, if $B_i$ is centered on a flat part of $\partial U$ (say the midpoint of $\overline{A_3X_1}$), then thinking of the flat part locally as the upper half plane in $\mathbb{R}^2$, we can take  
$$h(x,y)= y \asymp d(\cdot,\overline{A_3X_1}).$$
For the portion of $U$ that is not contained inside any ball $B_i$ or $(1-\varepsilon)T$, comparison of principal Dirichlet eigenfunctions and a Harnack inequality argument gives
\begin{align}
    \label{triangleper3}
    \varphi_U(x)\asymp \varphi_T(P_i)
\end{align}
where $P_i$ is the point corresponding to the $X_i$ closest to $x$. Equations (\ref{triangleper1}), (\ref{triangleper2}), (\ref{triangleper3}), and Theorem \ref{triangleprofile}, together yield an explicit caricature function $\Phi_U\asymp \varphi_U$.

\begin{Exm} \label{polygons}
    \normalfont 
    The proof strategy in Theorem \ref{triangleprofile} more generally applies to a large class of polygons in $\mathbb{R}^2$ with a fixed number of sides. To be more specific, let $P\subseteq \mathbb{R}^2$ be a polygon with sides $l_i\subseteq \mathbb{R}^2$ and angles $\alpha_i>0$ for $i=1,2,...,N$, listed in counterclockwise order. Also, let $P_i\in \mathbb{R}^2$ denote the vertices of $P$ corresponding to angle $\alpha_i$. Several possible polygons for which our methods apply are shown in Figure \ref{polygonfigure} below. 
\end{Exm}

By abuse of notation, we will write $l_i$ to denote either the $i$th side as a subset of $\mathbb{R}^2$, or its side length. We equip the polygon $P$ not with the Euclidean distance, but with the geodesic distance obtained by minimizing lengths of paths in $P$ connecting two points. Thus all distances and balls in the remainder of this example are defined with respect to the geodesic distance on $P$. This allows us to treat, for example, polygonal domains with a slit (e.g. Figure \ref{polygonfigure}), where we think of a slit as two different sides of the polygon glued together. For instance, the square with one slit in Figure \ref{polygonfigure} should be thought of as a \Quote{polygon} with $7$ sides and $7$ angles.

    \begin{figure}[H]
  \centering
  \includegraphics[scale=0.45]{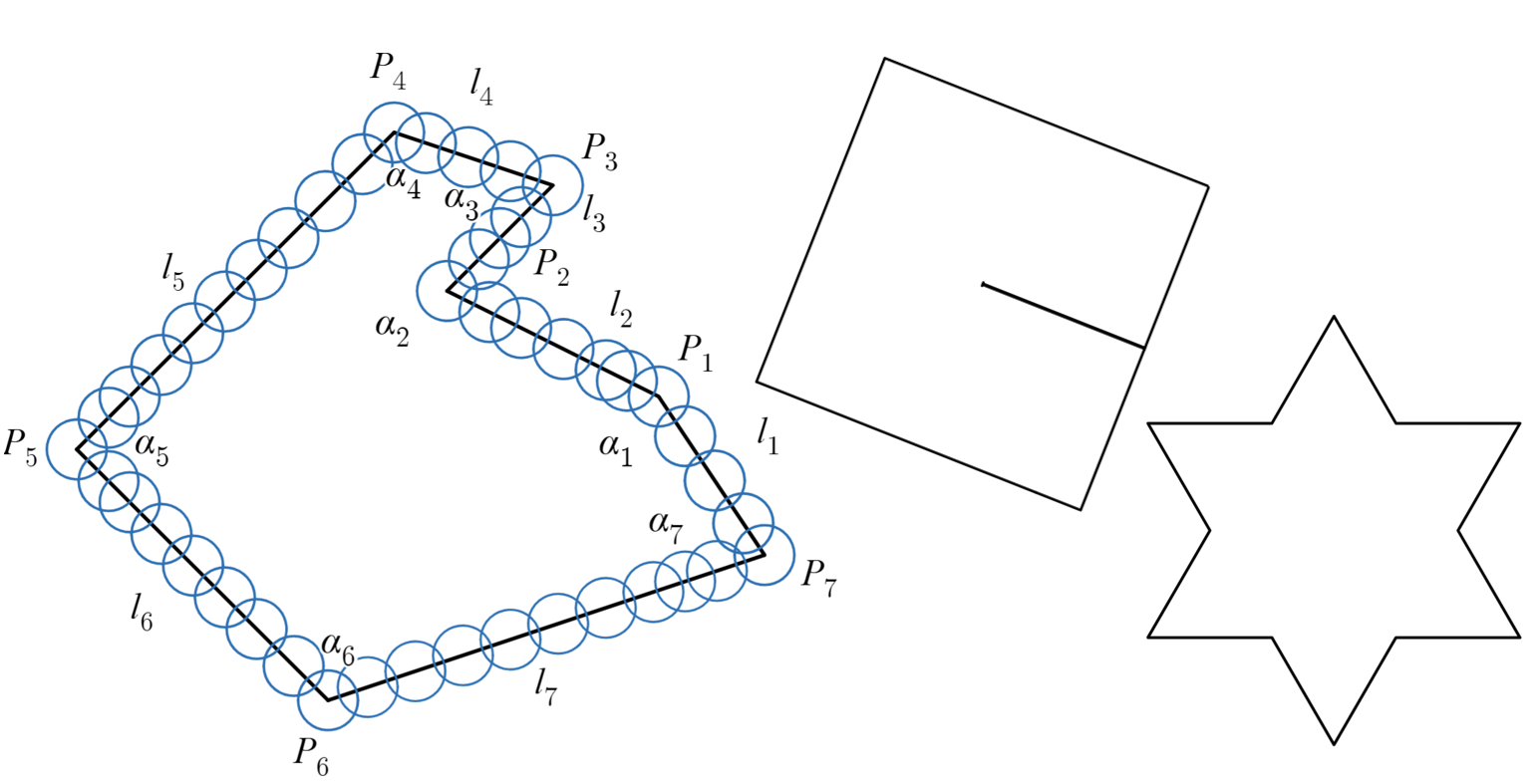}
  \caption{Some polygonal domains in $\mathbb{R}^2$ for which Theorem \ref{polygonal} applies.}
  \label{polygonfigure}
\end{figure}

We make the assumption that there are constants $C\geq 1$, $c\in (0,\frac{1}{2}\min l_i)$ and $r>0$ such that the following holds:

\begin{enumerate}
    \item $r/C\leq l_i\leq Cr$ for all $i=1,2,...,N$.
    \item There is a collection of at most $2NCr/c$ balls $\{B_j\}$ with radius $c$ that cover $\partial P$, where each $B_j$ is either centered at a vertex point of $P$, or is centered at a non-vertex boundary point in $\partial P$. Assume that each $B_j$ intersects either one $l_i$ or two $l_i$, depending on if $B_j$ is centered at a vertex or not. (See Figure \ref{polygonfigure}.)

    \item Further assume that if $B_j$ intersects the sides $\{l_j\}_{j\in J}$, then for all $i\notin J$, $d_i(\cdot ):=\text{dist}(\cdot,l_i)$ has $r/C\leq d_i \leq Cr$ on $B_j$.

    \item The angles $\alpha_i$ are uniformly bounded below: $\min\{\alpha_i:1\leq i\leq N\}>\alpha$ for some $\alpha>0$.
\end{enumerate}
The first assumption implies that all the sides $l_i$ of $P$ have roughly comparable length (excluding, for example, Figure \ref{trianglesperturbation2}). The second assumption is to ensure we can apply boundary Harnack principle (BHP) to compare $\varphi_P$ with a harmonic function near the boundary $\partial P$. The third assumption ensures that nonadjacent sides of the polygon are neither too close nor too far away from each other. The fourth assumption gives \Quote{local inner uniformity}, a necessary hypothesis to apply the BHP from \cite{lierllsc}. 

With essentially the same line of reasoning used to prove Theorem \ref{triangleprofile}, we get the following result.

\begin{Theo}
    \label{polygonal}
    Fix $N\geq 3$. Let $P\subseteq \mathbb{R}^2$ be a polygon with $N$ sides $l_1,l_2,...,l_N$ and angles $\alpha_i>0$, as pictured in Figure \ref{polygonfigure}.  Suppose $P$ satisfies the above four assumptions. Let $d_i(\cdot):=\textup{dist}(\cdot,l_i)$ be the geodesic distance (on $P$) to the $i$th side. If $\varphi_P$ is the principal Dirichlet Laplacian eigenfunction of $P$ normalized so that $\|\varphi_P\|_{L^2(P)}=1$, then 
    $$\varphi_{P}\asymp \frac{d_1d_2\cdots d_N(d_1+d_2)^{\pi/\alpha_1-2}(d_2+d_3)^{\pi/\alpha_2-2}\cdots(d_N+d_1)^{\pi/\alpha_N-2}}{r^{\pi/\alpha_1+\pi/\alpha_2+\cdots+\pi/\alpha_N-N+1}},$$
    where the implied constants depend only on $C,c,\alpha$, and $N$. 
\end{Theo}

\begin{Exm} \label{regularpoly}
\normalfont
Consider a regular polygon $P(n,l)\subseteq \mathbb{R}^2$ with $n$ vertices and side length $l$ (so that the perimeter of $P(n,l)$ is $nl$). To find a caricature function for $\varphi_{P(n,l)}$, we could use Theorem \ref{polygonal}, but it only applies when there is a fixed number of sides $n$, and does not give an expression that is uniformly comparable to $\varphi_{P(n,l)}$ over all $n$. However, this difficulty can be overcome by combining boundary Harnack principle with comparison of principal Dirichlet eigenfunctions, which give the following result.    

\begin{Theo}    
    \label{polythm}
    Let $P(n,l)\subseteq \mathbb{R}^2$ be a regular polygon as above. Label the sides of $P(n,l)$ counterclockwise as $l_1,l_2,...,l_n$, and let $d_i(x)=\textup{dist}(x,l_i)$. Then if $\varphi_{P(n,l)}$ is the principal Dirichlet Laplacian eigenfunction of $P(n,l)$ normalized so that $\|\varphi_{P(n,l)}\|_{L^2(P(n,l))}=1$, we have
    \begin{align*}
        \varphi_{P(n,l)}\asymp \frac{\min\{d_1d_2,d_2d_3,...,d_nd_1\}\cdot \min\{d_1+d_2,d_2+d_3,...,d_n+d_1\}^{\frac{n}{n-2}-2}}{(nl)^{\frac{n}{n-2}+1}}.
    \end{align*}
    The implied constants are absolute.
\end{Theo}
The proof of Theorem \ref{polythm}, which we sketch in Section \ref{proofsketch}, is similar to that of Theorem \ref{triangleprofile}. The key difference is that the proof of Theorem \ref{polythm} also relies on comparing the principal Dirichlet eigenfunctions of $P(n,l)$ and a ball in $\mathbb{R}^2$.

As in Example \ref{triangleexample}, it is possible to use Theorem \ref{polythm} together with our main results to obtain explicit expressions comparable to $\varphi_U$ when $U\supseteq P(n,l)$ is a polygonal perturbation of $P(n,l)$ (e.g. Section \ref{introduction}, Figure \ref{polygon1}). 

\end{Exm}

\begin{Exm} 
    \label{3dcube}
    \normalfont
   In Examples \ref{triangleexample}, \ref{polygons}, and \ref{regularpoly}, the only reason why we restrict to polygonal domains $P\subseteq \mathbb{R}^2$ is because harmonic functions are often explicitly computable in $\mathbb{R}^2$, but not in $\mathbb{R}^n$ for $n\geq 3$. Specifically, the boundary Harnack principle implies that near an angle $\theta_0$,  $\varphi_P$ behaves like the harmonic profile $h(r,\theta)=r^{\pi/\theta_0}\sin(\pi\theta/\theta_0)$ of the infinite planar cone $\{re^{i\theta}\in \mathbb{R}^2:r>0,\theta\in (0,\theta_0)\}$. 
     In higher dimensions $n\geq 3$, it is known that the harmonic profile of $\{(r,\theta)\in \mathbb{R}^n:r>0,\theta\in U\subseteq \mathbb{S}^{n-1}\}$ is of the form $h(x)=\|x\|^{\alpha}\varphi_U(x/\|x\|)$ (see e.g. Proposition 6.31, \cite{gyryalsc}). Here $\alpha$ is an exponent determined by $\alpha(\alpha+n-2)=\lambda_U$, where $\lambda_U$ is the first Dirichlet eigenvalue of $U\subseteq \mathbb{S}^{n-1}$.
     In full generality, the exact values of $\lambda_U$ (and hence $\alpha$) are not known. In \cite{ocsc}, a numerical algorithm due to Grady Wright is used to approximate $\lambda_U$ when $U\subseteq \mathbb{S}^2$ is a spherical equilateral triangle.

     In this example, we work in dimension $n=3$, and we describe the caricature function of a cube with a smaller cube removed (Figure \ref{cubefig}). The caricature function given will be fully explicit except for an exponent $\alpha$ which is not known to have an explicit expression. To begin, let $V=[0,L]^3\subseteq \mathbb{R}^3$ be a cube with side length $L$, and for any $\varepsilon\in (0,L/4)$, define
     $$U=V\backslash \{L-\varepsilon \leq x,y,z\leq L\}.$$

       \begin{figure}[H]
  \centering
  \includegraphics[scale=0.135]{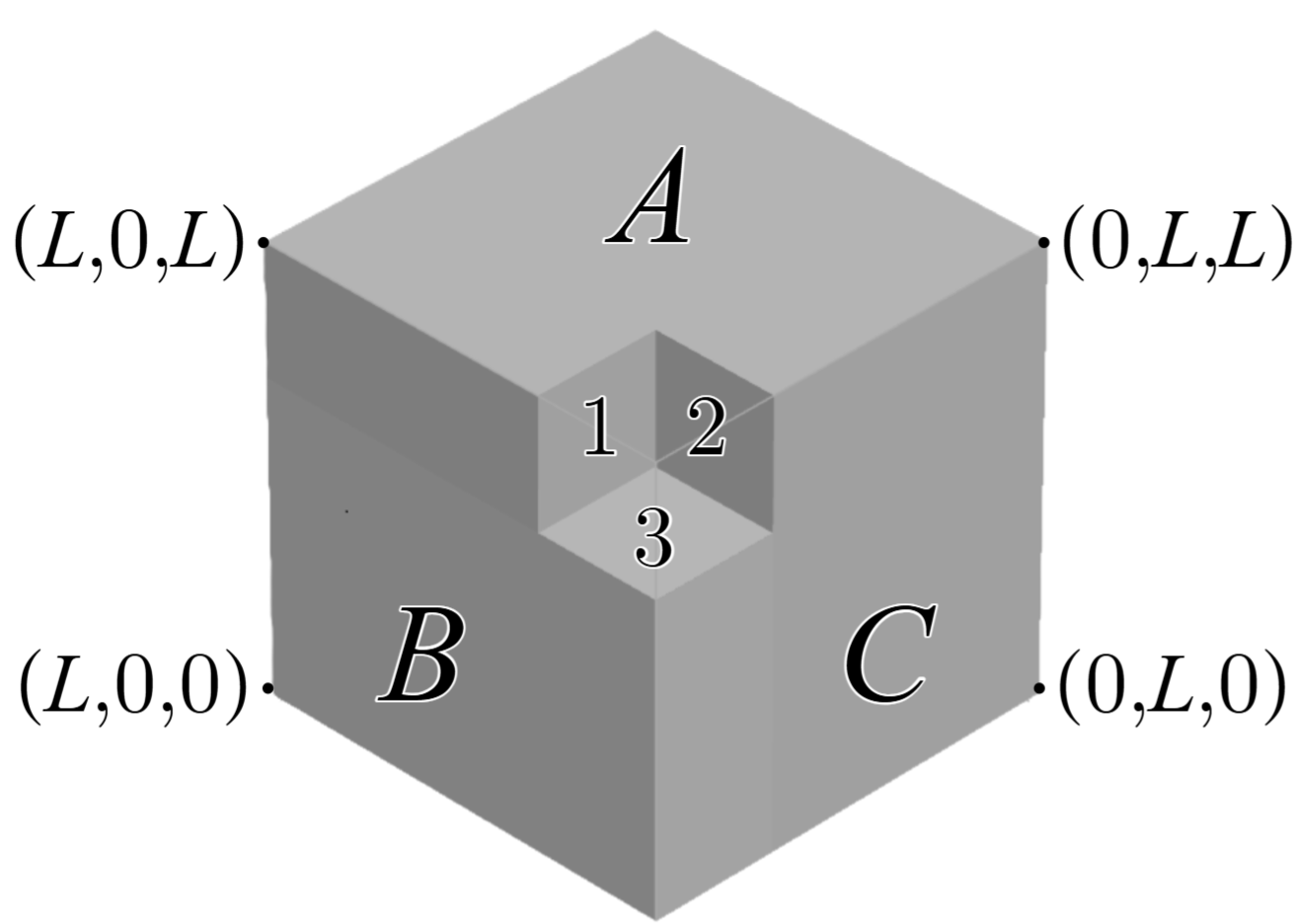}
  \caption{$U$ is obtained from removing a small $\varepsilon$-cube from $[0,L]^3\subseteq \mathbb{R}^3$.}
  \label{cubefig}
\end{figure}
     
\end{Exm}

As in Figure \ref{cubefig}, label the outward faces of $U$ as $A,B,C,1,2,3$. For a label $i\in \{A,B,C,1,2,3\}$ and $x\in U$, let $d_i(x)$ denote the length of the shortest path in $U$ joining $x$ to face $i$. For $\varepsilon>0$ and $x\in U$, we define in coordinates a distinguished point $x_{\varepsilon}\in [0,L]^3$ by $x_{\varepsilon}=(L-2\varepsilon,L-2\varepsilon,L-2\varepsilon)$. Like in Example \ref{triangleexample}, $x_{\varepsilon}$ may be replaced with any other $x'_{\varepsilon}\in [0,L]^3$ that is at distance $\asymp \varepsilon$ away from the small removed cube, and distance $\gtrsim \varepsilon$ away from the boundary of $[0,L]^3$.

Via a similar argument as in Example \ref{triangleexample}, we obtain the following result. 

\begin{Theo}  
    \label{cubethm}
    Fix $L>0$ and $\varepsilon\in (0,L/4)$. Let $U\subseteq \mathbb{R}^3$ be the domain obtained from removing an $\varepsilon$-cube from the cube $V=[0,L]^3$, as in Figure \ref{cubefig}. Let $\varphi_U$ (resp. $\varphi_V$) denote the principal Dirichlet Laplacian eigenfunction of $U$ (resp. $V$), normalized so that $\|\varphi_U\|_{L^2(U)}=1$ (resp. $\|\varphi_V\|_{L^2(V)}=1$). Then
    \begin{align*}
        \varphi_U\asymp\begin{cases}
            \displaystyle\Big(\prod_{i\in \{1,2,3,A,B,C\}}\frac{d_i}{d_i+\varepsilon}\Big)
            \Big(\frac{(d_1+d_2+\varepsilon)(d_2+d_3+\varepsilon)(d_3+d_1+\varepsilon)}{(d_1+d_2)(d_2+d_3)(d_3+d_1)}\Big)^{4/3}
            \\ \hspace{0.3in}\displaystyle \times \Big(\frac{d_1^2+d_2^2+d_3^2}{(d_1+\varepsilon)^2+(d_2+\varepsilon)^2+(d_3+\varepsilon)^2}\Big)^{\frac{\alpha+1}{2}}\varphi_V(x_{\varepsilon})& \text{on }U\cap [L-2\varepsilon,L]^3
            \\  \displaystyle \varphi_V & \text{on }U\backslash[L-2\varepsilon,L]^3. 
        \end{cases}
    \end{align*}
    All of the implied constants are absolute. Here
    $$\varphi_V(x,y,z)=\Big(\frac{2}{L}\Big)^{3/2}\sin\Big(\frac{\pi x}{L}\Big)\sin\Big(\frac{\pi y}{L}\Big)\sin\Big(\frac{\pi z}{L}\Big),$$
    and $\alpha\in \mathbb{R}$ is determined by $\alpha(\alpha+1)=\lambda_S$, where $\lambda_S$ is the first Dirichlet eigenvalue of
    $$S=\mathbb{S}^2\backslash \{(\theta,\psi):\theta\in (0,\pi/2),\psi\in (0,\pi/2)\}$$
    with respect to the Laplace-Beltrami operator on $\mathbb{S}^2=\{(\theta,\psi):\theta\in [0,2\pi),\psi\in [0,\pi]\}$. 
\end{Theo}

The interpretation of Theorem \ref{cubethm} in terms of the local boundary behavior of $\varphi_U$ near a vertex of $U$ is as follows. Near the vertex at the intersection of the faces $\{A,B,1\}$, $\varphi_U$ is comparable to the principal Dirichlet eigenfunction of a $3$-dimensional box; Theorem \ref{sample} and the boundary Harnack principle (BHP) in \cite{lierllsc} gives

\begin{align}
    \label{0.001}
    \varphi_U\asymp \frac{d_1d_Ad_B}{(d_1+\varepsilon)(d_A+\varepsilon)(d_B+\varepsilon)}\varphi_V(x_{\varepsilon}).
\end{align}
Near the vertex at the intersection of the faces $\{A,1,2\}$, the domain $U$ is locally the Cartesian product of face $A$ (as domain in $\mathbb{R}^2$) with $[0,1]\subseteq \mathbb{R}$, and comparing $\varphi_U$ to the principal Dirichlet eigenfunction of such a domain yields 
\begin{align}
\label{0.002}
    \varphi_U\asymp \frac{d_Ad_1d_2(d_1+d_2)^{-4/3}}{(d_A+\varepsilon)(d_1+\varepsilon)(d_2+\varepsilon)(d_1+d_2+\varepsilon)^{-4/3}}\varphi_V(x_{\varepsilon}).
\end{align}
Near the vertex $(L-\varepsilon,L-\varepsilon,L-\varepsilon)$, 
$\varphi_U(x)$ can be compared to the harmonic profile $\|x\|^{\alpha}\varphi_S(x/\|x\|)$, described at the beginning of this example. By Theorem \ref{triangleprofile}, we know an explicit expression comparable to $\varphi_S$. Thus,
\begin{align}
    \nonumber
    \varphi_U& \asymp \Big(\prod_{i\in \{1,2,3\}}\frac{d_i}{d_i+\varepsilon}\Big)
            \Big(\frac{(d_1+d_2+\varepsilon)(d_2+d_3+\varepsilon)(d_3+d_1+\varepsilon)}{(d_1+d_2)(d_2+d_3)(d_3+d_1)}\Big)^{4/3}
            \\ \label{0.003}& \hspace{0.5in}\times\Big(\frac{d_1^2+d_2^2+d_3^2}{(d_1+\varepsilon)^2+(d_2+\varepsilon)^2+(d_3+\varepsilon)^2}\Big)^{\frac{\alpha+1}{2}}\varphi_V(x_{\varepsilon}).
\end{align}

By the obvious symmetries of $U$, Equations (\ref{0.001}), (\ref{0.002}) and (\ref{0.003}) cover the local boundary behavior of $\varphi_U$ near all vertices on the removed little cube.

\section{Preliminaries}
\label{preliminaries}

\subsection{Notation}
$\|x\|$ - Euclidean norm of $x\in \mathbb{R}^n$
\\
$C_c(\Omega)$ - the space of continuous compactly supported functions in $\Omega$
\\
$\text{supp}(f):=\text{supp}(|f|d\mu)$ for a function $f$ defined on a measure space $(X,\mu)$
\\
$A\lesssim B$ denotes $A\leq CB$ where $C>0$ is some constant depending only on important parameters
\\ $A\lesssim_{\alpha} B$ denotes $A\lesssim B$ where the implied constant depends on $\alpha$
\\ $A\asymp B$ denotes $A\lesssim B\lesssim A$
\\ $R>0$ always denotes the parameter in Assumption \ref{assumption2} 
\\ $X\subseteq Y$ only means $X$ is a subset of $Y$, and does not imply whether or not $X$ can equal $Y$ 
\\ For $a,b\in \mathbb{R}$, $a\vee b:=\max\{a,b\}$ and $a\wedge b:=\min\{a,b\}$
\\ $\tau_U:=\inf\{t>0:X_t\notin U\}$ is the first time $X_t$ leaves $U$
\\ $\sigma_U:=\inf\{t>0:X_t\in U\}$ is the first time $X_t$ hits $U$

\subsection{Dirichlet space $(X,d,\mu,\mathcal{E},\mathcal{D}(\mathcal{E}))$}
\label{dirichletspace}

The setting and definitions introduced in this section are taken from the works \cite{gyryalsc}, \cite{hebischlsc}, and \cite{lierllsc}. Let $X$ be a connected, locally compact, and separable space equipped with a positive Radon measure $\mu$ with $\text{supp}(\mu)=X$. We will assume that $X$ is equipped with a Dirichlet form $(\mathcal{E},\mathcal{D}(\mathcal{E}))$ with domain $\mathcal{D}(\mathcal{E})\subseteq L^2(X,\mu)$. Further, assume that
\begin{enumerate}
    \item $\mathcal{E}$ is \textit{strictly local}, i.e. for any $u,v\in \mathcal{D}(\mathcal{E})$ with compact support such that $u$ is constant on a neighborhood of the support of $v$, it holds that $\mathcal{E}(u,v)=0$;
    \item $\mathcal{E}$ is \textit{regular}, i.e. there is $\mathcal{C}\subseteq  \mathcal{D}(\mathcal{E})\cap C_c(X)$ that is dense in $C_c(X)$ with respect to the uniform norm and dense in $\mathcal{D}(\mathcal{E})$ with respect to the norm $(\|f\|_2^2+\mathcal{E}(f,f))^{1/2}$. We call $\mathcal{C}$ a \textit{core} of the Dirichlet form $\mathcal{E}$.
\end{enumerate}
A strictly local and regular Dirichlet form $(\mathcal{E},\mathcal{D}(\mathcal{E}))$ can be written as $\mathcal{E}(u,v)=\int_X d\Gamma(u,v)$, where, for fixed $u,v\in \mathcal{D}(\mathcal{E})$, $\Gamma(u,v)$ is a signed Radon measure which we call the \textit{carr\'{e} du champ}. Moreover, we can define an \textit{intrinsic distance} on $X$ by setting
\begin{align}
    \label{intrinsicdist}
    d(x,y):=\text{sup}\{f(x)-f(y):f\in\mathcal{C}\text{ and }d\Gamma(f,f)\leq d\mu\}.
\end{align}
In full generality, $d$ is symmetric and satisfies the triangle inequality. However, it is possible that $d(x,y)=0$ or $d(x,y)=\infty$ for $x\neq y$. For example, \cite{ablsc} gives an example of $d(x,\cdot)=\infty$ on a full measure subset of an infinite-dimensional space. Therefore, $d$ is in general only a pseudo-distance on $X$. In this paper, we will always assume that the following properties on the pseudo-distance $d$ are satisfied:
\begin{enumerate}
    \item[(A)] The pseudo-distance $d$ is finite everywhere, continuous, and generates the original topology on $X$. 
    \item[(B)] $(X,d)$ is a complete metric space.
\end{enumerate}
Assumptions (A) and (B) are satisfied in all of the examples from Section \ref{examples}. In Example \ref{soue}, the intrinsic distance is comparable to the Euclidean distance on $\mathbb{R}^n$. In Example \ref{rm} below, the intrinsic distance equals the Riemannian distance. We refer the reader to (\cite{fukushima},\cite{chenfuku}) for the general theory of Dirichlet forms, and to \cite{sturm} for the intrinsic distance associated to a Dirichlet form.

\begin{Assum}\label{assumption1}
\normalfont
Henceforth, $(X,d,\mu,\mathcal{E},\mathcal{D}(\mathcal{E}))$ will denote a Dirichlet space satisfying all of the above conditions with induced intrinsic distance $d$. We will sometimes only write $(X,d)$ (resp. $(X,d,\mu)$) to emphasize the metric (resp. metric measure) structure of $X$, and only write $(\mathcal{E},\mathcal{D}(\mathcal{E}))$ to emphasize the Dirichlet form.

\end{Assum}

The intrinsic distance $d$ on $X$ allows us to define the lengths of curves, and hence the geodesic distance on a open subset $U\subseteq X$, as in the below definitions.

\begin{Dfn}
    \normalfont
    In a metric space $(X,d)$, we define the length of a continuous curve $\gamma:[a,b]\to X$ by
    $$L(\gamma)=\sup\Big\{\sum_{j=1}^{n}d(\gamma(t_{j-1}),\gamma(t_{j})):n\in\mathbb{N},a\leq t_0<\cdots <t_n\leq b\Big\}.$$
\end{Dfn}

\begin{Dfn}
    \normalfont
    For an open subset  $U$ of a metric space $(X,d)$, we define the \textit{intrinsic distance $d_U$ associated to $U$} by
    \begin{align*}
        d_U(x,y) & = \inf\{L(\gamma):\gamma:[0,1]\to U\text{ continuous, }\gamma(0)=x,\gamma(1)=y\}.
    \end{align*}
\end{Dfn}

\begin{Dfn}
    \normalfont A metric space $(X,d)$ is called a \textit{length space} if $d_X(x,y)=d(x,y)$ for any points $x,y\in X$. That is, the intrinsic distance on $X$ coincides with the distance induced by the Dirichlet form $(\mathcal{E},\mathcal{D}(\mathcal{E}))$. Any Dirichlet space $(X,d,\mu,\mathcal{E},\mathcal{D}(\mathcal{E}))$ as in Assumption \ref{assumption1} is a length space, see e.g. Theorem 2.11, \cite{gyryalsc}.
\end{Dfn}

\begin{Dfn}
    \normalfont
    For an open subset $U$ of a length metric space $(X,d)$, let $\widetilde{U}$ be the completion of the metric space $(U,d_U)$, equipped with the natural extension of the metric $d_U$. We write $\partial U=\widetilde{U}\backslash U$.
\end{Dfn}

\subsection{Inner uniform domains}

We now review the definition of inner uniform domains and some properties of such domains. Inner uniform domains were introduced by Martio and Sarvas \cite{martiosarvas} as well as Jones \cite{jones}, and we refer the reader to \cite{martiosarvas}, \cite{jones}, \cite{gyryalsc}, \cite{lierllsc} for further basic properties and examples of inner uniform domains.   

\begin{Dfn} 
    \normalfont
    \label{iudefn}
    Let $U$ be an open connected subset of a length metric space $(X,d)$. We say that $U$ is $(C_0,c_0)$\textit{-inner uniform} if there are constants $C_0\in [1,\infty),c_0\in (0,1]$ such that for any $x,y\in U$, there exists a continuous curve $\gamma:[0,1]\to U$ connecting $x$ to $y$ with length at most $C_0d_U(x,y)$, and such that for any $z\in \gamma([0,1])$, 
    \begin{align}
        \label{iu}
        d(z,\partial U)\geq c_0\frac{d_U(x,z)d_U(y,z)}{d_U(x,y)}.
    \end{align}
\end{Dfn}

As $\max\{d_U(x,z),d_U(y,z)\}\geq d_U(x,y)/2$, it is sometimes convenient to replace (\ref{iu}) in the definition of inner uniformity with the equivalent condition
\begin{align}
    \label{iu2}
    d(z,\partial U)\geq \frac{c_0}{2}\min\{d_U(x,z),d_U(y,z)\}.
\end{align}

\begin{Exm}
    \label{iuexamples}
    \normalfont
    Consider $X =\mathbb{R}^n$ equipped with the Euclidean distance. 
    \begin{enumerate}
        \item Any open convex set $U\subseteq \mathbb{R}^n$ with $B(x,r_1)\subseteq U\subseteq B(x,r_2)$ is inner uniform, with $(C_0,c_0)$ depending only on $r_2/r_1$.
        \item As noted in \cite{jones}, any bounded Lipschitz domain in $\mathbb{R}^n$ is inner uniform for some $(C_0,c_0)$, but the class of all bounded Lipschitz domains in $\mathbb{R}^n$ fails to be \textit{uniformly} inner uniform. 
        \item In $\mathbb{R}^2$, the Von Koch snowflake is $(C_0,c_0)$-inner uniform for some $(C_0,c_0)$. This follows from the fact that the snowflake is the quasiconformal image of a disk in $\mathbb{R}^2$, see \cite{martiosarvas}. The inner uniformity of the snowflake is also proven directly in \cite{gyryalsc}.
        \item In $\mathbb{R}^n$, consider the unbounded domain $U_{\beta}=\{x\in \mathbb{R}^n:x_{n}>(x_1^2+x_2^2+\cdots+x_{n-1}^2)^{\beta/2}\}$ for $\beta>0$. Then $U_{\beta}$ is inner uniform if and only if $\beta=1$. When $\beta\in (0,1)$, the constant $c_0\in (0,1)$ does not exist because of the lack of admissible paths connecting $(0,0,...,0,1)$ to points near the origin. With similar reasoning, the constant $C_0\in [1,\infty)$ does not exist when $\beta>1$. When $\beta=1$, the set $U_{\beta}$ is the domain above the graph of a Lipschitz function, thus inner uniform (see Proposition 6.6 of \cite{gyryalsc}). Any bounded domain in $\mathbb{R}^n$ that locally has a cusp like the one in $U_{\beta}$ for $\beta\in (0,1)$ is not inner uniform.
    \end{enumerate} 
\end{Exm}

\begin{Dfn}
    \normalfont The open balls in $(U,d_U)$ and $(\widetilde{U},d_U)$ are defined by
    $$B_U(x,r):=\{z\in U:d_U(z,x)<r\}\text{ and }B_{\widetilde{U}}(x,r)=\{z\in \widetilde{U}:d_U(z,x)<r\}.$$ 
    By convention, the radii of such balls are taken to be minimal. We also define the \textit{inner diameter} of $U$ to be 
    \begin{align*}
        \text{diam}_U:=\sup_{x,y\in U} d_U(x,y).
    \end{align*}
\end{Dfn}
The next proposition, which is a restatement of Lemma 3.20 in \cite{gyryalsc}, states that inner uniform domains $U$ contain a ball with radius comparable to its inner diameter $\text{diam}_U$.
\begin{Prop}
    \label{ballinside}
    Let $U$ be a $(C_0,c_0)$-inner uniform domain. For every ball $B=B_{\widetilde{U}}(x,r)$, there is some $y\in B$ with $d(y,\partial U)\geq c_0r/8$ and $d_U(y,x)=r/4$.
    In particular, if $U=B_U(x,r)$ for $r>0$ chosen to be minimal, then there exists a ball $B_X(y,c_0r/8)$ centered at some $y\in U$ that is contained inside $U$.
\end{Prop}

\subsection{$R$-doubling, $R$-Poincar\'{e} inequality, and $R$-parabolic Harnack inequality}
\label{R-DPH}

Following \cite{hebischlsc}, we now introduce the notions of $R$-doubling and $R$-scale-invariant Poincar\'{e} inequality on a Dirichlet space. 
\begin{Dfn}
    \label{Rdoubling}
    \normalfont
    Fix $R\in (0,\infty]$. We say that $(X,d,\mu)$ is $R$-\textit{doubling} if there is a constant $D_0$ such that 
    for any ball $B\subseteq X$ of radius at most $R$, 
    $$\mu(2B)\leq D_0\mu(B).$$
    A $R$-doubling space with $R=\infty$ will simply be called \textit{doubling}. 
    \begin{Dfn}
        \label{Rpoincare}
        \normalfont
        Fix $R\in (0,\infty]$.  We say $(X,d,\mu,\mathcal{E},\mathcal{D}(\mathcal{E}))$ satisfies a $R$-\textit{scale-invariant Poincar\'{e} inequality} if there is a constant $P_0$ such that for any ball $B\subseteq X$ of radius $r\leq R$ and for any $f\in \mathcal{D}(\mathcal{E})$,
        $$\min_{\xi\in\mathbb{R}}\int_{B}|f-\xi|^2d\mu \leq P_0 r^2\int_B d\Gamma(f,f).$$
    \end{Dfn}
\end{Dfn}

\begin{Assum}
    \label{assumption2}
    \normalfont
    We will assume that the Dirichlet space $(X,d,\mu,\mathcal{E},\mathcal{D}(\mathcal{E}))$ satisfies both $R$-doubling and a $R$-scale-invariant Poincar\'{e} inequality for some $R\in (0,\infty]$.
\end{Assum}

We also define a parabolic Harnack inequality up to scale $R$ from \cite{hebischlsc}. In the below, $\mathcal{L}$ denotes the infinitesimal generator associated to the Dirichlet form $(\mathcal{E},\mathcal{D}(\mathcal{E}))$.
\begin{Dfn}
    \label{RPHI}
    \normalfont
    Fix $R\in (0,\infty]$. We say that $(X,d,\mu,\mathcal{E},\mathcal{D}(\mathcal{E}))$ satisfies a $R$-\textit{scale-invariant parabolic Harnack inequality} if there exists a constant $H_0$ such that for any $r\in (0,R], x\in X$, and nonnegative weak solution $u$ of the heat equation $(\partial_t-\mathcal{L})u=0$ in $Q=(0,r^2)\times B(x,r)$,
    $$\sup_{Q_-} u(t,x)\leq H_0\inf_{Q_+}u(t,x),$$
    where $Q_+=(r^2/4,r^2/2)\times B(x,r/2)$ and $Q_-=(3r^2/4,r^2)\times B(x,r/2)$.
    
\end{Dfn}

\begin{Rmk}
\normalfont
As observed in \cite{hebischlsc}, if a space satisfies any one of Definitions \ref{Rdoubling}, \ref{Rpoincare}, or \ref{RPHI} with constant $R>0$, then it also satisfies the same definition for all finite $R'\geq R$ with larger constants $D'_0,P'_0$, or $H'_0$. Thus, we only have to distinguish between the cases where $R<\infty$ and $R=\infty$.
\end{Rmk}

To emphasize the connection between volume doubling, Poincar\'{e} inequality, and parabolic Harnack inequality, we mention the following theorem, due to the works \cite{grigoryan}, \cite{saloff-coste}, and \cite{sturm2}.    

\begin{Theo}\label{breakthrough} Let $(X,d,\mu,\mathcal{E},\mathcal{D}(\mathcal{E}))$ be a Dirichlet space and let $R\in (0,\infty]$. Then the following are equivalent.
\begin{enumerate}
    \item The Dirichlet space is $R$-doubling and satisfies a $R$-scale-invariant Poincar\'{e} inequality.
    \item The Dirichlet space satisfies a $R$-scale-invariant parabolic Harnack inequality.
    \item There exists $c_i>0$ such that the associated semigroup admits two-sided heat kernel bounds
    $$\frac{c_1}{\mu(B(x,\sqrt{t}))}\exp\Big(-\frac{d(x,y)^2}{c_2t}\Big)\leq p_X(t,x,y) \leq \frac{c_3}{\mu(B(x,\sqrt{t}))}\exp\Big(-\frac{d(x,y)^2}{c_4t}\Big),$$
    for all $t\in (0, R^2)$ and $x,y\in X$.
\end{enumerate}
\end{Theo}

For later use, we record a consequence of the $R$-doubling property; see e.g. \cite{hebischlsc} or \cite{asti}.

\begin{Prop}
    \label{volumeratioballs}
    Let $(X,d,\mu)$ be a $R$-doubling space with doubling constant $D_0$. 
    For all radii  $0<s<r<R$ and points $x,y\in X$ with $d(x,y)\leq r$, we have
    $$\frac{\mu(B(x,r))}{\mu(B(y,s))}\leq D_0^2\Big(\frac{r}{s}\Big)^{\log_2 D_0}.$$
\end{Prop}

In addition to Euclidean space, we end this section with two further examples of Dirichlet spaces satisfying both of Assumptions \ref{assumption1} and \ref{assumption2}.

\begin{Exm}\label{rm} \normalfont
Let $X=(M,g)$ be a complete Riemannian manifold and $\mu$ the Riemannian measure on $M$. The canonical Dirichlet form on $M$ is given by
\begin{align*}
    \mathcal{E}(f,f)=\int_M g(\nabla f,\nabla f) d\mu,\hspace{0.2in}f\in W^{1,2}(M).
\end{align*}
Here $\nabla f$ is the Riemannian gradient of $f$. The distance function (\ref{intrinsicdist}) induced by the Dirichlet form $\mathcal{E}$ is exactly the Riemannian distance on $M$, see e.g. \cite{sturm}. The class of Riemannian manifolds for which our results apply to are those with Ricci curvature bounded below. This means that for some $K\in \mathbb{R}$ we have $\text{Ric}\geq -Kg$ as symmetric 2-tensors. 

$(M,g)$ satisfies volume doubling (VD) and Poincar\'{e} inequalities (PI) up to scale $R<\infty$ if $K>0$ and up to scale $R=\infty$ if $K=0$. The volume doubling is a consequence of Bishop-Gromov volume comparison, and the Poincar\'{e} inequality is due to Buser \cite{buser}. See also Theorems 5.6.4 and 5.6.5 of \cite{asti}. 

The model space for $K<0$ is a sphere $\mathbb{S}^{n-1}(r)$ of radius $r$ in $\mathbb{R}^n$. By definition, for a sphere we could simply assume $K=0$  as it is weaker than $K<0$. However, in practice we will treat $\mathbb{S}^{n-1}(r)$ as satisfying (VD) and (PI) up to scale $R\asymp r$ (see Remark \ref{spherer} below).   
\end{Exm}

\begin{Exm}
    \normalfont
    \label{lg} Let $X=G$ be a connected unimodular Lie group equipped with Haar measure $d\mu$. Consider a family of left-invariant vector fields $\mathcal{X}=\{X_1,...,X_k\}$ which generates the Lie algebra $\mathfrak{g}$ in the sense that the span of the $X_i$ and their iterated Lie brackets generate $\mathfrak{g}$. We can define a Dirichlet form on $G$ by setting, for all $f\in L^2(G)$ such that $X_i f\in L^2(G)$,
    \begin{align*}
        \mathcal{E}_{\mathcal{X}}(f,f)=\int_G \sum_{i=1}^{k}|X_if|^2 d\mu. 
    \end{align*}
   The Dirichlet form $\mathcal{E}_{\mathcal{X}}$ induces a left-invariant distance function $d_{\mathcal{X}}$ on $G$ that is compatible with the topology of $G$ and such that $(G,d_{\mathcal{X}})$ is a complete metric space. Informally speaking, the distance $d_{\mathcal{X}}$ is obtained by minimizing lengths of paths that stay tangent to the span of $\mathcal{X}$. 
   
   In order to ensure that $G$ satisfies volume doubling (VD) and Poincar\'{e} inequalities (PI), as is assumed in Assumption \ref{assumption2}, we assume that $G$ has polynomial volume growth. 
   Indeed, the conjunction of (VD) and (PI) is equivalent to polynomial volume growth, see e.g. \cite{saloff-coste}, \cite{asti}, \cite{saloff-coste2}, as well as Theorem 2.2 of \cite{gyryalsc} from which this example is taken. A concrete example is the Heisenberg group, which is topologically $\mathbb{R}^3$ equipped with Lebesgue measure, the vector fields $\mathcal{X}=\{\partial_x-y\partial_z/2,\partial_y+x\partial_z/2\}$, and product $$(x_1,y_1,z_1)\cdot(x_2,y_2,z_2)=(x_1+x_2,y_1+y_2,z_1+z_2+(x_1y_2-x_2y_1)/2).$$
\end{Exm}

\subsection{Heat kernel, the generator $\mathcal{L}$, and Dirichlet eigenvalues}
\label{heatkernelsect}

Let $(\mathcal{E},\mathcal{D}(\mathcal{E}))$ be a Dirichlet form as in Assumption \ref{assumption1}. Let $(P_t)_{t>0}$ be the associated heat semigroup. By Theorem \ref{breakthrough},
    $(P_t)_{t>0}$ has a density with respect to the measure $\mu$, i.e. there is a non-negative measurable function $p_X(t,x,y)$ such that 
    $$P_t f(x)=\int_X p_X(t,x,y)f(y)d\mu(y).$$
Let $U\subseteq X$ be a domain (i.e. a nonempty connected open set with respect to the intrinsic distance $d$ as defined in (\ref{intrinsicdist})). On $U$, we can consider the heat semigroup $(P^{D}_{U,t})_{t>0}$ with Dirichlet boundary conditions. This semigroup can be constructed as the (sub-)Markovian semigroup associated to the Dirichlet form $(\mathcal{E}^D_U,\mathcal{D}(\mathcal{E}^D_U))$ on $L^2(U,d\mu)$; see Definitions 2.34 and 2.37 of \cite{gyryalsc}. We include the relevant notions from \cite{gyryalsc} and \cite{lierllsc} for completeness.

\begin{Dfn}
\label{flocu}
    For any open set $U\subseteq X$, set 
    \begin{align*}
    \mathcal{F}_{\textup{loc}}(U)=\{u\in L^2_{\textup{loc}}(U):\textup{for any compact }K\subseteq U\textup{ there exists }u^{\#}\in \mathcal{D}(\mathcal{E})\textup{ with }u=u^{\#}|_K\textup{ a.e.}\}.
    \end{align*}
\end{Dfn}

On the open set $U$, we can define a localized version of the carr\'{e} du champ operator $\Gamma_U$, defined on $\mathcal{F}_{\text{loc}}(U)\times \mathcal{F}_{\text{loc}}(U)$ by setting for all precompact $\Omega\subseteq U$
\begin{align*}
    \Gamma_U(u,v)|_{\Omega}=\Gamma(u^{\#},v^{\#})|_{\Omega},
\end{align*}
where $u^{\#}$ and $v^{\#}$ are as in Definition \ref{flocu} (with $K=\overline{\Omega}$). 
\begin{Dfn}
\label{fufcu}
    For any open set $U\subseteq X$, set
    \begin{align*}
        &\mathcal{F}(U) = \Big\{u\in \mathcal{F}_{\textup{loc}}(U):\|u\|^2_{L^2(U)}+\int_U d\Gamma_U(u,u)<\infty\Big\}
        \\ & \mathcal{F}_c(U)=\{u\in \mathcal{F}(U):\textup{supp}(u)\textup{ is compact in }U\}.
    \end{align*}
\end{Dfn}
With Definition \ref{fufcu}, we can now give a formal definition of the Dirichlet form and heat semigroup on the domain $U$ with Dirichlet boundary conditions.

\begin{Dfn}
    We define $(\mathcal{E}^D_U,\mathcal{D}(\mathcal{E}^D_U))$ to be the closure of $(\mathcal{E},\mathcal{F}_c(U))$ in $L^2(U,d\mu)$, and define $(P^D_{U,t})_{t>0}$ to be the associated heat semigroup.
\end{Dfn}

Assumptions \ref{assumption1} and \ref{assumption2} combined with Theorem \ref{breakthrough} imply the following well-known result.

\begin{Theo}
    Let $(X,d,\mu,\mathcal{E},\mathcal{D}(\mathcal{E}))$ be a Dirichlet space satisfying Assumptions \ref{assumption1} and \ref{assumption2}. Let $(P^D_{U,t})_{t>0}$ be the $L^2(U,d\mu)$-heat semigroup with Dirichlet boundary conditions on a bounded domain $U\subseteq X$. Then $P^D_{U,t}$ admits a Hilbert-Schmidt kernel $p^D_U(t,x,y)$ with respect to $\mu|_U$. Moreover, we can express
    \begin{align}
        \label{eigenexpansion}p^D_U(t,x,y)=\sum_{j=1}^{\infty}e^{-\lambda_j(U)t}\varphi_j^U(x)\varphi_j^U(y),
    \end{align}
    where $0<\lambda_1(U)<\lambda_2(U)\leq \cdots \leq \lambda_j(U)\to\infty$ and $\{\varphi_j^U\}_{j=1}^{\infty}$ is an orthonormal basis of $L^2(U)$. 
\end{Theo}

The Dirichlet heat kernel associated to $U$ satisfies domain monotonicity: for any open sets $U,V\subseteq X$ with $U\subseteq V$, 
\begin{align}
    \label{domainmono}
    p^D_U(t,x,y)\leq p^D_V(t,x,y)\text{ for all }t>0\text{ and }x,y\in U.
\end{align}
Let $(\mathcal{L}^D_U,\mathcal{D}(\mathcal{L}^D_U))$ be the associated infinitesimal generator, which is a nonpositive self-adjoint operator on $L^2(U,d\mu)$. Then $\{\lambda_j(U)\}_{j=1}^{\infty}$ are the eigenvalues of $-\mathcal{L}^D_U$ counted with multiplicity, and $\varphi_j^U$ are the corresponding Dirichlet eigenfunctions. We will often write $\lambda_U:=\lambda_1(U)$ for the smallest positive eigenvalue of $-\mathcal{L}$, which can also be expressed by the variational formula
\begin{align}
    \label{rayleigh}
    \lambda_U=\inf\Big\{\frac{\mathcal{E}(f,f)}{\|f\|^2_{L^2(U)}}:f\in \mathcal{D}(\mathcal{E})\cap C_c(U),f\neq 0\Big\}=\inf\Big\{\frac{\mathcal{E}(f,f)}{\|f\|^2_{L^2(U)}}:f\in \mathcal{D}(\mathcal{E}^D_U),f\neq 0\Big\}.
\end{align}
From (\ref{rayleigh}), we see that for open sets $U,V\subseteq X$ with $U\subseteq V$, the principal Dirichlet eigenvalue satisfies domain monotonicity: 
\begin{align}   
    \label{eigmono}
    \lambda_1(U)\geq \lambda_1(V).
\end{align}
In fact, we more generally have $\lambda_k(U)\geq \lambda_k(V)$ for all $k\geq 1$. The principal Dirichlet eigenfunction $\varphi^U_1$ can be characterized as the minimizer of (\ref{rayleigh}), and we will often write $\varphi_U$ in place of $\varphi^U_1$. By the parabolic Harnack inequality (e.g. Theorem 2.30 in \cite{gyryalsc}), the function $(t,x)\mapsto e^{-\lambda_U t}\varphi_U(x)$, a weak solution to the heat equation, has a continuous representative. Thus, we will always assume that $\varphi_U$ is continuous on $U$.

We state a lemma giving the scale of $\lambda_U$ when $U$ is inner uniform.

\begin{Lemma}
    \label{scaleeig}
    Let $(X,d,\mu,\mathcal{E},\mathcal{D}(\mathcal{E}))$ be a Dirichlet space satisfying both Assumptions \ref{assumption1} and \ref{assumption2}. 
    \begin{enumerate}
        \item Suppose $U\subseteq X$ is any $(C_0,c_0)$-inner uniform domain with $\textup{diam}_U\leq 16R/c_0$. Then 
        $$\lambda_U\leq \frac{1024\cdot D_0}{c_0^2\textup{diam}_U^2}.$$
        \item There exists $a=a(D_0,P_0)>0$ and $\varepsilon_0=\varepsilon_0(D_0,P_0)\in (0,1)$ such that if $U\subseteq X$ is any domain with $\textup{diam}_U\leq \varepsilon_0 R$ and the $1.01\textup{diam}_U$-neighborhood of $U$ does not cover $X$, then
        \begin{align}
            \label{eiglower}
            \lambda_U\geq \frac{a}{\textup{diam}^2_U}.
        \end{align}
        
    \end{enumerate}
\end{Lemma}

\begin{proof}
    We can express $U=B_U(x,\alpha \text{diam}_U)$ for some $\alpha\in [1/2,1)$. By Proposition \ref{ballinside}, there is $y\in U$ such that
    $$B(y,c_0\text{diam}_U/16)\subseteq U\subseteq B(y,\text{diam}_U).$$
    By (\ref{eigmono}), we then have
$$\lambda_{B(y,\text{diam}_U)}\leq \lambda_U\leq \lambda_{B(y,c_0\text{diam}_U/16)}.$$
The first statement of the lemma now follows from a test function argument. Indeed, for $B=B(y,c_0\text{diam}_U/16)$, the test function $\phi(x):= d(x,\partial B)$ has $d\Gamma(\phi,\phi)/d\mu\leq 1$ (e.g. Theorem 2.11 of \cite{gyryalsc}), and thus
\begin{align*}
    \lambda_{B(y,c_0\text{diam}_U/16)} \leq \frac{\int_B d\Gamma(\phi,\phi)}{\int _B d(x,\partial B)^2 d\mu(x)}\leq \frac{\mu(B)}{\int_{\frac{1}{2}B}d(x,\partial B)^2 d\mu (x)} & \leq \frac{\mu(B)}{(\frac{c_0\text{diam}_U}{32})^2\mu(\frac{1}{2}B)},
\end{align*}
and $\mu(B)\leq D_0\mu(\frac{1}{2}B)$ by volume doubling.
The second statement of the lemma follows from Theorem 2.5 of \cite{hebischlsc} (the hypothesis of the theorem is satisfied because of Assumption \ref{assumption2} together with Corollary 3.4 of \cite{hebischlsc}). The factor of $10$ in that theorem can be lowered to $1.01$ by following the logic in Lemma 5.2.8 of \cite{asti}. 
\end{proof}

\begin{Rmk}
    \label{spherer}
    \normalfont  
    
    Consider a Dirichlet space $X$ satisfying Assumption \ref{assumption2} for $R=\infty$ but is bounded. The prototypical example is given by $\mathbb{S}^n(r)\subseteq \mathbb{R}^{n+1}$, a $n$-dimensional sphere of radius $r>0$. Clearly, $\mathbb{S}^n(r)$ also satisfies Assumption \ref{assumption2} for any finite $R>0$. When applying Lemma \ref{scaleeig} to $\mathbb{S}^n(r)$, we treat $\mathbb{S}^n(r)$ as a $R$-doubling space not for $R=\infty$ but for some $R\lesssim_{D_0,P_0} r$. 
    
    When $X=\mathbb{R}^n$ and $\mathcal{L}$ is a second-order uniformly elliptic operator as in Example \ref{soue}, there is a simpler proof of Lemma \ref{scaleeig} using scaling properties of Dirichlet Laplacian eigenvalues. In this case the balls are defined with respect to the intrinsic distance (\ref{intrinsicdist}), and thus the volume doubling constant $D_0$ depends not only on $n$ but also on the ellipticity $\Lambda\geq 1$.   
\end{Rmk}

\subsection{The associated Hunt process} \label{huntprocess}

This subsection follows \cite{gyryalsc}. For each strictly local regular Dirichlet space $(X,d,\mu,\mathcal{E},\mathcal{D}(\mathcal{E}))$, there is an (essentially unique) associated symmetric Markov process on $X$ with continuous sample paths, called the \textit{Hunt process}. The correspondence is explained in detail in \cite{fukushima} and \cite{chenfuku}. We will use $\{X_t:t\geq 0\}$ to denote a generic Hunt process, and write $\mathbb{E}^x$ and $\mathbb{P}^x$ to denote the expectation and probability measure with respect to continuous paths starting at $x\in X$. The transition density of $\{X_t:t\geq 0\}$ is given by $p_X(t,x,y)$, which we can assume to be continuous due to the parabolic Harnack inequality of Theorem \ref{breakthrough}.

Given a open set $U\subseteq X$, we can consider the (sub)Markovian Dirichlet heat semigroup $(P^D_{U,t})_{t>0}$ on $L^2(U,\mu)$, as in Section \ref{heatkernelsect}. If $\tau_U=\inf\{t>0:X_t\notin U\}$ is the first exit time of the Hunt process from $U$, then for any bounded continuous function $\phi$,
\begin{align}
    \label{huntexp}
    P^D_{U,t}\phi(x)=\int_{U}p^D_U(t,x,y)\phi(y)d\mu(y)=\mathbb{E}^x[\phi(X_t)\mathbf{1}_{\{\tau_U>t\}}].
\end{align}

\subsection{The domain $V$ and the family of perturbations $V_c$}
\label{VVc}

In this section, we will first define a family of domains $\{V_c:c\geq 1\}$ representing an inner uniform expansion of an inner uniform domain $V$, then define a function $h_V(\delta)$ representing the measure of the $\delta$-tube around $\partial V$, which will be used in Theorem \ref{lowerbound}.

\begin{Dfn}
    \label{Vc}
    \normalfont
    Let $f,g:[1,\infty)\to[1,\infty)$ be two continuous functions with $f(1)=1=g(1)$. Let $V\subseteq X$ be a bounded $(C_0,c_0)$-inner uniform domain.  We write $\{V_c:c\geq 1\}$ to denote any one-parameter family of bounded $(C_0,c_0)$-inner uniform domains of $X$ such that for all $c\geq 1$,
    \begin{align*}
        V_c\supseteq V\text{, }\mu(V_c)=f(c)\mu(V)\text{, and }\text{diam}_{V_c}=g(c)\text{diam}_{V}.
    \end{align*}
\end{Dfn}

\begin{Exm}
\normalfont
    As in Example \ref{soue}, consider $X=\mathbb{R}^n$ with its canonical metric measure structure. We note that if $V\subseteq \mathbb{R}^n$ is $(C_0,c_0)$-inner uniform and $cV$ is a homothetic copy of $V$ dilated by $c\geq 1$, then $cV$ is also $(C_0,c_0)$-inner uniform. Thus, a simple example of Definition \ref{Vc} is obtained by letting $V$ be a star-shaped inner uniform domain with respect to a point, and letting $V_c:=cV$ be a dilation of $V$ about that point. In this case $f(c)=c^n$ and $g(c)\asymp_{n,c_0} c$. 
\end{Exm}

Let $V$ be a bounded domain. We define $h_V:[0,1]\to [0,1]$ by 
\begin{align}
    \label{hv}
    h_V(\delta)=\frac{\mu(\{x\in V:d(x,\partial V)\leq  \delta\text{diam}_V\})}{\mu(V)}.
\end{align}
The function $h_V$ gives the ratio of the volume of the $\delta$-tubular neighborhood of $\partial V$ inside $V$. Clearly, $h_V(\delta)\to 0$ as $\delta\to 0$.

\begin{Rmk}
    \label{33}
    \normalfont
    In Theorem \ref{lowerbound}, the lower bound $\varphi_U\gtrsim \varphi_V$ will have an implied constant depending on $h_V$. If we could show for all  $\varepsilon>0$, we have $h_V(\delta)<\varepsilon$ for $\delta>0$ depending only on $\varepsilon>0$, the inner uniformity constants $(C_0,c_0)$ of $V$, and the volume doubling constant $D_0$, then this would yield an improved estimate in Theorem \ref{lowerbound} with the implied constant independent of $h_V$. However, the decay of $h_V(\delta)$ as $\delta\to 0$ depends strongly on the boundary of the inner uniform domain.
    
    For example, consider $\mathbb{R}^n$ equipped with the Euclidean distance and Lebesgue measure. If $V\subseteq \mathbb{R}^n$ is a domain with piecewise smooth boundary, then $h_V(\delta)\asymp  \delta$ as $\delta\to 0$. 
    
    If $V\subseteq \mathbb{R}^2$ is the Von Koch snowflake, then $h_V(\delta)\lesssim \delta^{2-M}$ as $\delta\to 0$, where $M=\log 4/\log 3\approx 1.26$ is the \textit{Minkowski dimension} of the snowflake boundary $\partial V$. An explicit expression for $h_V(\delta)$ for the snowflake is calculated in \cite{lapiduspearse}.

    This shows that inner uniformity alone is insufficient for obtaining explicit upper bounds on $h_V(\delta)$ as $\delta\to 0$. For instance, a planar rectangle and the Von Koch snowflake in $\mathbb{R}^2$ are both inner uniform (for some common $(C_0,c_0)$) by Example \ref{iuexamples}, yet their $h_V(\delta)$ decay at different rates. 
\end{Rmk}

Nevertheless, we supplement Theorem \ref{lowerbound} with a lemma that controls the decay of $h_V(\delta)$ as $\delta\to 0$ uniformly for convex domains with bounded eccentricity (Definition \ref{boundedecc}). Lemma \ref{convexvolume} below is used to prove Theorem \ref{examplethm1} in Section \ref{examples}.


\begin{Lemma}
    \label{convexvolume}
    Consider $\mathbb{R}^n$ equipped with the Euclidean distance and the Lebesgue measure $\mu$. Suppose $V\subseteq \mathbb{R}^n$ is any convex domain that is contained in between two concentric balls in the sense that $B(x,a)\subseteq V\subseteq B(x,A)$. Then
    \begin{align}
        \label{convex}
        h_V(\delta)=\frac{\mu(\{x\in V:d(x,\partial V)\leq \delta\textup{diam}(V)\})}{\mu(V)}\leq 1-\Big(1-\frac{2A}{a}\delta\Big)^n,
    \end{align}
    for all $\delta\in (0,a/(2A))$.
\end{Lemma}

\begin{proof}
    Without loss of generality assume that $x$ is the origin and that $B(0,1)\subseteq V\subseteq B(0,A/a)$. For $\gamma\in [0,1]$, let $\gamma V$ denote the dilation of $V$ with respect to the origin. By convexity of $V$, we have $\gamma V\subseteq V$ for all $\gamma\in [0,1]$. Suppose we have already shown that for all $x\in \partial V$ and $\gamma\in [0,1]$
    \begin{align}
        \label{convex1}
        B(\gamma x,1-\gamma)\subseteq \overline{V}. 
    \end{align}
    Then (\ref{convex1}) implies that 
    \begin{align*}
        \text{dist}(\gamma x,\partial V)\geq 1-\gamma=(1-\gamma)\cdot \frac{\text{diam}_V}{\text{diam}_V}\geq \frac{1-\gamma}{(2A/a)}\cdot \text{diam}_V. 
    \end{align*}
    We thus get
    \begin{align*}
        \gamma V\subseteq\Big\{x\in V:d(x,\partial V)\geq \frac{1-\gamma}{2A/a}\text{diam}_V\Big\},
    \end{align*}
    and the desired result (\ref{convex}) follows by setting $\delta=(1-\gamma)/(2A/a)$ and using $\mu(\gamma V)/\mu(V)=\gamma^n$. It remains to prove (\ref{convex1}). Continue to assume that $B(0,1)\subseteq V\subseteq B(0,A/a)$ and define
    \begin{align*}
        N:\mathbb{R}^n\to [0,\infty),\hspace{0.2in} N(x):=\inf\{t>0:x/t\in V\}
    \end{align*}
    to be the \Quote{norm} on $\mathbb{R}^n$ induced by the convex set $V\subseteq \mathbb{R}^n$. While strictly speaking $N(\cdot )$ does not satisfy all of the axioms of a norm, the following properties are easily verified for all $\alpha>0$ and $x,y\in\mathbb{R}^n$:
    \begin{align*}
        & N(\alpha x)=\alpha N(x),\hspace{0.2in}  N(x+y)\leq N(x)+N(y),\hspace{0.2in}  N(x)\leq \|x\|, 
    \end{align*}
    where $\|x\|$ is the Euclidean norm on $\mathbb{R}^n$. Moreover, we can express $V=\{x\in\mathbb{R}^n:N(x)<1\}$ and $\gamma V=\{x\in \mathbb{R}^n:N(x)<\gamma\}$. Then if $\|y-\gamma x\|\leq 1-\gamma$ for $x\in \partial V$ and $\gamma\in [0,1]$, then
    \begin{align*}
        N(y)\leq N(y-\gamma x)+N(\gamma x) \leq 1-\gamma+\gamma N(x)\leq 1,
    \end{align*}
    implying $y\in \{x\in\mathbb{R}^n:N(x)\leq 1\}=\overline{V}$ and that (\ref{convex1}) holds.
\end{proof}

\section{Proofs of main results} \label{proofsofmainresults}

For obtaining both upper and lower bounds for $\varphi_U$, we make use of a result from \cite{lierllsc} that dominates the Dirichlet heat kernel $p^D_V(t,x,y)$ of an inner uniform domain $V\subseteq X$ by its leading term $e^{-\lambda_V t}\varphi_V(x)\varphi_V(y)$. Theorem \ref{jannalierl} below is a special case of Theorem 7.9 in \cite{lierllsc}, for a (symmetric) strictly local Dirichlet form, which we assume in Assumption \ref{assumption1}.

\begin{Theo}
    \label{jannalierl}
    Let $(X,d,\mu,\mathcal{E},\mathcal{D}(\mathcal{E}))$ be a Dirichlet space satisfying Assumptions \ref{assumption1} and \ref{assumption2}. Let $V\subseteq X$ be a bounded $(C_0,c_0)$-inner uniform domain with $\textup{diam}_V\leq  KR$ for some auxilliary constant $K\geq 1$. Then, for some constants $a,A>0$ depending only on $C_0,c_0,D_0,P_0$ and an upper bound on $K$, we have 
    $$ae^{-\lambda_V t}\varphi_V(x)\varphi_V(y)\leq p^D_V(t,x,y)\leq Ae^{-\lambda_V t}\varphi_V(x)\varphi_V(y),$$
    for all $t\geq C\textup{diam}_V^2$, where $C=C(C_0,c_0)>0$.
\end{Theo}

For obtaining a lower bound for $\varphi_U$, the following Carleson-type estimate from Proposition 5.12 of \cite{lierllsc} will also be useful. Informally, it states that for inner uniform domains $V$, the values of $\varphi_V$ on and near $\partial V$ are controlled by the values of $\varphi_V$ away from $\partial V$.

\begin{Lemma}
    \label{carleson}
    Let $(X,d,\mu,\mathcal{E},\mathcal{D}(\mathcal{E}))$ be a Dirichlet space satisfying Assumptions \ref{assumption1} and \ref{assumption2}. There are constants $C_1=C_1(C_0,c_0)\geq 1$ and $C_2=C_2(C_0,c_0)\in (0,1)$ such that the following holds. For any bounded $(C_0,c_0)$-inner uniform domain $V\subseteq X$ with $\textup{diam}_V\leq R/C_1$, for any $\xi\in \widetilde{V}$, and for any $r\in (0,C_2\textup{diam}_V)$, we have
    \begin{align*}
        \varphi_V(y)\leq A\varphi_V(x_r),\text{ for all }y\in B_{\widetilde{V}}(\xi,r),
    \end{align*}
    where $x_r\in V$ is any point with $d_V(\xi,x_r)=r/4$ and $d(x_r,\partial V)\geq c_0r/8$. Here $A\geq 1$ depends only on $D_0,P_0,C_0,c_0$. 
\end{Lemma}

\subsection{Upper bound for $\varphi_U$}

We begin with bounding the $L^{\infty}$-norm of $\varphi_U$ from above and below.

\begin{Lemma}
    \label{maxphi}
    Let $(X,d,\mu,\mathcal{E},\mathcal{D}(\mathcal{E}))$ be a Dirichlet space satisfying both Assumptions \ref{assumption1} and \ref{assumption2}. Suppose $V\subseteq X$ is a bounded $(C_0,c_0)$-inner uniform domain with $\textup{diam}_V\leq R$. If $U\subseteq X$ is any bounded domain with $U\supseteq V$ and $\lambda_U\geq R^{-2}$, then
        $$\frac{1}{\mu(U)}\leq \|\varphi_U\|^2_{L^{\infty}(U)}\lesssim_{D_0,P_0,c_0} \frac{1}{\mu(V)}.$$
\end{Lemma}

\begin{proof}
The lower bound follows immediately from $$1=\|\varphi_U\|^2_{L^2(U)}\leq \mu(U)\cdot \|\varphi_U\|^2_{L^{\infty}(U)}.$$Now we prove the upper bound on $\|\varphi_U\|^2_{L^{\infty}(U)}$.
    For all $t>0$ and $x\in U$, (\ref{eigenexpansion}) and (\ref{domainmono}) imply that
    $$e^{-\lambda_U t}\varphi_U^2(x) \leq p^D_U(t,x,x)\leq p_{X}(t,x,x).$$
    Thus, by Theorem \ref{breakthrough}, we have $\varphi_U^2(x)\lesssim_{D_0,P_0} e^{\lambda_Ut}/\mu(B(x,\sqrt{t}))$. By choosing $t=1/\lambda_U\leq R^2$ and using the domain monotonicity of the Dirichlet eigenvalue,  we get
    $$\varphi_U^2(x)\lesssim\frac{1}{\mu(B(x,\lambda_U^{-1/2}))}\leq \frac{1}{\mu(B(x,\lambda_V^{-1/2}))}.$$
    Appealing to Lemma \ref{scaleeig} and Proposition \ref{volumeratioballs}, we get that for some constant $A=A(D_0)>0$,
    \begin{align*}
        \varphi_U^2(x) & \lesssim  \frac{1}{\mu(B(x,c_0A^{-1/2}\text{diam}_V/16))}
        \\ & \leq \frac{\mu(B(x,\text{diam}_V))}{\mu(B(x,c_0A^{-1/2}\text{diam}_V/16))}\cdot \frac{1}{\mu(V)} \lesssim \frac{1}{\mu(V)},
    \end{align*}
    which proves the upper bound. 
\end{proof}

\begin{Theo}
    \label{ub}
    Suppose $V\subseteq X$ is a bounded $(C_0,c_0)$-inner uniform domain with $\textup{diam}_V\leq 16R/c_0$. Let $\{V_c:c\geq 1\}$ be a family of bounded $(C_0,c_0)$-inner uniform domains all containing $V$, as in Definition \ref{Vc}. 
    
    \begin{enumerate}
        \item If $U\subseteq X$ is any arbitrary domain with $V\subseteq U\subseteq V_c$, then for all $x\in U$
    $$\varphi_U^2(x) \leq A_1\varphi_{V_c}^2(x).$$
        \item If $U\subseteq X$ is a $(C_0,c_0)$-inner uniform domain with $V\subseteq U\subseteq V_c$, then for all $x\in V
        $
        $$\varphi_U^2(x)\geq \frac{1}{A_2}\varphi_V^2(x).$$
    \end{enumerate}
    Here $A_1,A_2>1$ are constants depending only on $D_0,P_0,C_0,c_0$ and an upper bound on $g(c)^2$. 
    
    \end{Theo}

\begin{proof}
    Note that if $\text{diam}_V\leq 16R/c_0$, then $\text{diam}^2_{V_c}\leq 16g(c)^2R/c_0$. Therefore, applying Theorem \ref{jannalierl} with $V_c$ in place of $V$, $K=16g(c)^2/c_0>1$, and $t=C\text{diam}_V^2$ gives
    $$p^D_{V_c}(t,x,y)\leq A' \varphi_{V_c}(x)\varphi_{V_c}(y),$$
    where $A'>0$ is a constant that depends only on $D_0,P_0,C_0,c_0$ and an upper bound on $g(c)^2$.
    It follows from (\ref{eigenexpansion}) and (\ref{domainmono}) that for $t=C\text{diam}_V^2$, 
    \begin{align}   
        \label{uvc}
        \varphi_U^2(x)\leq e^{\lambda_U t}p^D_U(t,x,x)\leq A' e^{\lambda_U t}\varphi^2_{V_c}(x).
    \end{align}
    Due to Lemma \ref{scaleeig},
    $\lambda_U\cdot \text{diam}^2_{V_c}\leq \lambda_V \cdot g(c)^2 \text{diam}^2_V\lesssim_{D_0,c_0} g(c)^2,$
    so the first statement of the theorem follows. The proof of the second statement is by replacing the roles of $U$ and $V_c$ in (\ref{uvc}) by $V$ and $U$, and by similar reasoning.
\end{proof}

\subsection{Lower bound for $\varphi_U$ when $c\geq 1$ is close to $1$}

For $x\in V$, we now begin to derive a lower bound for $\varphi_U(x)$ in terms of $\varphi_V(x)$. 

\begin{Lemma}
    \label{preliminaryLB}
    Suppose $V\subseteq X$ is a bounded $(C_0,c_0)$-inner uniform domain with $\textup{diam}_V\leq 16R/c_0$. For any domain $U\subseteq X$ with $U\supseteq V$ and $x\in V$,
    \begin{align*}
        \varphi_U(x)& \gtrsim_{D_0,P_0,C_0,c_0} \Big(\int_V\varphi_V\cdot \varphi_U d\mu\Big)\varphi_V(x).
    \end{align*}
\end{Lemma}

\begin{proof}
    It follows from (\ref{eigenexpansion}) that for all $x\in U$ and $t>0$,
    $$\int_U p_U^D(t,x,y)\varphi_U(y)d\mu(y)=e^{-\lambda_U t}\varphi_U(x)\leq \varphi_U(x).$$
Thus, applying Theorem \ref{jannalierl} with  $t= C\text{diam}^2_V$ and with $C=C(C_0,c_0)>0$, we get
\begin{align*}
    \varphi_U(x) & \geq \int_V p_V^D(t,x,y) \varphi_U(y)d\mu(y)
    \\ & \geq \int_V ae^{-\lambda_V t}\varphi_V(x)\varphi_V(y)\varphi_U(y)d\mu(y),
\end{align*}
where $a$ depends only on $D_0,P_0,C_0,c_0$. 
We know from Lemma \ref{scaleeig} that
$-\lambda_V\text{diam}_V^2\geq -1024D_0/c_0^2$, which completes the argument.
\end{proof}

For a domain $V\subseteq X$ and $\delta>0$, we define the set $$V(\delta):=\{x\in V:d(x,\partial V)\geq 2\delta \text{diam}_V\}.$$

\begin{Lemma}
    \label{pathawayfromboundary}
    Let $V\subseteq X$ be a $(C_0,c_0)$-inner uniform domain. Suppose $\delta>0$ is chosen so that $V(\delta)\cap \partial V=\varnothing$. Then, for any $x,y\in V(\delta)$, a continuous path $\gamma:[0,1]\to V$ connecting $x$ and $y$ from the definition of inner uniformity has  $d(z,\partial V)\geq c_0\delta\textup{diam}_V/2$ for all $z\in \gamma([0,1])$. 
\end{Lemma}

\begin{proof}
    Let $z\in \gamma([0,1])$ and without loss of generality assume $d(x,z)\leq d(y,z)$. Then from (\ref{iu2}), we have $d(z,\partial V)\geq c_0d(x,z)/2$. On the other hand, 
    $$d(z,\partial V)\geq d(x,\partial V)-d(x,z)\geq \frac{c_0}{2}(d(x,\partial V)-d(x,z)).$$
    Adding these inequalities, we get $2d(z,\partial V)\geq c_0d(x,\partial V)/2\geq c_0\cdot 2\delta\text{diam}_V/2$.
\end{proof}

The next lemma uses the parabolic Harnack inequality to prove that for an inner uniform domain $V$, the values of $\varphi_V$ at scale $\delta \cdot\text{diam}_V$ away from the boundary $\partial V$ are all comparable. 

\begin{Lemma}
    \label{phiharnack}
    Let $V\subseteq X$ be a bounded $(C_0,c_0)$-inner uniform domain with $\textup{diam}_V\leq 2R$. Suppose $\delta>0$ is chosen so that $V(\delta)\cap \partial V=\varnothing $. Then for all $x,y\in V(\delta)$, we have
    $$\varphi_V(x)\leq H^{1/\delta}\varphi_V(y),$$
    where $H=H(D_0,P_0,C_0,c_0)>1$.
    
\end{Lemma}

\begin{proof}
    Since $V(\delta)\cap \partial V=\varnothing$, we can assume that $\delta \in (0,1/2)$. For any $z\in V(\delta)$, we have $B(z,\delta\text{diam}_V)\subsetneq  V$. Define 
    $$v:(0,\delta^2\text{diam}^2_V)\times B(z,\delta\text{diam}_V)\to \mathbb{R},\enspace v(t,x)=e^{-\lambda_V t}\varphi_V(x).$$
    By Lemma 6.4.5 of \cite{fukushima}, $v(t,x)> 0$. Moreover, $v$ is a weak solution of the heat equation $(\partial_t -\mathcal{L})v=0$ on its domain. The parabolic Harnack inequality then implies that as long $r:=\delta \text{diam}_V\leq R$,
    \begin{align}
        \label{ballphi}
        e^{-\lambda_V r^2}\varphi_V(x)\leq H_0 e^{-\lambda_V r^2/4}\varphi_V(y)\text{ for all }x,y\in B(z,r/2).
    \end{align}
    By Lemma \ref{scaleeig}, we get that for some constant $H_0'=H_0'(D_0,P_0,c_0)>1$,
    \begin{align}
        \label{ballharnack}
        \varphi_V(x)\leq H'_0 \varphi_V(y)\text{ for all }x,y\in B(z,r/2).
    \end{align}
    Now let $x,y\in V(\delta)$ be arbitrary with $y\notin B(x,r/2)$. Connect $x$ to $y$ by a continuous path $\gamma:[0,1]\to V$ guaranteed by the definition of inner uniformity. We can construct a sequence of overlapping balls $B(\gamma(t_i),r_i)$ for $i=1,2,...,N$ with the number of balls $N\leq \lceil 10C_0/(c_0\delta)\rceil $ and with the following properties:
    \begin{align*}
    \begin{cases}
          0=t_1<t_2<\cdots<t_N=1,
         \\ r_1,r_N=r/2\text{ and }r_i=c_0r/4\text{ for }i\neq 1,N.
    \end{cases}
    \end{align*}
    It follows from Lemma \ref{pathawayfromboundary} that $d(\gamma(t_i),\partial U)\geq c_0\delta/2$. Thus, applying (\ref{ballharnack}) with $B(z,r/2)$ replaced by $B(\gamma(t_i),r_i)$ and iterating this estimate along $\gamma$, we get
    $\varphi_V(x)\leq (H_0')^{\lceil 10C_0/(c_0\delta)\rceil}\varphi_V(y)$.
\end{proof}

With the help of Lemmas \ref{preliminaryLB}, \ref{pathawayfromboundary}, and \ref{phiharnack}, we are now in a position to prove the lower bound $\varphi_U\gtrsim \varphi_V$.

\begin{Theo}
    \label{lowerbound}
    Let $(X,d,\mu,\mathcal{E},\mathcal{D}(\mathcal{E}))$ be a Dirichlet space satisfying Assumptions \ref{assumption1} and \ref{assumption2}. There is a constant $C=C(D_0,P_0,C_0,c_0)\geq 1$ such that the following holds. Let $V\subseteq X$ be a $(C_0,c_0)$-inner uniform domain with $\textup{diam}_{V}\leq R/C$. Let $\{V_c:c\geq 1\}$ be a family of bounded $(C_0,c_0)$-inner uniform domains all containing $V$, as in Definition \ref{Vc}. There exists $c=c(f,g,D_0,P_0,C_0,c_0)>1$ sufficiently close to $1$ such that as long as the $2\textup{diam}_V$-neighborhood of $V_c$ does not cover $X$, we have
    $$\varphi_U(x)\gtrsim\varphi_V(x),$$
    for any arbitrary domain $U\subseteq X$ with $V\subseteq U\subseteq V_c$. The implied constant depends only on $D_0,P_0,C_0,c_0,$ and an upper bound on $h_V(\delta)$ as $\delta\downarrow 0$ (see Section \ref{VVc}, Remark \ref{33}).
\end{Theo}

\begin{proof}
    We begin the proof by upper bounding the inner diameter of $V_c$. Let $a>0$ and $\varepsilon_0>0$ be the constants depending only on $D_0,P_0$ from Lemma \ref{scaleeig}. We take $c=c(g)>1$ sufficiently close to $1$ so that $g(c)\leq 1.5$. This will ensure that as long as  
    \begin{align}
        \label{diamVbound}
            \text{diam}_V\leq (\sqrt{a}\wedge \varepsilon_0\wedge 1)R/1.5,\hspace{0.2in} \{x\in X:d(x,V_c)<2\text{diam}_V\}\neq X, 
    \end{align}

     the second statement of Lemma \ref{scaleeig} holds for the domain $V_c$, so that the hypothesis $\lambda_U\geq R^{-2}$ of Lemma \ref{maxphi} is satisfied. In fact, the upper bound on $\text{diam}_V$ in (\ref{diamVbound}) will also justify application of Lemma \ref{preliminaryLB}, Lemma  \ref{phiharnack}, and Theorem \ref{ub} in the calculation to follow below.    
  
    Let $\delta>0$ be sufficiently small enough as in Lemma \ref{phiharnack}. We have
    $$\int_{V(\delta)}\varphi^2_V d\mu \leq \mu(V)\cdot H^{2/\delta}\inf_{V(\delta)}\varphi_V^2.$$
    It then follows from Lemma \ref{preliminaryLB} and Lemma \ref{maxphi} that for all $x\in V$, 
    \begin{align}
        \nonumber
        \frac{\varphi_U(x)}{\varphi_V(x)}&\gtrsim_{D_0,P_0,C_0,c_0} 
        \inf_{V(\delta)}\varphi_V\int_{V(\delta)}\varphi_U d\mu
        \\ \nonumber & \geq \frac{\|\varphi_V\|_{L^2(V(\delta))}}{H^{1/\delta }\mu(V)^{1/2}}\int_{V(\delta)}\varphi_U d\mu 
        \\ \nonumber & \geq \frac{\|\varphi_V\|_{L^2(V(\delta))}}{H^{1/\delta }\mu(V)^{1/2}}\cdot \frac{\|\varphi_U\|^2_{L^2(V(\delta))}}{\|\varphi_U\|_{L^{\infty}(U)}}
        \\ \label{L2lowerbound}& \gtrsim_{D_0,P_0,c_0} \frac{\|\varphi_V\|_{L^2(V(\delta) )}\|\varphi_U\|^2_{L^2(V(\delta))}}{H^{1/\delta}}
    \end{align}
    We will now argue that for an appropriate choice of $\delta>0$, the principal eigenfunction $\varphi_U$ has $L^2$-norm concentrated on $V(\delta)$. We recall that $\varphi_U$ is $L^2(U)$-normalized. Extending $\varphi_U\equiv 0$ outside $U$ and using Theorem \ref{ub}, we get
    $$\int_V \varphi_U^2 d\mu=1-\int_{V_c\backslash V}\varphi_U^2 d\mu\geq 1-A\int_{V_c\backslash V} \varphi^2_{V_c}d\mu ,$$
    so that applying Lemma \ref{maxphi} to $\varphi_{V_c}$, 
    \begin{align*}
        \int_V \varphi_U^2 d\mu & \geq 1-A\cdot C(D_0,P_0,c_0)\frac{\mu(V_c\backslash V)}{\mu(V)}
        \\ & =1-A'(f(c)-1).
    \end{align*}
    Because $A'>1$ depends only on $D_0,P_0,C_0,c_0$ and an upper bound on $g(c)^2(\leq 9)$, and because $\lim_{c\to 1}f(c)=1$, we get for any $c=c(D_0,P_0,C_0,c_0,f)>1$ sufficiently close to $1$, 
    $$\int_V \varphi_U^2 d\mu \geq \frac{1}{2}.$$
    It follows that 
    \begin{align*}
        \int_{V(\delta)}\varphi_U^2 d\mu  & =\int_V \varphi_U^2 d\mu - \int_V \varphi_U^2(1-\mathbf{1}_{V(\delta)}) d\mu  
        \\ & \geq \frac{1}{2} -\frac{C(D_0,P_0,c_0)}{\mu(V)}(\mu(V)-\mu(V(\delta))).
        \\ & = \frac{1}{2}-C(D_0,P_0,c_0)h_V(2\delta),
    \end{align*}
    where $h_V$ is defined in (\ref{hv}). Choose $\delta=\delta(D_0,P_0,C_0,c_0,h_V)>0$ sufficiently small enough so that $h_V(2\delta)\leq 1/(4C(D_0,P_0,c_0))$. Then  $\int_{V(\delta)}\varphi_U^2 d\mu \geq 1 /4$. 
    
    The proof will be finished if $\|\varphi_V\|_{L^2(V(\delta))}$ is bounded below. If necessary assume that $\text{diam}_V$ is an even smaller scalar multiple of $R$ to satisfy the hypothesis of Lemma \ref{carleson}. Let $r=C_2\text{diam}_V/2$ as in Lemma \ref{carleson} and replace the $\delta>0$ previously chosen with $\delta\wedge (c_0C_2/32)$. This choice of $\delta>0$ ensures that any point $x_r$ from Lemma \ref{carleson} is properly contained in $V(\delta)$. For this choice of $\delta>0$,
    $$1=\int_V\varphi_V^2 d\mu \leq \mu(V)\sup_V \varphi_V^2\leq \mu(V)\cdot A^2\sup_{V(\delta)}\varphi_V^2,$$
    where $A\geq 1$ is now the constant from Lemma \ref{carleson}. This implies 
    \begin{align}
        \nonumber \int_{V(\delta)}\varphi_V^2 d\mu  & \geq \mu(V(\delta))\inf_{V(\delta)}\varphi_V^2 
        \\ \nonumber &\geq \mu(V(\delta))\frac{1}{H^{1/\delta}}\sup_{V(\delta)}\varphi_V^2
        \\ \label{thm1-1}& \geq \frac{\mu(V(\delta))}{A^2H^{1/\delta}\mu(V)}.
    \end{align}
    The left-hand side is bounded below if we choose $\delta>0$ even smaller so that $V(\delta)$ contains the ball from Proposition \ref{ballinside}, and then apply Proposition \ref{volumeratioballs}. This finishes the proof.
\end{proof}

\subsection{Lower bound for $\varphi_U$ when $f(c)$ is bounded above}
The next theorem is quite general and gives a lower bound of the form $\varphi_U\gtrsim \varphi_V$ (without requiring $c$ to be sufficiently close to $1$ in $V\subseteq U\subseteq V_c$). It instead only requires that some point $x_U\in U$ where $\varphi_U$ is not too small compared to $\mu(U)^{-1/2}$ connects to $V$ via a bounded number of overlapping balls. This subsection can be skipped on first reading.

\begin{Theo} 
\label{lowerbound2} Let $(X,d,\mu,\mathcal{E},\mathcal{D}(\mathcal{E}))$ be a Dirichlet space satisfying Assumptions \ref{assumption1} and \ref{assumption2}. Let $V\subseteq X$ be a bounded $(C_0,c_0)$-inner uniform domain and let $\{V_c:c\geq 1\}$ be a family of bounded $(C_0,c_0)$-inner uniform domains all containing $V$, as in Definition \ref{Vc}. Assume $g_0>1$ is an upper bound for $g(c)$ and assume the $1.01g_0\textup{diam}_V$-neighborhood of $V_c$ does not cover $X$.
    
    Let $U\subseteq X$ be a domain with $V\subseteq U\subseteq V_c$. Suppose there is a point $x_U \in U$ with $\varphi_U(x_U)\geq \beta/\sqrt{\mu(U)}$ for some $\beta>0$. Suppose the domains $U,V$ satisfy the following property: there exists $\widetilde{C}\geq 1,\widetilde{\delta}\in (0,1)$ and a chain of at most $N\leq K/\widetilde{\delta}$ overlapping balls $\{B_i\}_i$, such that
    \begin{enumerate}
        \item Each $B_i$ has radius $\widetilde{\delta}\textup{diam}_U$ and are precompact subsets of $U$. The first ball $B_1$ is centered at $x_U$ and the last ball $B_N$ has a center in $V$.
        \item $ \textup{diam}_U\leq \widetilde{C}\textup{diam}_V$.
    \end{enumerate}
    There exists a constant $C=C(D_0,P_0,g_0,\widetilde{C})>0$ such that whenever $\textup{diam}_V\leq CR$,
    \begin{align*}
        \varphi_U(x)\gtrsim \frac{\beta^2}{f(c)} \varphi_V(x),
    \end{align*}
    for all $x\in V$. The implied constant depends only on $D_0,P_0,C_0,c_0,\widetilde{\delta},\widetilde{C},K$.
\end{Theo}

\begin{Rmk}
    \normalfont
    The constant $C$ in the above theorem is of the form $C'/(g_0\vee\widetilde{C})$ where $C'$ depends only on $D_0,P_0$. If $X$ is unbounded and $R=g_0=\infty$, then we do not require $\text{diam}_V\leq CR$ in Theorem \ref{lowerbound2} (and hence in Corollary \ref{corollary1} below).
\end{Rmk}

\begin{proof}
    Let $a>0$ and $\varepsilon_0>0$ be the constants depending only on $D_0,P_0$ from Lemma \ref{scaleeig}. As long as the $1.01g_0\text{diam}_V$-neighborhood of $V_c$ does not cover $X$ and that $\text{diam}_V\leq (\sqrt{a}\wedge \varepsilon_0\wedge 1)R/g_0$, we have from (\ref{L2lowerbound}) that
    \begin{align}
        \label{L2lowerbound2}
        \frac{\varphi_U(x)}{\varphi_V(x)}\gtrsim_{D_0,c_0}\frac{\|\varphi_V\|_{L^2(V(\delta))}\|\varphi_U\|^2_{L^2(V(\delta))}}{H^{1/\delta}},
    \end{align}
    for all sufficiently small $\delta>0$. We now aim to show that for an appropriate $\delta>0$, $\varphi_U\gtrsim\varphi_V$ on $V(\delta)$. By the hypothesis and Lemma \ref{maxphi} applied to $V$, we obtain
    \begin{align*}
        &  \varphi^2_U(x_U)\geq \frac{\beta^2}{\mu(V_c)}=\frac{\beta^2}{f(c)\mu(V)},
        \\ &\|\varphi_V\|^2_{L^{\infty}(V)}\lesssim_{D_0,P_0,c_0} \frac{1}{\mu(V)}. 
    \end{align*}
    It follows that for all $x\in V$,
    \begin{align}
        \label{thm2-1}\varphi^2_V(x)\lesssim_{D_0,P_0,c_0}\frac{f(c)}{\beta^2}\varphi_U^2(x_U).
    \end{align}
    Let $\{B_i\}_i$ be a chain of radius $\widetilde{\delta}\text{diam}_U$-balls as in the hypothesis. On each ball $B_i$, applying parabolic Harnack inequality to $(t,x)\mapsto e^{-\lambda_U t}\varphi_U(x)$ gives
    \begin{align*}
        \varphi_U(x)\leq H_0 \exp\Big(3\lambda_U (\widetilde{\delta}\text{diam}_U)^2/4\Big)\varphi_U(y)\lesssim_{D_0,c_0,\widetilde{C},\widetilde{\delta}} H_0\varphi_U(y)\lesssim A\varphi_U(y),
    \end{align*}
    for all $x,y\in B_i$. Here $A\geq 1$ is some constant depending only on $D_0,P_0,c_0,\widetilde{C}$, and $\widetilde{\delta}$. Iterating this estimate along the chain of overlapping balls $\{B_i\}$, we get
    \begin{align}
    \label{thm2-3}
    \varphi_U(x_U)\leq A^{K/\widetilde{\delta}}\varphi_U(y), \text{ for all }y\in B_N.
    \end{align}
    We now choose a $y_0\in B_N$ separated away from the the boundary of $V$. Proposition \ref{ballinside} applied to $B_N$ implies there is $y_0\in B_N\cap V$ with $d(y_0,\partial V)\geq c_0\widetilde{\delta}\text{diam}_U/8\geq c_0^2\widetilde{\delta}\text{diam}_V/16$. Applying Proposition \ref{phiharnack} with the domain $V$ but with $\varphi_V$ replaced by $\varphi_U$ gives 
    \begin{align}
        \label{thm2-4}
        \varphi_U(y_0)\leq H^{1/\delta}\varphi_U(x),\text{ for all }x\in V(\delta):=\{x\in V:d(x,\partial V)\geq 2\delta\text{diam}_V\},
    \end{align}
    as long as $\delta=c_0^2\widetilde{\delta}/32>0$. Combining (\ref{thm2-1}), (\ref{thm2-3}), (\ref{thm2-4}), we get for all $x\in V(\delta)$,
    \begin{align}
        \label{thm2-5}
        \varphi^2_V(x)&\lesssim_{D_0,P_0,c_0}\frac{f(c)}{\beta^2}A^{2K/\widetilde{\delta}}H^{2/\delta}\varphi^2_U(x).
    \end{align}
    Hence, from (\ref{L2lowerbound2}),
    \begin{align*}
        \frac{\varphi_U(x)}{\varphi_V(x)}\gtrsim_{D_0,P_0,c_0}\frac{\beta^2}{H^{3/\delta}f(c)A^{2K/\widetilde{\delta}}}\Big(\int_{V(\delta)}\varphi^2_V d\mu \Big)^{3/2}.
    \end{align*}
    For $\delta=c_0^2\widetilde{\delta}/32>0$, the ball from Proposition \ref{ballinside} is contained inside $V(\delta)$. Thus, reasoning as in (\ref{thm1-1}), $(\int_{V(\delta)}\varphi_V^2 d\mu)^{3/2}$ is bounded below by a positive constant depending only on $D_0,P_0,C_0,c_0$. The conclusion follows. 
\end{proof}

In Theorem \ref{lowerbound2}, by considering the case when $x_U\in V$, we get the following corollary. Informally, Corollary \ref{corollary1} (and also Theorem \ref{lowerbound3} below) states that an inequality of the form $\varphi_U\gtrsim \varphi_V$ holds on $V$ as long as $V$ contains a point where $\varphi_U$ is not too small.

\begin{Cor}
    \label{corollary1}
    Let $(X,d,\mu,\mathcal{E},\mathcal{D}(\mathcal{E}))$ be a Dirichlet space satisfying Assumptions
    \ref{assumption1} and \ref{assumption2}. Let $V\subseteq X$ be a bounded $(C_0,c_0)$-inner uniform domain and let $\{V_c:c\geq 1\}$ be a family of bounded $(C_0,c_0)$-inner uniform domains all containing $V$, as in Definition \ref{Vc}. Assume $g_0\geq 1$ is an upper bound for $g(c)$ and assume the $1.01g_0\textup{diam}_V$-neighborhood of $V_c$ does not cover $X$. 

     Let $U\subseteq X$ be a domain with $V\subseteq U\subseteq V_c$. Suppose there is a point $x_U\in V$ with $d(x_U,\partial U)\geq \widetilde{\delta}\textup{diam}_U$ and $\varphi_U(x_U)\geq \beta/\sqrt{\mu(U)}$ for some $\widetilde{\delta},\beta>0$.
    Also suppose there exists $\widetilde{C}>1$ such that
    $\textup{diam}_U\leq \widetilde{C}\textup{diam}_V$. Then there exists a constant $C=C(D_0,P_0,g_0,\widetilde{C})>0$ such that whenever $\textup{diam}_V\leq CR$,
    \begin{align}
        \label{cor1} \varphi_U(x)\gtrsim\frac{\beta^2}{f(c)}\varphi_V(x).
    \end{align}
     The implied constant in (\ref{cor1}) depends only on $D_0,P_0,C_0,c_0,\widetilde{C},\widetilde{\delta}$. In particular, when $U$ satisfies an $\alpha$-exterior ball condition and that $\varphi_U(x_U)\geq (1-\varepsilon)\|\varphi_U\|_{L^{\infty}(U)}$ for $\varepsilon=\varepsilon(D_0,P_0,\alpha)\in (0,1)$ as in Theorem \ref{separate2}, then 
     \begin{align}
         \label{cor2}\varphi_U(x)\gtrsim \frac{1}{f(c)}\varphi_V(x),
     \end{align}
     with the implied constant in (\ref{cor2}) depending only on $D_0,P_0,C_0,c_0,\widetilde{C},\alpha$.
\end{Cor}

\begin{proof}
    The first statement (\ref{cor1}) follows from Theorem \ref{lowerbound2} by replacing $\widetilde{\delta}>0$ with $\widetilde{\delta}/2>0$, taking the chain of balls to be the single ball $B(x_U,\widetilde{\delta}\text{diam}_U/2)$, and taking $K=\widetilde{\delta}/2$. For the second statement,
    by Theorem \ref{separate2} and Lemma \ref{scaleeig}, we have
    \begin{align}
        \label{0.1}
        d(x_U,\partial U)\gtrsim_{D_0,P_0,\alpha}\frac{1}{\sqrt{\lambda_U}}\geq \frac{1}{\sqrt{\lambda_V}}\gtrsim_{D_0,c_0}\text{diam}_V\geq \frac{1}{\widetilde{C}}\text{diam}_U.
    \end{align}
    So $d(x_U,\partial U)\geq \gamma \text{diam}_U$ for some constant $\gamma>0$ depending only on $D_0,P_0,c_0,\alpha,\widetilde{C}$. Then (\ref{cor2}) follows from (\ref{cor1}) by taking $\widetilde{\delta}=\gamma/2$ and $\beta=1-\varepsilon$. 
\end{proof}

\subsection{Proof sketch of results from Section \ref{examples}}
\label{proofsketch}

\textit{Proof sketch of Theorem \ref{triangleprofile} (planar and spherical triangles).} In this proof sketch, all implied constants in the symbol $\asymp$ will only depend on the angle lower bound $\alpha>0$.
The class of (planar or spherical) triangles with angles bounded below is $(C_0,c_0)$-inner uniform with the constants of inner uniformity depending only on $\alpha>0$. We first explain why Theorem \ref{triangleprofile} holds for planar triangles, then explain how the proof strategy is similar for spherical triangles.  

Let $T\subseteq \mathbb{R}^2$ be a triangle with angles bounded below by $\alpha>0$. Proposition 5.12 of \cite{lierllsc} implies the following: there exists a small constant $c_{\alpha}>0$ such that whenever $c\in (0,c_{\alpha})$, $z\in \partial T$ and $x,y\in \overline{B(z,c\cdot \text{diam}(T))\cap T}$,
\begin{align}
    \label{phih}
    \frac{\varphi_T(x)}{\varphi_T(y)}\asymp \frac{h(x)}{h(y)}, 
\end{align}
where $h$ is any nonnegative harmonic function vanishing on $ \overline{B(z,c\cdot \text{diam}(T))}\cap \partial T$. Let $z=z_1$ be the vertex corresponding to angle $\alpha_1$. Applying (\ref{phih}) with the harmonic function $h=h_1(r,\theta)=r^{\pi/\alpha_1}\sin(\pi\theta/\alpha_1)$ (in polar coordinates with origin at $z_1$) and choosing an appropriate $y$ gives 
\begin{align}
    \label{triangle1}
    \varphi_T\asymp \frac{h}{\text{diam}(T)^{\pi/\alpha_1+1}}
\end{align}
on the closure of the small ball $\overline{B(z_1,c\cdot \text{diam}(T))\cap T}$. Moreover, on the same small ball,
\begin{align}
    \label{triangle2}
    h\asymp d_2d_3(d_2+d_3)^{\pi/\alpha_1-2}, \hspace{0.2in} \frac{d_1}{\text{diam}(T)}\asymp \frac{(d_3+d_1)^{\pi/\alpha_2-2}}{\text{diam}(T)^{\pi/\alpha_2-2}}\asymp \frac{(d_1+d_2)^{\pi/\alpha_3-2}}{\text{diam}(T)^{\pi/\alpha_3-2}}\asymp 1.
\end{align}
Combining (\ref{triangle1}) with (\ref{triangle2}) shows that the desired result (\ref{triangleprofile1}) holds on $\overline{B(z_1,c\cdot \text{diam}(T))\cap T}$. By the same argument, the desired result (\ref{triangleprofile1}) is true on all balls centered on $\partial T$ (if $z_1$ is replaced by a non-vertex point on $\partial T$, then the angle $\alpha_1$ is replaced by $\pi$). A covering argument implies (\ref{triangleprofile1}) holds on $\{x\in T:d(x,\partial T)\leq c\cdot \text{diam}(T)\}$. 

On the complement $\{x\in T:d(x,\partial T)>c\cdot \text{diam}(T)\}$, we use Lemma \ref{maxphi}, the parabolic Harnack inequality (Lemma \ref{phiharnack}), and Theorem \ref{separate2} to conclude that $\varphi_T\asymp 1/\text{diam}(T)$, while
$$d_1d_2d_3(d_1+d_2)^{\pi/\alpha_3-2}(d_2+d_3)^{\pi/\alpha_1-2}(d_3+d_1)^{\pi/\alpha_2-2}\asymp \text{diam}(T)^{\pi/\alpha_1+\pi/\alpha_2+\pi/\alpha_3-3}.$$
The proof is thus finished for planar triangles. 

For spherical triangles, the only difference is that in (\ref{triangle1}), instead of comparing ratios of $\varphi_T$ and $h$, we compare ratios of $\varphi_T$ and another principal Dirichlet eigenfunction $\varphi_W$. For illustration, endow $\mathbb{S}^2$ with spherical coordinates $\{(\theta,\psi):\theta\in [0,2\pi),\psi\in [0,\pi]\}$, for which the Laplace-Beltrami operator takes the form
\begin{align*}
    \Delta_{\mathbb{S}^2}=\frac{1}{\sin^2\psi}\frac{\partial^2}{\partial\theta^2}+\frac{1}{\sin \psi} \frac{\partial}{\partial \psi}\Big(\sin \psi\frac{\partial}{\partial \psi}\Big).
\end{align*}
Suppose $z_1$ is a vertex of a spherical triangle at the north pole of $\mathbb{S}^2$, and with angle $\alpha_1$. Consider the wedge shape $W=\{(\theta,\psi)\in \mathbb{S}^2:\theta\in (0,\alpha_1)\}$, which has principal Dirichlet eigenfunction $\varphi_W(\theta,\psi)\asymp \sin(\pi\theta/\alpha_1)\cdot (\sin\psi)^{\pi/\alpha_1}$ (in fact this expression is the eigenfunction $\varphi_W$, up to a $L^2(W)$-normalizing constant). Then $\varphi_W$ replaces the role of $h$, allowing us to prove that (\ref{triangleprofile1}) holds for spherical triangles as well. \QED{} 
\\~\\
\textit{Proof sketch of Theorem \ref{polythm} (regular polygons).} In this proof, all implied constants in the symbol $\asymp$ will be absolute. First assume that $P=P(n,l)$ is inscribed inside the unit ball $B(0,1)$, so that $l=2\sin(\pi/n)$. Note that the family of regular polygons is $(C_0,c_0)$-inner uniform with absolute constants $(C_0,c_0)$. Proposition 5.12 of \cite{lierllsc} then implies that there exists an absolute constant $c'\in (0,1)$ such that whenever $c\in (0,c')$, $z\in \partial P$, and $x,y\in \overline{B(z,cl)\cap P}$,
\begin{align}
    \label{4.0}
    \varphi_P(x)\asymp \frac{\varphi_P(y)}{h(y)}h(x),
\end{align}
where $h$ is any nonnegative harmonic function vanishing on $\overline{B(z,c l)}\cap \partial P$. We take $h(r,\theta)=r^{n/(n-2)}\sin(n\theta/(n-2))$ if $z$ is a vertex of $P$ and $h=d(\cdot,l_i)$ if $z$ is a nonvertex point on side $l_i$. 

In (\ref{4.0}), we take $y=y_z\in \overline{B(z,cl)\cap P}$ so that $d(z,y_z)$ is maximal. Then take the number of vertices $n$ large enough so that all such points $\{y_z:z\in \partial P\}$ intersect the interior of the circle $B(0,\cos(\pi/n))$ inscribed inside $P$. By Theorem \ref{ub}, if $z$ is a vertex of $P$, then
\begin{align}
    \label{4.1}
    \varphi_P(y_z)\asymp \varphi_{B(0,1)}(1-cl,0)\asymp \frac{1}{n}.
\end{align}
(To obtain the right hand side of (\ref{4.1}), we used that for any $r>0$,  $\varphi_{B(0,r)}(x)\asymp (r-|x|)/r^2$, which follows from e.g. Theorem \ref{ellipsoid}.) The estimate (\ref{4.1}) also holds if $z$ is a non-vertex boundary point of $P$. On the other hand, we have $h(y_z)=(c\sin(\pi/n))^{\pi/\theta_z}$, where $\theta_z\in \{\pi,\pi(n-2)/n\}$ is the interior angle of $P$ corresponding to $z$. So $h(y_z)\asymp 1/n$ and thus $\varphi_P(y_z)\asymp h(y_z)$. From (\ref{4.0}) we get, for all $x\in B(z,cl)$,
\begin{align}
    \label{4.2}
    \varphi_P(x)\asymp h(x)\asymp \begin{cases}
        d_i(x) & \text{if }z\text{ is not a vertex of }P\text{ and }z\in l_i
        \\ d_i(x)d_{i+1}(x)\Big\{d_i(x)+d_{i+1}(x)\Big\}^{n/(n-2)}& \text{if }z\text{ is the vertex in between }l_i,l_{i+1}.
    \end{cases} 
\end{align}
The indices above are taken modulo $n$. (On the other hand, note that $\varphi_P\asymp 1$ away from the balls $B(z,cl)$ due to Lemma \ref{maxphi} and Lemma \ref{phiharnack}.)
After some rewriting of the expression (\ref{4.2}), we see that the theorem statement holds for all regular polygons $P(n,2\sin(\pi/n))$ inscribed in the unit circle (in which case $nl=2n\sin(\pi/n)\asymp 1$). We then conclude the theorem by scaling properties of Dirichlet Laplacian eigenfunctions. \QED{}

\section{Separating maximum of $\varphi_U$ away from $\partial U$}
\label{separationsection}

In this section, we show that if a bounded domain $U\subseteq X$ satisfies an exterior ball condition (Definition \ref{alphaball}), then a point $x_U\in U$ with $\varphi_U(x_U)$ approximately $\|\varphi_U\|_{L^{\infty}(U)}$ is separated away from the boundary of $U$ by a factor of $\lambda_U^{-1/2}$. In $\mathbb{R}$ and $\mathbb{R}^2$, we weaken the hypothesis on $x_U$ to $\varphi_U(x_U)\gtrsim 1/\sqrt{\mu(U)}$ and prove an analogous result. These results are collected in Theorems \ref{separate2} and \ref{separate3} below. Both theorems cover the case  when $\varphi_U(x_U)=\|\varphi_U\|_{L^{\infty}(U)}$. 

We first recall a result due to Rachh and Steinerberger \cite{rs} that applies when $U\subseteq \mathbb{R}^2$ is a simply connected domain, then we modify their proof strategy to obtain Theorem \ref{separate2} below.

\begin{Theo} 
    \label{rachstein}
    There is an universal constant $c>0$ such that the following holds. Let $U\subseteq \mathbb{R}^2$ be a simply connected domain. Let $\varphi_U$ be the principal Dirichlet Laplacian eigenfunction of $U$. Let $x_U\in U$ be a point with $|\varphi_U(x_U)|=\|\varphi_U\|_{L^{\infty}(U)}$. Then
    $$\text{d}(x_U,\partial U)\geq \frac{c}{\sqrt{\lambda_U}}.$$
    The same result holds (with the same universal constant) if $(\varphi_U,\lambda_U)$ is replaced by $(\varphi^U_k,\lambda_k(U))$. Furthermore, no such result holds with $c$ replaced by $c_d$ in $\mathbb{R}^d$ for $d\geq 3$. 
\end{Theo}

Theorem \ref{rachstein} is only a special case of the main result in \cite{rs}, which was stated for solutions of $-\Delta u=Vu$ where $V$ is a potential on $\mathbb{R}^2$. In Theorem \ref{rachstein}, the lower bound is due to \cite{rs}, while the observation that the same result does not hold in higher dimensions is due to Hayman \cite{hayman}. 

\begin{Dfn}
    \label{alphaball}
    \normalfont
    In a Dirichlet space $(X,d)$, we say that a bounded domain $U\subseteq X$ satisfies an \textit{$\alpha$-exterior ball condition} if there is $\alpha\in (0,1]$ such that for every $x\in U$, there is a ball $B\subseteq B(x,2d(x,\partial U))\backslash U$ with radius $\alpha d(x,\partial U)$.
\end{Dfn}

In Definition \ref{alphaball}, we do not require the exterior ball $B$ to intersect $\partial U$. Consider $\mathbb{R}^n$ equipped with the Euclidean metric and a bounded domain $U\subseteq \mathbb{R}^n$. In this case, Definition \ref{alphaball} can be thought of as a sort of uniform exterior cone condition on $U$. Indeed, suppose for each $x\in U$ and $x'\in\partial U$ with $d(x,x')=d(x,\partial U)$, there is a spherical cone $C\subseteq U^c$ based at $x$ with opening angle $\theta$ and height $d(x,\partial U)$. Then $U$ satisfies an $\alpha$-exterior ball condition (in the sense of Definition \ref{alphaball}) with $\alpha=\alpha(\theta,n)$. In particular, the family of convex domains in $\mathbb{R}^n$ satisfy Definition \ref{alphaball} with $\alpha=\alpha(n)$.

\begin{Theo}
    \label{separate2}
    Let $(X,d,\mu,\mathcal{E},\mathcal{D}(\mathcal{E}))$ be a Dirichlet space satisfying both Assumptions \ref{assumption1} and \ref{assumption2}. There exist constants $C=C(D_0,P_0)>1$ and $\varepsilon=\varepsilon(D_0,P_0,\alpha)\in (0,1)$ such that whenever $U\subseteq X$ is a bounded domain with $\textup{diam}(U)\leq R/C$ and satisfying an $\alpha$-exterior ball condition, 
    $$d(x_U,\partial U)\gtrsim_{D_0,P_0,\alpha} \frac{1}{\sqrt{\lambda_U}},$$
    for any point $x_U\in U$ with $\varphi_U(x_U)\geq (1-\varepsilon)\|\varphi_U\|_{L^{\infty}(U)}$. The diameter $\textup{diam}(U)$ is with respect to $d$. The same result holds (with the same implied constant) if $(\varphi_U,\lambda_U)$ is replaced by $(\varphi_k^U,\lambda_k(U))$ and $x_U$ is any point with $|\varphi_k^U(x_U)|\geq (1-\varepsilon) \|\varphi_k^U\|_{L^{\infty}(U)}$.
\end{Theo}

\begin{proof}
    By replacing $\varphi_k^U$ with $-\varphi_k^U$ if necessary, we can assume $\varphi_k^U(x_U)>0$. Because the proof for $\varphi_k^U$ is the same as $\varphi_1^U=\varphi_U$, we only prove the desired result for $\varphi_U$.  
    
    We first assume $\varphi_U(x_U)=\|\varphi_U\|_{L^{\infty}(U)}$. Recall that $\varphi_U$ is the principal Dirichlet eigenfunction of $-\mathcal{L}^D_U$; let $\{X^x_t:t\geq 0\}$ be the Hunt process associated to $-\mathcal{L}^D_U$ started at the point $x$.  By (\ref{eigenexpansion}) and (\ref{huntexp}), we have for all $t>0$ and $x\in U$
    \begin{align}
        \label{twoways}
        e^{-\lambda_Ut}\varphi_U(x)=P^D_{U,t}\varphi_U(x)=\mathbb{E}^x[\varphi_U(X_t)\mathbf{1}_{\{\tau_U>t\}}].
    \end{align}
    By taking $x=x_U$, we get $1\leq e^{\lambda_U t}\mathbb{P}^{x_U}(\tau_U>t)$, or equivalently
    $$-\frac{\log \mathbb{P}^{x_U}(\tau_U>t)}{\lambda_U}\leq t.$$
    The result will follow if we can show that for some sufficently large $C>0$ to be chosen later, the probability $\mathbb{P}^{x_U}(\tau_U\leq t)$ is bounded away from $0$ for $t=C^2d(x_U,\partial U)^2$. Let $B\subseteq U^c$ be a ball of radius $\alpha d(x_U,\partial U)$ corresponding to $x_U$ in Definition \ref{alphaball}, and let $\sigma_B=\inf\{t>0:X^{x_U}_t\in B\}$.
    Because $B\subseteq U^c$ and $X_t$ has continuous paths, we have
    \begin{align*}
        \mathbb{P}^{x_U}(\tau_U\leq C^2d(x_U,\partial U)^2) & \geq \PP^{x_U}(\sigma_B\leq C^2 d(x_U,\partial U)^2).
    \end{align*}
    The probability of the complement is 
    \begin{align}
        \nonumber
        \mathbb{P}^{x_U}(\sigma_B\geq C^2 d(x_U,\partial U)^2) & = \int_{X\backslash B} p^D_{X\backslash B}(C^2
        d(x_U,\partial U)^2, x_U,y )dy
        \\ \nonumber & \leq \int_{X\backslash B} p_X(C^2d(x_U,\partial U)^2,x_U,y)dy
        \\ \label{s1}& = 1 - \int_B p_X(C^2d(x_U,\partial U)^2, x_U,y)dy
        \\ \label{s2} & \leq 1-\int_{B} \frac{c_1}{\mu(B(x_U,Cd(x_U,\partial U)))}\exp\Big(-\frac{d(x_U,y)^2}{c_2 \cdot C^2 d(x_U,\partial U)^2}\Big)dy.
    \end{align}
    In (\ref{s1}), we used the stochastic completeness of $X$ (see Theorem 2.1 of \cite{hebischlsc}). In (\ref{s2}), we used Theorem \ref{breakthrough} assuming that $C^2d(x_U,\partial U)^2\leq R^2$. Since $d(x_U,y)\leq 2d(x_U,\partial U)$ for all $y\in B$, we get from (\ref{s2}) and Proposition \ref{volumeratioballs} that
    \begin{align}
        \nonumber \mathbb{P}^{x_U}(\sigma_B\geq C^2 d(x_U,\partial U)^2) & \leq 1 - c_1\exp\Big(-\frac{4}{c_2\cdot C^2}\Big)\frac{\mu(B)}{\mu(B(x_U,C d(x_U,\partial U)))}
        \\ \label{s3}& \leq 1 - c_1\exp\Big(-\frac{4}{c_2 \cdot C^2}\Big)\cdot \frac{1}{D_0^2}\Big(\frac{\alpha}{C}\Big)^{\log_2 D_0}.
    \end{align}
    The right-hand side of (\ref{s3}) can be assumed to be strictly less than $1$ if $C>\sqrt{2}$.
    Write $\psi(C)$ to denote the right-hand side of (\ref{s3}). We have that $
    0\leq \psi(C)<1$ for $C\in (\sqrt{2},\infty)$ with $\lim_{C\to\infty}\psi(C)=1$, and that $\psi$ is monotone increasing on $[\sqrt{8/c_2\log_2(D_0)},\infty)$. Thus, by choosing $C$ appropriately so that $C$ depends only on $D_0,P_0$, we see that $$\mathbb{P}^{x_U}(\sigma_B\geq C^2d(x_U,\partial U)^2)\leq \text{const}(D_0,P_0,\alpha)<1.$$ This concludes the argument when $\varphi_U(x_U)=\|\varphi_U\|_{L^{\infty}(U)}$. 

    For the general case when $\varphi_U(x_U)\geq (1-\varepsilon)\|\varphi_U\|_{L^{\infty}(U)}$, let $\varepsilon\in (0,1)$ be arbitrary for now.  Note that if $x_U\in U$ has $\varphi_U(x_U)\geq (1-\varepsilon)\|\varphi_U\|_{L^{\infty}(U)}$, then reasoning as in (\ref{twoways}) we get
    $$t\geq -\frac{1}{\lambda_U}\log \Big(\frac{\mathbb{P}^x(\tau_U>t)}{1-\varepsilon}\Big).$$
    We already proved that for $t=C^2d(x_U,\partial U)$, the probability $\mathbb{P}^x(\tau_U>t)\leq \theta(D_0,P_0,\alpha)<1$ for some $\theta\in (0,1)$ bounded away from $1$. So the result follows once we choose $\varepsilon>0$ sufficiently small so that $\theta/(1-\varepsilon)$ is also bounded away from $1$. 
\end{proof}

Next, when the Dirichlet space is $\mathbb{R}$ or $\mathbb{R}^2$ equipped with a divergence form second-order uniformly elliptic operator as in Example \ref{soue}, we prove a result similar to Theorem \ref{separate2}. In Theorem \ref{separate3} below, we only assume that $x_U\in U$ is a point for which $\varphi_U(x_U)\gtrsim 1/\sqrt{\mu(U)}$, rather than assuming $\varphi(x_U)$ is close to $\|\varphi_U\|_{L^{\infty}(U)}$. 

Let $\{X_t:t\geq 0\}$ be the Hunt process associated to the operator (\ref{soueo}), introduced in Section \ref{huntprocess}. In $\mathbb{R}$ or  $\mathbb{R}^2$, we prove a lemma about $\{X_t:t\geq 0\}$ hitting a small ball of radius $\alpha>0$ with high probability.

\begin{Lemma}
    \label{2dimBM}
    Consider a divergence form second-order uniformly elliptic operator $\mathcal{L}$ on $\mathbb{R}^2$ with ellipticity constant $\Lambda\geq 1$ as in (\ref{soueo}), and let $\{X^x_t:t\geq 0\}$ be the associated Hunt process started at some $x\in \mathbb{R}^2$. Let $\alpha\in (0,1)$. Suppose $B$ is any Euclidean ball of radius $\alpha $ contained inside the annulus 
    $$A_x:=\{y\in \mathbb{R}^2:1<|x-y|<2\}.$$ Then for every $\varepsilon>0$, there is $C=C(\alpha,\Lambda,\varepsilon)>0$ such that 
    \begin{align*}
        \mathbb{P}^x(\sigma_B\geq C)<\varepsilon.
    \end{align*}
    Here $\sigma_B=\inf\{t>0:X^x_t\in B\}$ is the first hitting time of $B$. The same result holds with $\mathbb{R}^2$ replaced by $\mathbb{R}$, with the constant $C$ depending only on $\Lambda$ and $\varepsilon$.
\end{Lemma}

\begin{proof}
    We first prove the result for $\mathbb{R}^2$. Let $\varepsilon>0$ be given. Write $B=B(z,\alpha)$ for some $z\in A_x$, and consider a ball $B(z,S)$ concentric with $B(z,\alpha)$ and with $S\geq 4$. For any $C>0$, we have
    \begin{align}
        \label{bm0}
        \PP^x(\sigma_{B(z,\alpha)}\geq C) & \leq \mathbb{P}^x(\sigma_{B(z,\alpha )}\geq \tau_{ B(z,S)})+\mathbb{P}^x(\tau_{ B(z,S)}\geq C).
    \end{align}
    We claim that for sufficiently large $S=S(\alpha,\Lambda,\varepsilon)$,
    \begin{align}
        \label{bm.1}
        \mathbb{P}^x(\sigma_{B(z,\alpha)}\geq \tau_{ B(z,S)})< \frac{\varepsilon}{2}.
    \end{align}
    
     For simplicity of notation, assume $z=0$ is the origin; the general case is similar. By Lemma 6.1 of \cite{griglsc}, we can choose a nonnegative continuous $\mathcal{L}$-harmonic function $u(x)$ on $\{x\in \mathbb{R}^2:|x|\geq \alpha\}$ such that $u(x)=0$ whenever $|x|=\alpha$, and with $u(x)\asymp_{\Lambda}\log(|x|/\alpha )$. On the other hand, by 
    the unique solvability of the the Dirichlet problem for $\mathcal{L}$ on the annulus (see e.g. \cite{chenzhao}, Theorem 1.1),
    \begin{align}
        \label{bm2} 
        u(x)=\mathbb{E}^{x}[u(X^{x}_{\sigma_{B(0,\alpha)}\wedge \tau_{ B(0,S)}})].
    \end{align}
    From (\ref{bm2}) we get
    \begin{align}
        \label{bm3}
        \log\frac{|x|}{\alpha} \asymp_{\Lambda} \mathbb{E}^{x}[u(X^{x}_{\sigma_{B(0,\alpha)}\wedge \tau_{B(0,S)}})] \asymp_{\Lambda} \log \frac{S}{\alpha}\cdot \PP^x(\sigma_{B(0,\alpha)}\geq \tau_{B(0,S)}).
    \end{align}
    Since $|x|<2$ by hypothesis, (\ref{bm3}) implies that (\ref{bm.1}) holds for sufficiently large $S=S(\alpha,\Lambda,\varepsilon)$. Fix such a value of $S$ and note that $B(z,S)=B(0,S)\subseteq B(x,3S)$. This gives

    \begin{align}
        \label{bm4}
        \mathbb{P}^x(\tau_{B(z,S)}\geq C) \leq \mathbb{P}^x(\tau_{B(x,3S)}\geq C) \leq \frac{\mathbb{E}^x[\tau_{B(x,3S)}]}{C}.
    \end{align}
    By Proposition 8.3 of Chapter VII in \cite{bass}, we have $\mathbb{E}^x[\tau_{B(x,3S)}] \lesssim_{\Lambda} (3S)^2$. Thus, the left hand side of (\ref{bm4}) is at most $\varepsilon/2$ for $C=C(S,\Lambda)=C(\alpha,\Lambda,\varepsilon)$ sufficiently large enough. Combining with (\ref{bm.1}) and (\ref{bm0}) gives the result for the case of $\mathbb{R}^2$. The proof for $\mathbb{R}$ is essentially the same. The only difference is that we choose a nonnegative $\mathcal{L}$-harmonic function $u(x)$ on $(\infty,-\alpha]\cup [\alpha,\infty)$ with $u(x)\asymp_{\Lambda} |x|-\alpha$.
\end{proof}

\begin{Theo}
    \label{separate3}
    Consider a divergence form second-order uniformly elliptic operator $\mathcal{L}$ on $\mathbb{R}^n$ with ellipticity constant $\Lambda\geq 1$ as in (\ref{soueo}). Let $\mu$ denote the Lebesgue measure on $\mathbb{R}^n$.

    \begin{enumerate}
        \item Suppose $n=2$. Let $V\subseteq  \mathbb{R}^2$ be a bounded $(C_0,c_0)$-inner uniform domain with respect to the Euclidean distance and let $\{V_c:c\geq 1\}$ be a family of bounded $(C_0,c_0)$-inner uniform domains all containing $V$, as in Definition \ref{Vc}. Suppose $U$ is a domain satisfying an $\alpha$-exterior ball condition, and with $V\subseteq U\subseteq V_c$. For the principal Dirichlet eigenfunction $\varphi_U$ of $-\mathcal{L}$, suppose $x_U\in U$ is a point with $\varphi_U(x_U)\geq \beta/\sqrt{\mu(U)}$ for some $\beta>0$. Then
        \begin{align}
            \label{R2separation}
            d(x_U,\partial U) \gtrsim\frac{1}{\sqrt{\lambda_U}}. 
        \end{align}
        The implied constants depend only on $c_0,\alpha,\beta,\Lambda$ and an upper bound on $f(c),c\geq 1$. 
        \item Suppose $n=1$. Let $U=(u_1,u_2)\subseteq \mathbb{R}$ be a bounded interval. For any point $x_U\in U$ with $\varphi_U(x_U)\geq \beta/\sqrt{\mu(U)}$ for some $\beta>0$, we have 
        \begin{align}
            \label{R2separation2}
            d(x_U,\partial U)\gtrsim u_2-u_1,
        \end{align}
        where the implied constants depend only on $\beta$ and $\Lambda$.
    \end{enumerate}

\end{Theo}

\begin{Rmk}
    \normalfont
     With essentially the same proof, Theorem \ref{separate3} also holds with $(\varphi_U,\lambda_U)$ replaced by $(\varphi^U_k,\lambda_k(U))$, and with any $x_U\in U$ such that $|\varphi^U_k(x_U)|\geq \beta/\sqrt{\mu(U)}$. In this case, the implied constants in (\ref{R2separation}) and (\ref{R2separation2}) depend also on $k$. 
\end{Rmk}

\begin{proof}
    We first prove (\ref{R2separation}). Setting $x=x_U$ in (\ref{twoways}) and using 
    Lemma \ref{maxphi} with $R=\infty$, we get for all $t>0$
    \begin{align*}
        e^{-\lambda_U t}\frac{\beta}{\sqrt{\mu(U)}}\leq e^{-\lambda_U t}\varphi_U(x_U)\leq \|\varphi_U\|_{\infty}\mathbb{P}^{x_U}(\tau_U>t)\lesssim_{c_0} \sqrt{\frac{f(c)}{\mu(U)}}\mathbb{P}^{x_U}(\tau_U>t).
    \end{align*}
    Writing $c_0'>0$ for a constant depending only on $c_0$, and rearranging,
    \begin{align}
        \label{a0}
        -\frac{1}{\lambda_U}\log \Big\{\frac{c_0'\sqrt{f(c)}}{\beta}\mathbb{P}^{x_U}(\tau_U>t)\Big\}\leq t.
    \end{align}
    Let $B\subseteq \mathbb{R}^2\backslash U$ denote a ball of radius $\alpha\text{dist}(x_U,\partial U)$ from the exterior ball condition. For any $C>0$, 
    \begin{align}
        \label{a1}
        \mathbb{P}^{x_U}(\tau_U>C^2 d(x_U,\partial U)^2) & \leq \mathbb{P}^{x_U}(\sigma_B >C^2 d(x_U,\partial U)^2).
    \end{align}
    By a scaling argument (see for example Proposition 1.1 in Chapter VII of \cite{bass}), we may assume $d(x_U,\partial U)=1$ by replacing the uniformly elliptic matrix $A(x):=(a_{ij}(x))$ by $A(\Phi(x))$, where $\Phi:\mathbb{R}^2\to \mathbb{R}^2$ is dilation by $d(x_U,\partial U)$ with respect to $x_U$. Then by Lemma \ref{2dimBM}, for large enough $C=C(c_0,f,\alpha,\beta,\Lambda)$, we have
    \begin{align*}
        \mathbb{P}^{x_U}(\tau_U>C^2d(x_U,\partial U)^2) <0.99\frac{\beta}{c_0'\sqrt{f(c)}},
    \end{align*}
    which implies by (\ref{a0}) that (\ref{R2separation}) holds. Now suppose $n=1$. The same reasoning leading to (\ref{R2separation}) shows that $d(x_U,\partial U)\gtrsim 1/\sqrt{\lambda_U}$. Then to obtain (\ref{R2separation2}), we only need to observe that the variational formula for Dirichlet eigenvalues implies that $\lambda_U$ is comparable to the principal Dirichlet Laplacian eigenvalue of $U$, which is equal to $\pi ^2/(u_2-u_1)^2$.
\end{proof}

Let $\mathcal{L}$ be a divergence form second-order uniformly elliptic operator on $\mathbb{R}^2$. We now combine Corollary \ref{corollary1} and Theorem \ref{separate3} to give the following very concrete result. Theorem \ref{lowerbound3} improves  Corollary \ref{corollary1} in that the implied constant no longer depends on the parameter $\widetilde{\delta}>0$ in $d(x_U,\partial U)\geq \widetilde{\delta}\text{diam}_U$.

\begin{Theo}
\label{lowerbound3}
    Consider the setting of Example \ref{soue} with $n=2$. Let $V\subseteq \mathbb{R}^2$ be a bounded $(C_0,c_0)$-inner uniform domain and let $\{V_c:c\geq 1\}$ be a family of bounded $(C_0,c_0)$-inner uniform domains all containing $V$, as in Definition \ref{Vc}. Suppose $U\subseteq \mathbb{R}^2$ is a domain satisfying an $\alpha$-exterior ball condition with $V\subseteq U\subseteq V_c$. Suppose $\textup{diam}_U\leq \widetilde{C}\textup{diam}_V$ for some $\widetilde{C}>0$. Suppose there exists a point $x_U\in V$ with $\varphi_U(x_U)\geq \beta/\sqrt{\mu(U)}$ for some $\beta>0$. Then for all $x\in V$,
    $$\varphi_U(x)\gtrsim \varphi_V(x).$$
    The implied constant depends only on $C_0,c_0,\widetilde{C},\alpha,\beta,\Lambda$, and an upper bound on $f(c),c\geq 1$.
\end{Theo}

\begin{proof}
    This follows from combining (\ref{cor1}) of Corollary \ref{corollary1} with Theorem \ref{separate3}.
\end{proof}

\section{Appendix}
\label{append}

\subsection{A caricature function for some $C^{1,1}$-domains in $\mathbb{R}^n$} \label{round}

Throughout this section, we work with Dirichlet Laplacian eigenfunctions. We aim to use our main results, which compare principal Dirichlet eigenfunctions, to obtain a caricature function $\Phi_U$ for certain $C^{1,1}$ domains $U\subseteq \mathbb{R}^n$. Recall that caricature functions for triangles, polygons, and regular polygons were given in Theorems \ref{triangleprofile}, \ref{polygonal}, and \ref{polythm} respectively. The results in this section are intended to provide more examples of caricature functions.

For the rest of this section, we define $\rho_U(x):=\text{dist}(x,\partial U)$. For a fixed bounded $C^{1,1}$ domain $U\subseteq \mathbb{R}^n$, Dirichlet heat kernel estimates from \cite{zhang} imply that for all $x\in U$ sufficiently close to the boundary $\partial U$, for appropriately chosen $t>0$, the Dirichlet heat kernel has $p^D_U(t,x,x)\asymp \rho_U^2(x)$. At the same time, intrinsic ultracontractivity (e.g. \cite{ow} or \cite{lierllsc}) implies $p^D_U(t,x,x)\asymp \varphi_U^2(x)$. We can thus deduce that $\varphi_U\asymp \rho_U$. For sufficiently smooth domains $U$, the observation that $\varphi_U$ decays linearly in $\rho_U$ near the boundary is not new and seems to be well-known in the literature even before \cite{zhang}. For instance, Theorem 7.1 in \cite{daviessimon} proves a lower bound of the form $\varphi_U\gtrsim\rho_U$ for $C^{\infty}$ domains $U$ without appealing to heat kernel estimates. 

The existing results in the literature thus yield caricature functions for a wide class of $C^{1,1}$-domains. For example, in Theorem \ref{ellipsoid} below, we give caricature functions for ellipsoids in Euclidean space, which essentially follows from known results.   

\begin{Theo}
    \label{ellipsoid}
    Let $n\geq 2$ and let $U\subseteq \mathbb{R}^n$ be an ellipsoid of bounded eccentricity with constant $K$ (Definition \ref{boundedecc}). Then if $\varphi_U$ is the principal Dirichlet Laplacian eigenfunction of $U$ normalized so that $\|\varphi_U\|_{L^2(U)}=1$, 
    \begin{align*}
        \varphi_U\asymp \frac{\rho_U}{\textup{diam}(U)^{(n+2)/2}},
    \end{align*}
    where the implied constants depend only on the dimension $n$ and the eccentricity constant $K$. 
\end{Theo}

However, one cannot hope for $\varphi_U\asymp \rho_U$ to hold uniformly over all $C^{1,1}$ domains. Consider, for example, a square $S$ and a triangle $T$ in $\mathbb{R}^2$, as in Figure \ref{roundedfigures} below.
\begin{figure}[H]
  \centering
  \includegraphics[width=0.5\linewidth]{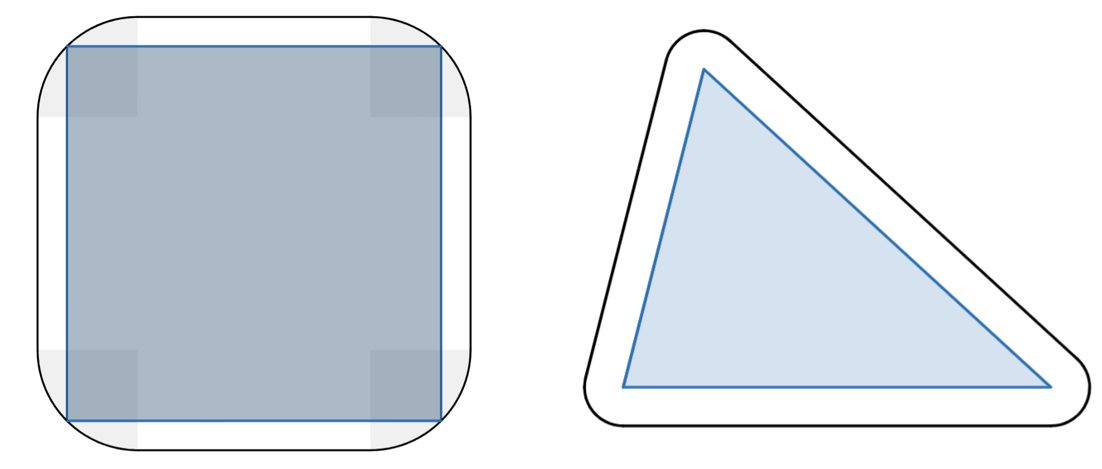}
  \caption{A square and a triangle in $\mathbb{R}^2$ are expanded at scale $\varepsilon>0$.}
  \label{roundedfigures}
\end{figure}
The square $S$ is enlarged by adding four quarter circles of radius $\varepsilon>0$, and the triangle $T$ is enlarged by its $\varepsilon$-neighborhood. Note that these enlarged domains are $C^{1,1}$ domains. Let us focus on the enlarged square $S'\supseteq S$. If $\varepsilon>0$ is bounded away from $0$, then it follows from known results that (modulo some $L^2$-normalizing constant) $\varphi_{S'}(x)\asymp \rho_{S'}(x)$. However, if $\varepsilon>0$ is close to $0$, then we expect a quadratic term to appear in $\varphi_{S'}$, since $S'$ \Quote{converges} to $S$ as $\varepsilon\to 0$ and that $\varphi_S(x)$ decays like $\rho_S^2(x)$ as $x$ approaches a vertex of $S$. Thus for the square, $\varphi_{S'}\asymp \rho_{S'}$ is insufficient to describe $\varphi_{S'}$ when $\varepsilon>0$ is small, and similarly for the triangle.

To then obtain a caricature function uniformly over all $\varepsilon>0$, we use boundary Harnack principle in conjunction with comparison of principal Dirichlet eigenfunctions. The main result of the section is Theorem \ref{roundingthm}, which addresses perturbations of bounded convex domains in $\mathbb{R}^n$. Next, Theorem \ref{roundingthm} is used to obtain  Corollaries \ref{roundedsquare} and \ref{roundedtriangles}, which give caricature functions for both of the enlarged domains in Figure \ref{roundedfigures}.   

In the proofs to follow, we do not directly define $C^{1,1}$ domains, but we instead consider domains satisfying both an interior and exterior ball condition (Definition \ref{epsball} below). Because every domain in $\mathbb{R}^n$ satisfying a two-sided ball condition is $C^{1,1}$ and vice versa (see e.g. \cite{barb}, Theorem 1.0.9), our approach applies to many $C^{1,1}$ domains.

\begin{Dfn}
    \label{epsball}
   We say that a domain $U\subseteq \mathbb{R}^n$ satisfies a two-sided $\varepsilon$-ball condition if, for all $z\in \partial U$, there exists two Euclidean balls $B_1$ and $B_2$ of radius $\varepsilon>0$ such that $B_1\subseteq U$, $B_2\subseteq U^c$, and $z\in (\partial B_1)\cap (\partial B_2)$.
\end{Dfn}

Next, we recall that for any $\varepsilon>0$, the Dirichlet Green function of the ball $B(0,\varepsilon)\subseteq\mathbb{R}^n$ is given by
\begin{align}
    \label{greenfunction}
    G^D_{B(0,\varepsilon)}(x,y) & = \begin{cases}
        \displaystyle -\frac{1}{2\pi}\log\|x-y\|+\frac{1}{2\pi}\log \frac{\|x\|\cdot \|y-x^*\|}{\varepsilon}, & n = 2
        \\ \displaystyle \frac{1}{n(n-2)\omega_n}\Bigg(\frac{1}{\|x-y\|^{n-2}}-\frac{\varepsilon^{n-2}}{\|x\|^{n-2}\cdot \|y-x^*\|^{n-2}}\Bigg), & n\geq 3
    \end{cases}
\end{align}
where $\omega_n$ is the volume of the unit ball in $\mathbb{R}^n$ and $x^*:=\varepsilon^2 x/\|x\|^2$ is inversion with respect to the ball $B(0,\varepsilon)$. The Dirichlet Green function of the complement of the ball, $G^D_{B(0,\varepsilon)^c}(x,y)$, is also given by the same expression (\ref{greenfunction}), but defined for $x,y\in B(0,\varepsilon)^c$. (By abuse of notation, we write $B(0,\varepsilon)^c$ in place of the open set $\mathbb{R}^n\backslash\overline{B(0,\varepsilon)}$.)
For $\varepsilon=1$, the derivation of $G^{D}_{B(0,1)}(x,y)$ can be found in e.g. Chapter 2 of \cite{evans}.

The next lemma states that the Dirichlet Green functions of a ball and the exterior of a ball decay linearly near the boundary. The proof uses the explicit expressions (\ref{greenfunction}) and is omitted. 
\begin{Lemma}
    \label{greenlinear}
Suppose $n\geq 2$ and let $B(0,\varepsilon)\subseteq \mathbb{R}^n$ be an Euclidean ball. For any $\varepsilon>0$, we have 
\begin{align*}
    &  G^D_{B(0,\varepsilon)}(0,y) \asymp \frac{\rho_{B(0,\varepsilon)}(y)}{\varepsilon^{n-1}},\hspace{0.15in}\text{whenever }\rho_{B(0,\varepsilon)}(y)\leq \frac{3\varepsilon}{4}
    \\ & G^{D}_{B(0,\varepsilon)^c}(2\varepsilon\mathbf{e}_n,y)\asymp \frac{\rho_{B(0,\varepsilon)^c(y)}}{\varepsilon^{n-1}},\hspace{0.15in}\text{whenever }\rho_{B(0,\varepsilon)^c}(y)\leq \frac{3\varepsilon}{4},
\end{align*}
where the implied constants depend only on the dimension $n$. Here $\mathbf{e}_n$ denotes the $n$th standard basis vector in $\mathbb{R}^n$.
\end{Lemma}

\begin{Lemma}
    \label{colinear}
    Let $U\subseteq \mathbb{R}^n$ be a domain satisfying a two-sided $\varepsilon$-ball condition. Suppose that given $x\in U$, $z=z_x\in \partial U$ has $\rho_U(x)=\|x-z\|$. Let $B_1\subseteq U$ and $B_2\subseteq U^c$ be two Euclidean balls from Definition \ref{epsball} with $\partial B_1\cap \partial B_2=\{z\}$. Then $x$ must lie on the infinite ray starting at $z$ and passing through the center of $B_1$.
\end{Lemma}

\begin{proof}
    It suffices to prove the lemma when $z=0\in \mathbb{R}^n$ is the origin. Further assume without loss of generality that $B_1=B_1(-\varepsilon\mathbf{e}_n,\varepsilon)$ and $B_2=B_2(\varepsilon\mathbf{e}_n,\varepsilon)$. Since $\rho_U(x)=\|x-z\|=\|x\|$, the ball $B(x,\|x\|)$ must be completely contained in the closure of $U$. In particular, it does not intersect $B_2$. A necessary and sufficient condition for $B(x,\|x\|)\cap B_2(\varepsilon \mathbf{e}_n,\varepsilon)=\varnothing$ is if $\|x-\varepsilon \mathbf{e}_n\|\geq \|x\|+\varepsilon$.

    Elementary arguments give the following inequality: for $x=(x_1,x_2,...,x_n)\in\mathbb{R}^n$, $\|x-\varepsilon \mathbf{e}_n\|\leq \|x\|+\varepsilon$, with equality if and only if $x_1,x_2,...,x_{n-1}=0$ and $x_n\leq 0$. Applying this inequality to the previous observation $\|x-\varepsilon \mathbf{e}_n\|\geq \|x\|+\varepsilon$, we deduce that $x=(0,0,...,0,x_n)$ with $x_n\leq 0$, which is exactly the desired conclusion.
\end{proof}

\begin{Dfn}
    \label{boundedecc}
    A convex domain $V\subseteq \mathbb{R}^n$ has \textit{bounded eccentricity} with constant $K\geq 1$ if for some $x\in \mathbb{R}^n$ and radii $0<a\leq A<\infty$,
    $$B(x,a)\subseteq V\subseteq B(x,A),$$
    with $ a/A\in (K^{-1}, K)$.
\end{Dfn} 

Recall from Example \ref{iuexamples} that  convex sets of bounded eccentricity are inner uniform. The next theorem is the main result of this section. 

\begin{Theo}
    \label{roundingthm}
Let $n\geq 2$ and $\varepsilon>0$. Suppose $V\subseteq \mathbb{R}^n$ is a convex domain having bounded eccentricity with constant $K\geq 1$. Let $U\subseteq \mathbb{R}^n$ be a domain satisfying a two-sided $\varepsilon$-ball condition such that
$$V\subseteq U \subseteq V_{\varepsilon}:=\{x\in \mathbb{R}^n:\textup{dist}(x,V)< 2\varepsilon \}.$$
Then for any $x\in U$,
\begin{align}
    \label{rounding}
    \varphi_U(x)\asymp \frac{\rho_U(x)}{\rho_U(x)+\varepsilon}\Big(\varphi_V(x_{\varepsilon})\wedge\frac{1}{(\textup{diam}(V)+\varepsilon)^{n/2}}\Big)
\end{align}
where $x_{\varepsilon}\in V$ is any point with $\rho_V(x_{\varepsilon})\gtrsim \varepsilon$ and $d_U(x,x_{\varepsilon})\asymp \varepsilon$. The implied constants in (\ref{rounding}) depend only on $n$ and $K$.
\end{Theo}

\begin{Rmk}
    \normalfont
    In Theorem \ref{roundingthm}, choosing $x_{\varepsilon}$ to satisfy $\rho_V(x_{\varepsilon})\gtrsim \varepsilon$ is only possible if $\varepsilon <\textup{diam}(V)$. If $\varepsilon\geq  \textup{diam}(V)$, then Theorem \ref{roundingthm} holds with $\varphi_U(x)\asymp \rho_U/(\rho_U+\varepsilon)\cdot (\textup{diam}(V)+\varepsilon)^{-n/2}$. 
\end{Rmk}

\begin{proof} 
    Recall that $\rho_U(\cdot):=\text{dist}(\cdot,\partial U).$ We consider the cases when $\varepsilon< \text{diam}(V)$ and $\varepsilon\geq \text{diam}(V)$ separately.
    
    \textit{Case 1: Assume $\varepsilon<\textup{diam}(V)$}. Suppose $x\in U$ has $\rho_U(x)\leq \varepsilon/2$ and pick a point $z\in \partial U$ with $\|x-z\|=\rho_U(x)$. Because $U$ satisfies a two-sided $\varepsilon$-ball condition, there exist two $\varepsilon$-balls $B_1(p_1)\subseteq U$ and $B_2(p_2)\subseteq U^c$ corresponding to the boundary point $z$. By Lemma \ref{colinear}, $x$ must lie on the line segment connecting $p_1$ to $z$, and thus for all such points $x$,
    \begin{align}
        \label{rhorhorho}
        \rho_{B_1}(x)=\rho_{B_2}(x)=\rho_U(x)\in [0,\varepsilon/2].
    \end{align}
    We now observe that the hypotheses on $U$ imply that $U$ is an inner uniform domain with constants of inner uniformity depending only on $K$, so we may apply boundary Harnack principle (\cite{lierllsc}, Proposition 5.12) to $U$. If $\varepsilon\in (0,\varepsilon_K)$ is sufficiently small, then for all $y_1,y_2\in B_{\widetilde{U}}(z,\varepsilon/2)$,
    \begin{align}
        \label{bhp}
        \frac{\varphi_U(y_1)}{\varphi_U(y_2)}\asymp \frac{G^D_U(p_1,y_1)}{G^D_U(p_1,y_2)}.
    \end{align}
    In Equation (\ref{bhp}), we take $y_1=x$ and $y_2$ to be the point on the line segment connecting $z\in \partial U$ and $p_1$ such that $\rho_{B_1}(y_2)=\rho_{B_1}(x)+\varepsilon/4$. By the domain monotonicity of Dirichlet Green functions, Lemma \ref{greenlinear}, and (\ref{rhorhorho}), we get
    $$\varphi_U(x)\lesssim \frac{\varphi_U(y_2)}{G_{B_1}^D(p_1,y_2)}G^D_{B_2^c}(p_1,x)\asymp \frac{\varphi_U(y_2)}{\rho_{B_1}(y_2)}\rho_{B_2^c}(x)\asymp \frac{\varphi_U(y_2)}{\rho_U(x)+\varepsilon}\rho_U(x).$$
    By the same argument using domain monotonicity, we get a matching lower bound for $\varphi_U(x)$, and thus
$$\varphi_U(x)\asymp \frac{\varphi_U(y_2)}{\rho_U(x)+\varepsilon}\rho_U(x).$$
The assumptions on $U$ imply that there is a chain of small overlapping Euclidean balls, each with radius $\asymp \varepsilon$ and precompact in $U$, connecting $y_2$ to $x_{\varepsilon}$. Because $y_2$ and $x_{\varepsilon}$ are at a distance comparable to $\varepsilon$ apart, the parabolic Harnack inequality (e.g. (\ref{ballphi})) implies that $\varphi_U(y_2)\asymp\varphi_U(x_{\varepsilon})$ with implied constants independent of $\varepsilon$. Next, comparison of principal Dirichlet eigenfunctions (Theorem \ref{ub}) imply that at the point $x_{\varepsilon}\in V$, we have $\varphi_V(x_{\varepsilon})\lesssim\varphi_U(x_{\varepsilon})\lesssim\varphi_{V'}(x_{\varepsilon})$, where $V'\supseteq V_{\varepsilon}$ is a larger dilate of $V$ with dilation factor to be determined in the next paragraph.

Suppose without loss of generality that $B(0,a)\subseteq V\subseteq B(0,A)$. Define the \Quote{norm} $N(w):=\inf\{t>0:w/t\in V\}$ induced by the convex set $V$, and note that $\gamma V=\{x\in \mathbb{R}^n:N(x)\leq \gamma\}$. If $w\in V_{\varepsilon}$, then there exists $y\in V$ with $\|w-y\|\leq 2\varepsilon$, and consequently
$$N(w)\leq N(w-y)+N(y)\leq \frac{\|w-y\|}{a}+1\leq \frac{2\varepsilon }{a}+1,$$
which implies $V_{\varepsilon}\subseteq (1+2\varepsilon/a)V$. Thus we can take $V'=(1+2\varepsilon/a)V$. Next, we claim that $\varphi_{V'}(x_{\varepsilon})\lesssim \varphi_{V}(x_{\varepsilon})$ with constants independent of $\varepsilon$. By scaling properties of Dirichlet Laplacian eigenfunctions,
$$\varphi_{(1+2\varepsilon/a)V}(x_{\varepsilon})=\frac{1}{(1+2\varepsilon/a)^{n/2}}\varphi_{V}\Big(\frac{x_{\varepsilon}}{1+2\varepsilon/a}\Big).$$
Since $\varepsilon\leq 2A$, the $(1+2\varepsilon/a)^{n/2}$ term is $\asymp 1$. Moreover, since $\varepsilon\leq \text{diam}(V)\leq 2A$, the distance between $x_{\varepsilon}$ and $x_{\varepsilon}/(1+2\varepsilon/a)$ is at most $2A(1-1/(1+2\varepsilon/a))\asymp_K \varepsilon$. Thus the parabolic Harnack inequality implies $\varphi_{V'}(x_{\varepsilon})\lesssim \varphi_{V}(x_{\varepsilon})$, finishing the first case if $\rho_U(x)\leq \varepsilon/2$. If $\rho_U(x)>\varepsilon/2$, then $\varphi_U(x)\asymp \varphi_V(x_{\varepsilon})$ by the same reasoning as above, while $\rho_U(x)/(\rho_U(x)+\varepsilon)\asymp 1$, consistent with the desired result (\ref{rounding}).

\textit{Case 2: Assume $\varepsilon\geq \textup{diam}(V)$.} This case is much better understood: when $\varepsilon$ is large compared to $\text{diam}(V)$, $U$ becomes \Quote{rounder and rounder}; with some routine details omitted, Zhang's heat kernel estimates \cite{zhang} for $C^{1,1}$ domains combined with intrinsic ultracontractivity from Lierl and Saloff-Coste (Theorem \ref{jannalierl}) imply that 
$$\varphi_U\asymp \frac{\rho_U}{\text{diam}(U)^{1+n/2}},$$
and this expression can be rewritten into the form (\ref{rounding}) since $\text{diam}(U)\asymp \text{diam}(V)+\varepsilon$, and $\varepsilon\geq \text{diam}(V)$ implies $\text{diam}(V)+\varepsilon \asymp \rho_U(x)+\varepsilon$.
\end{proof}

We now easily obtain from Theorem \ref{roundingthm} some explicit examples of caricature functions.

\begin{Cor}
    \label{roundedsquare}
    Let $V=[-l,l]^2\subseteq \mathbb{R}^2$ be a square. For $\varepsilon\in (0,l\sqrt{2}]$, consider the larger domain $U\supseteq V$ obtained from adding a quarter circle of radius $\varepsilon>0$ to each vertex of $V$, as in Figure \ref{roundedfigures}. For any $(x,y)\in U$, define $(x_{\varepsilon},y_{\varepsilon})\in V$ by 
    $$x_{\varepsilon}:=\Big(x\wedge(l-\frac{\varepsilon}{\sqrt{2}})\Big)\vee \Big(-(l-\frac{\varepsilon}{\sqrt{2}})\Big),\hspace{0.1in}y_{\varepsilon}:=\Big(y\wedge(l-\frac{\varepsilon}{\sqrt{2}})\Big)\vee \Big(-(l-\frac{\varepsilon}{\sqrt{2}})\Big).$$
    Then for any $(x,y)\in U$, 
    $$\varphi_U(x,y)\asymp \frac{\rho_U(x,y)}{\rho_U(x,y)+\varepsilon}\varphi_V(x_{\varepsilon},y_{\varepsilon}),$$
    where
    $$\varphi_V(x_{\varepsilon},y_{\varepsilon})=\frac{1}{l}\cos\Big(\frac{\pi x_{\varepsilon}}{2l}\Big)\cos\Big(\frac{\pi y_{\varepsilon}}{2l}\Big)\asymp \frac{\min\{x_{\varepsilon}+l,l-x_{\varepsilon}\}\min\{y_{\varepsilon}+l,l-y_{\varepsilon}\}}{l^3}.$$
    All of the implied constants are absolute, in particular independent of $\varepsilon$ and $l$.
\end{Cor}

Note that when $\varepsilon=l\sqrt{2}$, $U$ is an Euclidean ball of radius $l\sqrt{2}$, and as $\varepsilon\to 0$, $U$ shrinks continuously into the square $V$. Corollary \ref{roundedsquare} thus uniformly describes the principal Dirichlet eigenfunction of a domain whose boundary regularity transitions from $C^{\infty}$ to $C^{1,1}$ to Lipschitz.

The next example we present is when $U=\{x\in \mathbb{R}^2:\text{dist}(x,T)<\varepsilon\}$ is the $\varepsilon$-neighborhood of a triangle $T$. Before stating the result, let us specify one particular choice of points $x_{\varepsilon}\in \mathbb{R}^2$ from Theorem \ref{roundingthm} (there are many possible choices). Given a triangle $T=\triangle A_1A_2A_3$ as in Figure \ref{choicefigure} below, draw three circles of radius $\varepsilon$ centered at $A_1,A_2,A_3$. Let $T^{\varepsilon}\subseteq T$ be the triangle formed by connecting the three points where the angle bisectors intersect the circles. (Note that we need $\varepsilon\lesssim \text{diam}(T)$ in order for $T^{\varepsilon}\subseteq T$.) 

\begin{figure}[H]
  \centering
  \includegraphics[width=0.3\linewidth]{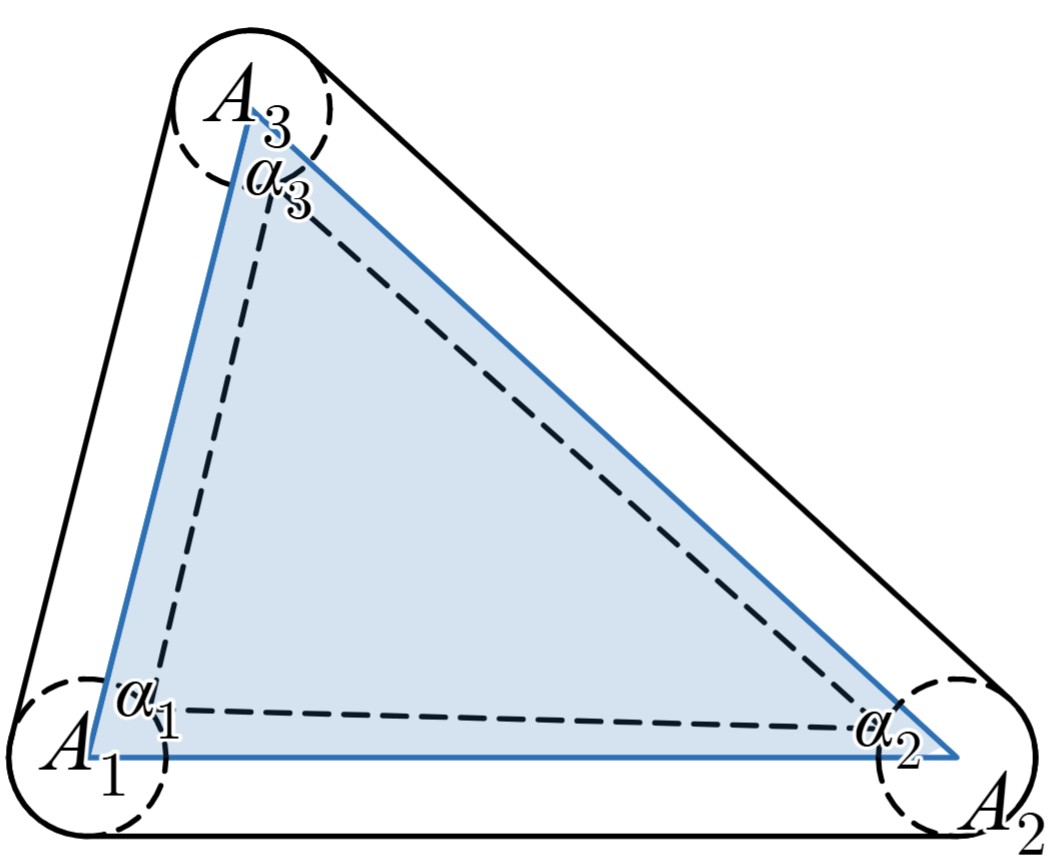}
  \caption{Given $x\in U$, $x_{\varepsilon}$ is the point in the smallest triangle with minimal distance to $x$.}
  \label{choicefigure}
\end{figure}

For any $x\in U$, define $x_{\varepsilon}\in T$ to be the closest point in $T^{\varepsilon}$ to $x$ (this allows for $x=x_{\varepsilon}$ if $x$ is already in $T^{\varepsilon}$). We then get the following.

\begin{Cor} 
    \label{roundedtriangles}
    Let $T=\triangle A_1A_2A_3\subseteq \mathbb{R}^2$ be a triangle with angles $\alpha_1,\alpha_2,\alpha_3\geq \alpha>0$, as in Figure \ref{choicefigure}. There exists a constant $c=c(\alpha)\in (0,1)$ such that the following holds. For all $\varepsilon\in (0,c\cdot \textup{diam}(T))$, if $U=\{x\in \mathbb{R}^2:\textup{dist}(x,T)<\varepsilon\}$ is the $\varepsilon$-neighborhood of $T$, then for all $x\in U$,
    $$\varphi_U(x)\asymp \frac{\rho_U(x)}{\rho_U(x)+\varepsilon}\frac{d_1d_2d_3(d_1+d_2)^{\pi/\alpha_3-2}(d_2+d_3)^{\pi/\alpha_1-2}(d_3+d_1)^{\pi/\alpha_2-2}}{\textup{diam}(T)^{\pi/\alpha_1+\pi/\alpha_2+\pi/\alpha_3-2}}(x_{\varepsilon}),$$
    where $x_{\varepsilon}\in T$ is chosen as above and $d_1(\cdot)=\textup{dist}(\cdot,\overline{A_2A_3})$, $d_2(\cdot)=\textup{dist}(\cdot,\overline{A_1A_3})$, $d_3(\cdot)=\textup{dist}(\cdot,\overline{A_1A_2})$. The implied constants depend only on $\alpha$.
\end{Cor}

\begin{Rmk}
    \normalfont
     As long as $U$ satisfies the hypotheses of Theorem \ref{roundingthm}, the conclusions of Corollaries \ref{roundedsquare} and \ref{roundedtriangles} continue to hold even if some of the straight sides of $U$ are replaced with \Quote{wiggly} $C^{1,1}$ curves. For example, if $U$ is the domain pictured in Figure \ref{wiggle}, then $\varphi_U$ is comparable to the same exact expression in Corollary \ref{roundedsquare}.   
\end{Rmk}

\begin{figure}[H]
  \centering
   \includegraphics[width=0.17\linewidth]{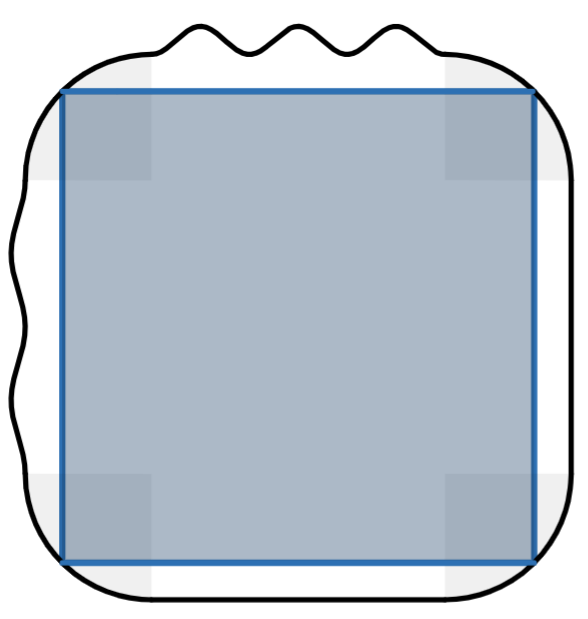}
  \caption{A perturbation of $U$ from Corollary \ref{roundedsquare}.}
  \label{wiggle}
\end{figure}

\subsection{Application of main results to Dirichlet heat kernel estimates}\label{applylierl}

Our main results can be used to give improved Dirichlet heat kernel estimates from \cite{lierllsc} which hold over large families of Euclidean domains. To begin, we state a simplified version of Theorem 7.9 \cite{lierllsc} below.
\begin{Theo}
    \label{lierlspecial}
    Consider Euclidean space $\mathbb{R}^n$ equipped with the Laplacian and Lebesgue measure $d\mu$. For $C_0\geq 1$ and $c_0\in (0,1)$, let $U\subseteq \mathbb{R}^n$ be a bounded $(C_0,c_0)$-inner uniform domain (Definition \ref{iudefn}). Let $\varphi_U>0$ be the principal Dirichlet Laplacian eigenfunction of $U$ normalized so that $\|\varphi_U\|_{L^2(U)}=1$. For $x\in U$ and $t>0$, let $B_U(x,\sqrt{t})\subseteq U$ denote a ball in $U$ with respect to the geodesic distance on $U$, with the convention that the radii of balls are taken to be minimal. Put
    $$V_{\varphi_U^2}(x,\sqrt{t}):=\int_{B_U(x,\sqrt{t})}\varphi_U^2 d\mu.$$
    Then the Dirichlet heat kernel $p^D_U(t,x,y)$ of $U$ satisfies the estimates
    \begin{align}
        \label{lierlspecial1}
        p^D_U(t,x,y)\leq c_1\frac{\varphi_U(x)\varphi_U(y)}{\sqrt{V_{\varphi^2_U}(x,\sqrt{t})}\sqrt{V_{\varphi^2_U}(y,\sqrt{t})}}\exp\Big(-\frac{\|x-y\|^2}{c_2t}\Big),
    \end{align}
    and 
    \begin{align}
        \label{lierlspecial2}
        p^D_U(t,x,y)\geq c_3\frac{\varphi_U(x)\varphi_U(y)}{\sqrt{V_{\varphi^2_U}(x,\sqrt{t})}\sqrt{V_{\varphi^2_U}(y,\sqrt{t})}}\exp\Big(-\frac{\|x-y\|^2}{c_4t}\Big)
    \end{align}
    for all $x,y\in U$ and $t>0$. The constants $c_1,c_2,c_3,c_4$ depend only on the dimension $n$ and the inner uniformity constants $C_0$ and $c_0$ of $U$.
\end{Theo}

Theorem \ref{lierlspecial} thus yields sharp Dirichlet heat kernel estimates for a relatively large class of domains in $\mathbb{R}^n$. For example, the above result applies when $U$ is the Von Koch snowflake in $\mathbb{R}^2$, or when $U$ is any bounded convex domain in $\mathbb{R}^n$ with bounded eccentricity (Definition \ref{boundedecc}). Also, any bounded Lipschitz domain $U\subseteq \mathbb{R}^n$ is $(C_0,c_0)$-inner uniform for some values of $(C_0,c_0)$. However, Theorem \ref{lierlspecial} leaves open the following question:
\begin{align}
    \label{question}
    &  \textit{Can the Dirichlet heat kernel estimates (\ref{lierlspecial1}) and (\ref{lierlspecial2}) be made even more explicit by}
    \\ \nonumber & \textit{finding an explicit function $\Phi_U$ that is comparable to $\varphi_U$?} 
\end{align}
Clearly, in view of Theorem \ref{lierlspecial},
\begin{align}
    \nonumber
    & \textit{Given some family of inner uniform domains $\mathcal{U}$, knowing that $\varphi_U\asymp \Phi_U$ uniformly}
    \\ \nonumber & \textit{for all }U\in \mathcal{U}\textit{ for some explicit function }\Phi_U \textit{ implies essentially explicit Dirichlet heat}
    \\ \label{question2}& \textit{kernel estimates holding uniformly over all }U\in \mathcal{U}.
\end{align}

Of course, it is not feasible to find $\Phi_U$ such that $\Phi_U\asymp \varphi_U$ for every single inner uniform domain $U$. However, because our main results yield explicit caricature functions for large collections of Euclidean domains, for such a collection $\mathcal{U}$, we are able to answer the question (\ref{question}) in the affirmative and achieve (\ref{question2}); we summarize our findings below.

Consider Euclidean space $\mathbb{R}^n$ equipped with the Laplacian. Let $\mathcal{U}$ be any of the below families of (inner uniform) Euclidean domains:
    \begin{enumerate}
        \item the class of planar triangles in $\mathbb{R}^2$ with angles uniformly bounded below by $\alpha>0$ (Theorem \ref{triangleprofile});
        \item a triangle with another small $\varepsilon$-triangle added for $\varepsilon>0$ (Theorem \ref{addonetriangle2} and Figure \ref{trianglesperturbation2});
        \item the collection of all \Quote{generic} polygonal domains in $\mathbb{R}^2$ with a fixed number of sides, all of comparable length, as in Theorem \ref{polygonal};
        \item the collection of all regular polygons in $\mathbb{R}^2$ as in Theorem \ref{polythm};
        \item the collection of all ellipsoids in $\mathbb{R}^n$ with bounded eccentricity as in Theorem \ref{ellipsoid};
        \item the collection of all squares and triangles with rounded $\varepsilon$-corners (Figure \ref{roundedfigures}, Corollaries \ref{roundedsquare} and \ref{roundedtriangles})
    \end{enumerate}
    In each of the above cases, we have essentially explicit Dirichlet heat kernel estimates of the form (\ref{lierlspecial1}) and (\ref{lierlspecial2}), where $\varphi_U$ is made explicit by the caricature function $\Phi_U$ provided in each case, and the constants $c_1,c_2,c_3,c_4$ are as in Theorem \ref{lierlspecial}. 

In the above, only Cases $2,4,$ and $6$ rely on our results about comparison of principal Dirichlet eigenfunctions. The analysis used to obtain Cases $1$ and $3$ uses results from \cite{lierllsc}, while Case $5$ is well-known at least to specialists of the subject. We also emphasize that the above list does not include every single Euclidean domain for which we can answer (\ref{question}) using our results, but just provides selected examples. Because our arguments are specific to the Laplacian in Euclidean space, we do not know if (\ref{question}) can be answered in other (possibly more general) settings. For example, it would be interesting to prove that (\ref{question2}) can be accomplished in $\mathbb{R}^n$ equipped with an uniformly elliptic operator $\mathcal{L}=\sum \partial_i(a_{ij}\partial_j)$ with suitable assumptions on the coefficients $(a_{ij})$.
\\~\\
\textbf{Acknowledgements.} BC would like to thank Stefan Steinerberger for generously answering questions about the paper \cite{rs}, and Zhen-Qing Chen and Leonard Gross for helpful advice. BC is in part supported by NSF Grant DGE–2139899. LSC is supported by NSF Grants DMS-2054593 and DMS-2343868.

\author{
  \noindent 
  Brian Chao
  \\ Department of Mathematics, Cornell University, Ithaca, NY 14853, USA.
  \\ E-mail: \texttt{bc492@cornell.edu}
}
\\~\\
\author{
  \noindent 
  Laurent Saloff-Coste
  \\ Department of Mathematics, Cornell University, Ithaca, NY 14853, USA.
  \\ E-mail: \texttt{lsc@math.cornell.edu}
}

\end{document}